\documentclass[11pt,a4paper]{amsart}
\usepackage{amssymb,amsfonts}
\usepackage{enumerate}
\usepackage{amscd}
\usepackage{makecell}
\usepackage{graphicx}
\usepackage{mathrsfs}
\usepackage{amsmath}
\usepackage{amsfonts}
\usepackage{hyperref}
\usepackage{booktabs}
\usepackage{cleveref}
\usepackage{tikz}
\usetikzlibrary{cd}
\usepackage{xypic}
\usepackage[utf8]{inputenc}
\usepackage[margin=1.20in]{geometry}

\newcommand{\V}{\mathcal{V}}
\newcommand{\D}{\mathcal{D}}
\newcommand{\R}{\mathbb{R}}
\newcommand{\C}{\mathbb{C}}

\newcommand{\im}{\operatorname{im}}
\newcommand{\Q}{\mathbb{Q}}

\newcommand{\Z}{\mathbb{Z}}

\newcommand{\E}{\mathbb{E}}

\newcommand{\RP}{\mathbb{RP}}

\newcommand{\ra}{\rightarrow}
\newcommand{\hy}{\mathbb H}
\newcommand{\inv}{^{-1}}

\newcommand{\lip}{\operatorname{Lip}}
\newcommand{\vol}{\operatorname{Vol}}

\newcommand{\rot}{\operatorname{rot}}

\newtheorem{thm}{Theorem}[section]
\newtheorem{cor}[thm]{Corollary}
\newtheorem{prop}[thm]{Proposition}
\newtheorem{lem}[thm]{Lemma}

\newtheorem{quest}[thm]{Question}
\newtheorem*{claim}{Claim}
\newtheorem*{claimA}{Claim A}
\newtheorem*{claimB}{Claim B}
\newtheorem*{claimC}{Claim C}
\newtheorem*{claimD}{Claim D}

\theoremstyle{definition}
\newtheorem{defn}[thm]{Definition}

\newtheorem{exmp}[thm]{Example}

\newtheorem{rem}[thm]{Remark}

\title[Flexible exponent, Legendrian maps]{Flexible exponent of geometric 3-manifolds,  \\ Legendrian maps of Seifert spaces}

\author{Jianru Duan}
\address{Department of Mathematical Sciences, Peking University, Beijing 100871, CHINA}
\email{duanjr@stu.pku.edu.cn}
\author{Jianfeng Lin}
\address{Department of Mathematics Science, Tsinghua University, Beijing, 100080, CHINA}
\email{linjian5477@mail.tsinghua.edu.cn}
\author{Shicheng Wang}
\address{Department of Mathematical Sciences, Peking University, Beijing 100871, CHINA}
\email{wangsc@math.pku.edu.cn}
\author{Zhongzi Wang}
\address{Department of Mathematical Sciences, Peking University, Beijing 100871 CHINA}
\email{wangzz22@stu.pku.edu.cn}
\author{Dongyi Wei}
\address{Department of Mathematical Sciences, Peking University, Beijing 100871 CHINA}
\email{jnwdyi@pku.edu.cn}

\begin{document}
\maketitle

\begin{abstract} A classical question in quantitative topology is to bound the mapping degree  $\operatorname{deg}(f)$ in terms of its Lipchitz constant $\operatorname{Lip}(f)$. For a closed, oriented manifold $M$, the flexible exponent $\alpha(M)$ is the infimum of  $\alpha\geqslant 0$ such that $|\deg f|\leqslant C(\lip f)^\alpha$ holds for all differentiable map $f:M\ra M$. %we define the flexible exponent \[\alpha(M):=\inf\{\alpha\geqslant 0\mid \exists C \text{ s.t. } |\operatorname{deg}(f)|\leqslant  C\operatorname{Lip}(f)^{\alpha} \text{ for all }f:M\to M\}.\]
The flexible exponent measures how effectively a manifold can wrap itself through self-maps.
  
    For geometric 3-manifolds $M$ in the sense of Thurston,  we give the complete result for $\alpha(M)$:
    \[
      \begin{tabular}{c|c|c|c|c|c}
            \toprule
                  Geometry of $M$  & \makecell{$\mathbb S^3,\ \mathbb E^3,\ \mathbb S^2\times \mathbb E^1$} & Nil  & Sol  & \makecell{$\mathbb H^2\times\mathbb E^1$} & $\mathbb H^3,\ \widetilde {\rm SL_2}$ \\ \midrule
                $\alpha(M)$ & $3$ & $\frac 83$ & $2$  & $1$ & $0$\\
            \bottomrule
    \end{tabular}
    \]
    
%    For non-geometric 3-manifolds, we have $\alpha(M)\leqslant 2$, the case of $\alpha(M)$ is unknown when $M$ connected sums with each prime covered by 
 %   either $S^3$ or $S^2\times \mathbb{E}^1$. 

    To prove $\alpha(M)=8/3$ for Nil 3-manifold $M$, we construct the so-called Legendrian map: a smooth self-map $f: M\to M$ such that $f$ is homotopic to the identity and $f$ maps all $S^1$-fibers into the orthogonal
    contact plane field simultaneously. Moreover, we prove that any Legendrian map must not be a diffeomorphism.
    
    \end{abstract}
\tableofcontents

\section{Introduction}
Let $M,N$ be two closed, connected, oriented manifolds  of the same dimension.
The \textit{degree} of a map $f\colon M\to N$, denoted by $\deg(f)$, is 
probably one of the  oldest concept in topology.
%The {\em set of mapping degrees} from $M$ to $N$ is defined by
%\[
%D(M,N):=\{d\in\Z \ | \ \exists \ f\colon M\to N, \ \deg(f)=d\}.
%\]
% When $M=N$, the {\em set of degrees of self-maps} $D(M,M)$ is denoted by $D(M)$.

 The study of various properties of mapping degrees  for various classes of manifolds $M$ and $N$, 
in particular when $M=N$, has a long history. Influenced  by the work of Thurston and Gromov \cite{Th,Gr},
the topic became very active and make many reactions between  topology, geometry, analysis, volume of representations, and  number theory (see \cite{BG,LS,KL,DLSW,NWW,BGM} and the references therein).  One of such reaction, initialed by Gromov \cite{Gr1},  is to study the constrain between  degrees and Lipchitz constants of maps, see \cite{Gu,BM}, and in particular the recent work of  Berdnikov,  Guth and  Manin \cite{BGM}.

Suppose $M$ and $N$ are closed connected orientable Riemannian $n$-manifolds. Given any map $f:M\ra N$,
 we define the \textit{Lipschitz constant} of $f$ as
\[
    \lip f:=\sup_{u,v\in M} \frac{d_N(f(u),f(v))}{d_M(u,v)}\in [0,+\infty].
\]
We call $f$ an \textit{$L$-Lipschitz mapping} for some positive number $L$ if $\lip f\leqslant L$.

For a closed orientable manifold $M$, we use $D(M)$ to denote the set of degrees of all self-maps of $M$. Following Gromov, a closed orientable manifold $M$ is called \textit{flexible} if and only if $D(M)$ is infinite. Equivalently, $M$ is flexible if and only if $M$ admits a self-map of degree greater than $1$.

%\begin{defn}
%    For any positive number $L$, define a function
%    \[
%        P_{M,N}(L):=\sup \{|\deg f|\mid f:M\ra N \text{ is a differentiable mapping with }\lip f\leqslant L\}.
 %   \]
%    \end{defn}
    %It is clear that $P_{M,N}$ is a non-negative function and is also non-decreasing. 
 %   When $M=N$ are the same Riemannian manifold we will abbreviate $P_{M,M}(L)$ by $P_M(L)$. 

\begin{defn}\label{Definition of the flexible exponents}
    Suppose $M$ is a closed connected orientable Riemannian $n$-manifolds. The \textit{flexible exponent} of $M$ is defined as %follows: let $\mathcal{A}(M)$ be the set of exponents $\alpha\geqslant 0$ such that $P_M(L)=\mathcal{O}_\alpha(L^\alpha)$ when $L\to+\infty$, i.e.
 \[\alpha(M):=\inf\{\alpha\geqslant 0\mid \exists C \text{ s.t. } |\operatorname{deg}(f)|\leqslant C\operatorname{Lip}(f)^{\alpha} \text{ for all }f:M\to M\}.\]    %$\mathcal{A}(M)$ is an interval in the form of $[\alpha_0,+\infty)$ or $(\alpha_0,+\infty)$ for some $\alpha_0\in [0,n]$. Define $\alpha_0$ %, which is the infimum of the set $\mathcal{A}(M)$, 
    %to be the flexible exponent of $M$, and is denoted by $\alpha(M)$. 
    %\[
       % \alpha(M):=\inf \{\alpha\geqslant 0\mid P_{M}(L)\lesssim L^\alpha \}.
    %\]
    \end{defn}
We will soon see that $\alpha(M)$ only depends on the homotopy type of a differentiable manifold $M$ and $\alpha(M)\in[0,\dim M]$. Moreover, $\alpha(M)>0$ if and only if $M$ is flexible (see Corollary \ref{Topological invariance of flexible exponent} and Corollary \ref{flexible if and only if positive exponent}). 
The flexible exponent measures how effectively a self-map can wrap itself. For a simply connected closed manifold $M$ it is known by \cite[Theorem A, Theorem D]{BGM} that $\alpha(M)=\dim M$ if and only if $M$ is formal. 

%Inspired by \cite{BGM},  

We study $\alpha(M)$ for 3-manifolds in this paper. By Thurston's geometrization picture, geometric 3-manifolds are building blocks of general 3-manifolds. Since $\alpha(M)$ connects the geometry and topology of $M$, it is natural to ask about $\alpha(M)$ for geometric $M$. Our main result Theorem \ref{main1} gives a complete answer to this question.

%Within Thurston's geometrization picture, which are confirmed, it is known when $D(M)$ is finite
% for closed orientable 3-manifolds. %Moreover $M$ is simply connected if and only $M$ is the 3-sphere $S^3$.
     
\subsection{Flexibility  of geometric 3-manifolds}

% Thurston's geometrization picture of 3-manifolds 
%  claims first the following stronger version of 
 We start by recalling some classical results on 3-dimensional topology. By Kneser--Milnor's prime decomposition theorem and Papakyriakopoulos's sphere theorem,
each closed orientable 3-manifold $M$ other than $S^3$ has a prime decomposition (unique up to orders and homeomorphisms)
$$M=(\#_{i=1}^mM_i)\#(\#_{j=1}^n N_j)\#(\#^kS^2\times S^1),$$ 
where each $M_i$ is aspherical, and each $N_j$ has finite fundamental group.

There are eight  homogeneous simply connected complete Riemannian 3-manifolds:
$$\mathbb H^3,\ \mathrm{\widetilde{PSL}(2,\R)},\ \mathbb H^2\times \mathbb E^1,\ \mathrm{Sol},\ \mathrm{Nil},\ \mathbb E^3,\ \mathbb S^3,\ \mathbb S^2\times \mathbb E^1.$$
We often use $\mathcal G$ to denote one of those spaces and $\operatorname{Isom}_+\mathcal G$
to denote its group of orientation preserving isometries. Call a compact orientable 3-manifold $M$ \textit{supporting $\mathcal G$ geometry}, 
or a \textit{$\mathcal G$-manifold}, if the interior of $M$ is homeomorphic to $ \mathcal G/\Gamma$ for some discrete, torsion-free subgroup $\Gamma\subset \operatorname{Isom}_+\mathcal G$ of finite
covolume. 
   Thurston's geometrization picture (confirmed by Thurston and Perelman) claims then that the each
Jaco--Shalen--Johanson decomposition piece of a prime 3-manifold
supports one of the eight geometries \cite{Th,Ha,Sc}.
We say that a closed orientable 3-manifold is \textit{geometric}, in the sense of Thurston,  if it supports one of above eight geometries.

In Thurston's geometrization picture, 
 it is verified that (see \cite[Corollary 4.3]{Wang} for example),  a closed orientable 3-manifold $M$ is flexible if and only if $M$ belongs to one of the following classes:
\begin{itemize}
    \item [(i)] $M$ is covered by a torus bundle over the circle, or

    \item [(ii)] $M$ is covered by $\Sigma\times S^1$ for some closed orientable surface $\Sigma$ with genus $>1$, or

    \item [(iii)] each prime factor of $M$ is covered by $S^3$ or $S^2\times S^1$.
\end{itemize}

 Note that $M$ belongs to class (i) if and only if $M$ admits either $\mathbb E^3$, Sol or Nil geometries; $M$ belongs to class (ii) if and only if $M$ admits $\mathbb H^2\times \mathbb E^1$ geometry. Since $\RP^3\#\RP^3$ is the only non-prime 3-manifold which is geometric, indeed admits $\mathbb S^2\times\mathbb E^1$-geometry.  So we have the following criterion:
    A closed orientable 3-manifold $M$ is flexible if and only if $M$ satisfies one of the following two statements:
    \begin{itemize}
        \item $M$ is geometric and supports either $\mathbb S^3,\ \mathbb E^3,\ \mathbb S^2\times \mathbb E^1,\ \mathbb H^2\times \mathbb E^1$, Nil or Sol geometry, or
        \item $M$ is non-geometric, and each prime factor of $M$ admits either $\mathbb S^3$ or $\mathbb S^2\times \mathbb E^1$ geometry.
    \end{itemize}

In this paper we  determine  the  flexible exponents of geometric 3-manifolds.  Below we abbreviate $\widetilde{\rm PSL}(2,\R)$ by $\widetilde{\rm SL_2}$.

\begin{thm}\label{main1}
    Suppose $M$ is a closed connected orientable geometric $3$-manifold. Then
  
    \begin{itemize}
        \item $\alpha(M)=3$ if $M$ supports either  $\mathbb S^3$, $\mathbb E^3$ or  $\mathbb S^2\times \mathbb E^1$ geometries; \item $\alpha(M)= 8/3$ if $M$ supports the {\rm Nil} geometry; 
        \item $\alpha(M)=2$ if $M$ supports the {\rm Sol} geometry;
                  \item $\alpha(M)=1$ if $M$ supports the  $\mathbb H^2\times \mathbb E^1$ geometry;
         \item $\alpha(M)=0$  if $M$ supports either  $\mathbb H^3$ or $\widetilde{\rm SL_2}$ geometries.
       \end{itemize} 
        \end{thm}
        
       \begin{rem} The paper \cite{LSW} computes the value $\alpha(M)$ when the 3-manifold $M$ is non-geometric. 
       \end{rem}
       %For non-geometric case, we have the following partial answer.
       
      % \begin{prop}\label{non-geometric} Suppose $M$ is a closed orientable non-geometric 3-manifold. Then
      % \begin{itemize}
       
   %     \item $0< \alpha(M)\leqslant 2$, if $M$ is connected sums and  each prime factor is covered by 
  %  $S^3$ or $S^2\times \mathbb{E}^1$. 
 %   \item $1\leqslant \alpha(M)\leqslant 2$   if $M$ is connected sums and  each prime factor is covered $S^2\times \mathbb{E}^1$. 
 %        \item $\alpha(M)=0$ in the remaining cases.
 %       \end{itemize}    
 %        \end{prop} 

Motivated by Theorem \ref{main1}, one may ask the following questions.
\begin{quest}\label{quest} Suppose $M$ and $N$ are closed orientable smooth manifolds.
\begin{enumerate}[\quad \rm (1)]
    \item Does $\alpha(M\times N)=\alpha(M)+ \alpha(N)$?
    \item  If $M$ and $N$ have a common finite cover, does $\alpha(M)= \alpha(N)$?
    \item If $M$ and $N$ supports the same homogeneous geometry, does $\alpha(M)= \alpha(N)$?
    \item If there exists a map $f:M\to N$ of non-zero degree, does $\alpha(M)\leqslant \alpha(N)$?
\end{enumerate}
\end{quest}

For geometric 3-manifolds, all questions above have positive answers by Theorem \ref{main1} and other  known facts in topology, geometry,  and non-zero degree maps of 3-manifolds. The answer appears to be unknown in higher dimensions.

\subsection{Legendrian maps on Seifert manifolds}
It is often expected that if a homotopy class of maps contains a covering, the best choice of a representative in this class should be a covering. In our study of $\alpha(M)$, this is true, as we will see, when $M$ supports  $\mathbb{H}^2\times \E^1$,  $\E^3$, 
  and Sol geometries.
  However, when  $M$ supports the Nil geometry, our map realizing $\alpha(M)=8/3$ is not a covering, but a composition of coverings and the so-called ``Legendrian maps''. We now give the definition of these maps, which is of independent interest. 

Among Thurston's  eight geometries,   six of them are Seifert geometries. %We use $\mathcal{S}$ to denote one of these Seifert geometries.
A closed orientable 3-manifold $M$ equipped with a Seifert geometry is called a \textit{Seifert manifold}. Suppose $M$ is a Seifert manifold. Then in the universal cover ${\mathcal{S}}$ of $M$, there is an canonical 1-dimensional (geodesic) foliation $\tilde {\mathcal V}$, which descends to a 1-dimensional foliation $\mathcal V$ on $M$.  This foliation induces a Seifert fibration on $M$. i.e., a homeomorphism between $M$ and the total space of an $S^1$-bundle over a 2-dimensional orbifold $\mathcal{O}$.
\footnote{There are some  Seifert fibrations that are not induced by any Seifert geometry. (e.g., the Seifert fibration of $S^3$ by torus knots). We will not consider these Seifert fibrations in this paper.} The Seifert geometry supported by $M$ is determined by the Euler characteristic of the orbifold $\chi(\mathcal{O})$, and the Euler number of the circle bundle $e(M)$. The precise relation is given in the following table 
 \[
      \begin{tabular}{c|c|c|c}
      \toprule
                    & $\chi(\mathcal{O})>0$ & $\chi(\mathcal{O})=0$  & $\chi(\mathcal{O})<0$ \\\midrule
           $e(M)=0$    & $\mathbb{S}^2\times\mathbb{E}^1$ & $\mathbb{E}^3$  & $\mathbb{H}^2\times \mathbb{E}^1$ \\\midrule
               $e(M)\neq 0$     & $\mathbb{S}^3$ & $\mathrm{Nil} $  & $\widetilde{\mathrm{SL}_{2}}$
                    \\
            \bottomrule
    \end{tabular}
    \]
We call ${\mathcal{S}}$ a \textit{product geometry} if $e(M)=0$, and  a \textit{non-product geometry} if $e(M)\neq 0$.  

Now we consider the orthogonal complement of $\tilde {\mathcal V}$ in ${\mathcal{S}}$, which is a
plane distribution $\tilde {\mathcal D}$. It descends to a 2-dimensional distribution $\mathcal{D}$ on $M$. When $\mathcal{S}$ is a product geometry, $\mathcal{D}$ is integrable and induces a foliation of $M$ by horizontal surfaces. When $\mathcal{S}$ is a non-product geometry, ${\mathcal D}$ is nowhere integrable and induces a contact structure on $M$.

%Below we call such $\mathcal D$ is the plane field associated to the Seifert fiberation of $M$.
 \begin{defn} \label{defn: legendrian and inverse legendrian} Let $M$ be a Seifert manifold and let $f:M\to M$ be a smooth map. %We make the following definitions:
\begin{enumerate}
    \item Call $f$  a \textit{Legendrian map} if  $f$ is homotopic to the identity and $f_{*}(\V_x)\subset \D_{f(x)}$ for all $x\in M$.
    \item Call $f$ an \textit{inverse Legendrian map} if $f$ is homotopic to the identity and $\V_{f(x)}\subset f_{*}(\D_{x})$ for all $x\in M$.
\end{enumerate}
\end{defn}

\begin{exmp}\label{exmp: ADE}
Let $S^3=S(\mathbb{H})$ be the group of unit quaternions. For any $g\in S^3$, we use $l_{g}:S^3\to S^3$ and $r_{g}: S^3\to S^3$ to denote the left and right multiplication by $g$ respectively. Let $X,Y,Z$ be the left invariant vector fields on $S^3$ whose value at $T_{1}S^3$ equals $i,j,k$ respectively. Then $X$ generates the Hopf fibration on $S^3$ with $\mathcal{V}=\mathrm{Span} X$ and $\mathcal{D}=\mathrm{Span} (Y,Z)$. Let $\xi=\frac{1+j}{\sqrt 2}$ and take $f=r_{\xi}:S^3\to S^3$. Then we have $\xi^{-1}\cdot i\cdot \xi=k$, $\xi^{-1}\cdot j\cdot \xi= j$ and $\xi^{-1}\cdot k\cdot \xi= -i$. This implies $f_{*}(X)=Z,\ f_{*}(Y)=Y,\ f_{*}(Z)=-X$. Hence the diffeomorphism $f$ is both Legendrian and an inverse Legendrian.

Now we explore the symmetry of $f$. Consider the subgroup 
\begin{equation}\label{eq: symmetry group of rotation}
G_0=\{l_{g}\circ r_{h}\mid g\in S^3, h\in \{\pm 1, \pm j\}\}\subset SO(4).
\end{equation}
Then any element of $G_0$ commutes with $f$ and preserves the Hopf fibration (generated by $X$). Therefore, for any finite subgroup $\Gamma\subset G_0$ that acts freely on $S^3$. The Hopf fibration descends to a Seifert fibration on $M=S^3/\Gamma$ and the diffeomorphism $f$ descends to a diffeomorphism $f_{\Gamma}:M\to M$ which is both Legendrian and inverse Legendrian.
\end{exmp}

It turns out that Example \ref{exmp: ADE} only give a small portion of Seifert manifolds that admit Legendrian maps.

\begin{thm}\label{thm: legendrian} A closed orientable Seifert manifold $M$ admits a Legendrian map if and only if $e(M)\neq 0$. In this case, the Legendrian map can be chosen to be arbitrarily $C^0$-close to the identity map.
\end{thm}

As we mentioned, Legendrian maps are essential in our computation of flexible exponent for Nil manifolds. Naturally, one may wonder whether the Legendrian maps provided by Theorem \ref{thm: legendrian} can be diffeomorphisms (just like those given in Example \ref{exmp: ADE}). This motivates the definition of inverse Legendrian maps: if a Legendrian map happens to be a diffeomorphism, then its inverse must be an inverse Legendrian map.

The existence of an inverse Legendrian map is much more restrictive. The following Theorem \ref{Thm: main 4} states that Example \ref{exmp: ADE} actually gives all Seifert manifolds that admit inverse Legendrian maps. 

\begin{thm}\label{Thm: main 4}
 A closed orientable Seifert manifold $M$ admits an inverse Legendrian map if and only if $M$ is isometric to $S^3/\Gamma$ for some finite subgroup $\Gamma$ of the group $G_0$ defined in {\rm (\ref{eq: symmetry group of rotation})}. 
\end{thm}

%\begin{rem}
%Let $M=SU(2)/\Gamma\to SU(2)/\Gamma$, where $\Gamma$ is any finite subgroup of $SU(2)$, acting on $SU(2)$ by the left multiplication. The action of $\Gamma$ commutes with the right multiplication of $S^{1}=\{(\begin{smallmatrix}e^{i\theta}, &0\\
%0,&e^{-i\theta}\end{smallmatrix})\}_{\theta}$ on $SU(2)$. This $S^1$-action makes $M$ a Seifert manifold. In Example \ref{exmp: ADE}, we will use the geodesic flow on $S(TS^2)$ to construct a diffeomorphism $f:M\to M$ that is both Legendrian and inverse Legendrian. 
%\end{rem}

%\begin{rem}
%If a Legendrian map is a diffeomorphism, then its inverse must be an inverse Legendrian map. Therefore, by Theorem \ref{main1}, if $M$ is nonproduct supports Nil or $\widetilde{SL}_{2}$ geometry, then any Legendrian map on $M$ can not be a diffeomorphism. 
%\end{rem}

%\begin{defn} Suppose $M$ is a Seifert manifold which supports non-product geometry, and Seifert fiberation $\mathcal V$ and contact field $\mathcal D$.
%Call a map $f: M\to M$ Legendrian, if $f$ maps the $S^1$-fibers $\mathcal V$  into the contact 
%field $\mathcal D$ simontanousely, and $f$  is homotopic to the identity.
%\end{defn}

 % \begin{thm}\label{main3} For each closed orientable Seifert manifolds $M$ supporting non-product geometry,
%there exists a Legendrian map $f: M\to M$, that is, $f$ is homotopic to the identity, and $f$ map the $S^1$-fiber $\mathcal V$  to the contact 
%field $\mathcal D$ simultanousely.

%Moreover Legendrian map  can not be a diffeomorphism when $M$ supports the $Nil$ or $\widetilde {SL_2}$ geometry.
%\end{thm}
\vskip 0.2truecm
\noindent{\bf Acknowledgement.} We thank Professor L. Guth and  Professor F. Manin for communications. 

J. Lin is partially partially supported by National Key
 R \&D Program of China (2025YFA1017500) and National Natural Science
 Foundation of China.(12271281). S. Wang  is partially supported by NSFC grant No. 12571071,  Z. Wang is partially supported by NSFC grant No. 125B2006.

 % \newpage
\section{The flexible exponent}
 Let $X$ and $Y$ be metric spaces and $f:X\ra Y$ be a continuous map, we define the \textit{Lipschitz constant} of $f$ as
\[
    \lip f:=\sup_{u,v\in M} \frac{d_Y(f(u),f(v))}{d_X(u,v)}\in [0,+\infty].
\]
We call $f$ an \textit{$L$-Lipschitz mapping} for some positive number $L$ if $\lip f\leqslant L<+\infty$.

The proposition below collects basic properties of a Lipschitz map which will be used often.
 
\begin{prop}\label{LipschitzBasicProperties}
    Let $X$, $Y$, $Z$ be metric spaces.
    \begin{enumerate}[\quad\rm(1)]
        \item If $f:X\ra Y$ and $g:Y\ra Z$ are Lipschitz maps, then
        \[
            \lip (g\circ f)\leqslant  \lip f\cdot \lip g.
        \]
        \item If $X$ and $Y$ are Riemannian manifolds and $f:X\ra Y$ is differentiable. For any $p\in X$, denote by $df:TX\ra TY$ the tangent map of $f$ and consider its operator norm
        \[
            \|df\|:=\sup_{p\in X} \|df_p\|,\quad\text{in which\quad } \|df_p\|:=\sup_{v\in T_pX,\ \|v\|=1} \|df_p v\|_{T_{f(p)}Y}.
          \]   Then we have
        \[
            \lip f=\|df\|.
        \]
    \end{enumerate}
\end{prop}
\begin{proof}
    (1) follows directly from the definition.
    
    For (2), since $\lip f\geqslant \frac{d_Y(f(u),f(v))}{d_X(u,v)}$ for all $u,v\in X$. Choose a unit speed geodesic $\gamma:[0,\epsilon)\ra X$ with $\gamma(0)=u$. Let $v$ tend to $u$ along $\gamma$, then we have
    \[
        \lip f\geqslant\|df_u(\dot\gamma(0))\|.
    \]
    Since the unit tangent vector $\dot \gamma(0)\in T_uX$ is arbitrary, this proves that $\lip f\geqslant \|df\|.$ For the other direction, choose any two distinct points $u,v\in X$ and connect them by differentiable paths $\gamma_n$ with $\operatorname{length}(\gamma_n)\ra d_X(u,v)$. The path $f\circ \gamma_n$ connects $f(u)$ and $f(v)$ with length no greater than $ \operatorname{length}(\gamma_n)\cdot\sup_{p\in X}\|df_p\|$. Let $n\ra \infty$, we have
    \[
        d_Y(f(u),f(v))\leqslant d_X(u,v)\cdot \|df\|
    \]
    and hence $\lip f\leqslant\|df\|.$ This finishes the proof.
\end{proof}

 Any map $f:M\ra N$ between closed connected $n$-manifolds induces a map $f_*: H_n(M)\to H_n(N)$ on the top dimensional integer homology group. If $M$ and $N$ are oriented then $f_*([M])=k\cdot [N]$ where $[M]$ and $[N]$ are the fundamental classes and $k$ is an integer. This integer $k$ is called the \textit{degree} of $f$, denoted by $\deg f$.

A basic relationship between the Lipschitz constant and the mapping degree of $f$ is the following
\begin{lem}\label{VolumeEstimates}
    Suppose $M$ and $N$ are closed connected oriented Riemannian $n$-manifolds. Let $f:M\ra N$ be any differentiable map, then
    \[
        |\deg f|\leqslant \frac{\vol M}{\vol N}\cdot(\lip f)^n.
    \]
\end{lem}
\begin{proof}
    Choose volume forms $\omega_M,\omega_N$ for $M,N$, respectively. By Proposition \ref{LipschitzBasicProperties}(2) we have $\|df\|= \lip f$, hence $|f^*\omega_N/\omega_M|\leqslant \|df\|^n= (\lip f)^n$. By volume estimates
    \[
        |\deg f|\cdot \vol N=\bigg|\int_M f^*\omega_N\bigg|\leqslant \int_M(\lip f)^n\omega_M=(\lip f)^n\vol M
    \]
    and this finishes the proof.
\end{proof}

\subsection{Homotopy invariance of flexible exponents}

\begin{defn}
    Suppose $M$ and $N$ are closed connected oriented Riemannian $n$-manifolds. For any positive number $L$, define a function
    \[
        P_{M,N}(L):=\sup \{|\deg f|\mid f:M\ra N \text{ is a differentiable map with }\lip f\leqslant L\}.
    \]
    
    It is clear that $P_{M,N}$ is a non-negative function and is also non-decreasing. The previous Lemma \ref{VolumeEstimates} shows that $P_{M,N}\leqslant (\vol M/\vol N) \cdot L^n$. When $M=N$ are the same Riemannian manifold we will abbreviate $P_{M,M}(L)$ by $P_M(L)$. 
\end{defn}

\begin{lem}\label{mapping degree relation under non-zero degree maps}
    Suppose $M_1,\ M_2$ and $N$ are closed connected orientable Riemannian $n$-manifolds. If there exists a non-zero degree continuous map $f:M_1\ra M_2$, then there exists positive constants $C_1, C_2$ depending on $f$ such that
    \[
         P_{M_2,N}(L)\leqslant C_1\cdot P_{M_1,N}(C_2L),\quad  P_{N,M_1}(L)\leqslant C_1\cdot P_{N,M_2}(C_2L).
    \]
\end{lem}
\begin{proof}
By $C^1$-approximation, we can homotope $f$ to a differentiable map $\bar f:M_1\ra M_2$ with non-zero degree.
    Suppose $h:M_2\ra N$ is any differentiable map. Consider the composition $h\circ \bar f:M_2\ra N$, we have 
    \[
        \lip (h\circ \bar f)\leqslant \lip h\cdot \lip {\bar f},\quad \deg(h\circ \bar f)=\deg h\cdot\deg {\bar f}.
    \]
    Hence $ \deg h\cdot\deg {\bar f}\leqslant P_{M_1,N}(\lip h\cdot \lip {\bar f})$. Since the choice of $h$ is arbitrary, this shows that 
    \[
        P_{M_2,N}(L)\leqslant (\deg {\bar f})^{-1} P_{M_1,N}(L\cdot \lip {\bar f})
    \]
    holds for all $L>0$.
    
    Similarly, suppose $p:N\ra M_1$ is any differentiable map. Consider the map $\bar f\circ p:N\ra M_2$, we have 
    \[
        \lip (\bar f\circ p)\leqslant \lip \bar f\cdot \lip p,\quad \deg(\bar f\circ p)=\deg \bar f\cdot\deg p.
    \]
    Hence $ \deg p\cdot\deg \bar f\leqslant P_{N,M_2}(\lip \bar f\cdot \lip p)$. Since the choice of $p$ is arbitrary, this shows that 
    \[
        P_{N,M_1}(L)\leqslant (\deg \bar f)^{-1} P_{N,M_2}(L\cdot \lip \bar f)
    \]
    holds for all $L>0$. Choose $C_1=(\deg \bar f)^{-1}$ and $C_2=\lip \bar f$ we get the desired inequalities.
\end{proof}

The polynomial growth rate of the function $P_{M,N}(L)$ is an intrinsic property of the topology of $M$ and $N$ and does not depend on the Riemannian metrics.

\begin{cor}\label{Mutual nonzero degree map have the same flexible exponent}
    Let $M_1,M_2$ be closed connected oriented Riemannian $n$-manifolds. Suppose there exists continuous maps $f:M_1\ra M_2$ and $g:M_2\ra M_1$ of non-zero degree. Then $\alpha(M_1)=\alpha(M_2)$.
\end{cor}
\begin{proof}
    It follows from the definition that \[
        \alpha(M)=\inf\{\alpha\geqslant 0\mid P_M(L)\leqslant CL^\alpha \text{ for some $C>0$}\}.
    \]
    Applying Lemma \ref{mapping degree relation under non-zero degree maps} to $f:M_1\ra M_2$  and $N=M_1$,
    the second inequality in Lemma \ref{mapping degree relation under non-zero degree maps} becomes
    
        $$P_{M_1}(L)\leqslant C_1\cdot P_{M_1,M_2}(C_2L)$$
        for some constants $C_i$, $i=1,2$.  Applying Lemma \ref{mapping degree relation under non-zero degree maps} to $g:M_2\ra M_1$  and $N=M_2$,
    the first inequality in Lemma \ref{mapping degree relation under non-zero degree maps} becomes

    $$P_{M_1,M_2}(L) \leqslant C_1'\cdot P_{M_2}(C_2' L)$$
    for some constants $C_i'$, $i=1,2$. This shows that $\alpha(M_1)\leqslant \alpha(M_2)$. By symmetry we also have $\alpha(M_2)\leqslant \alpha(M_1)$. Hence $\alpha(M_1)=\alpha(M_2)$.
\end{proof}

\begin{cor}\label{Topological invariance of flexible exponent}
     The flexible exponent $\alpha(M)$ of a closed connected oriented Riemannian manifold $M$ only depends on the homotopy type of $M$. Moreover, $\alpha(M)\leqslant \dim M$.
\end{cor}
\begin{proof}
    Suppose that $M_1$ and $M_2$ are homotopy equivalent, then the homotopy equivalence between them has degree $\pm 1$ and it follows from Corollary \ref{Mutual nonzero degree map have the same flexible exponent} that $\alpha(M_1)=\alpha(M_2)$. Lemma \ref{VolumeEstimates} implies that $\alpha(M)\leqslant \dim M$.
\end{proof}

%\begin{rem}
%    Following Gromov, a closed orientable manifold $M$ is called flexible if and only if $M$ admits a self-map of degree larger than 1.
%    It is a simple exercise that $\alpha(M)>0$ if and only if $M$ is flexible. The flexible exponent measures how effectively a self-map can wrap itself.
%\end{rem}

\subsection{Bounding flexible exponents}

The following Lemma \ref{prepare} shows that to  provide a lower bound for the flexible exponent $\alpha(M)$ of a manifold $M$,
it suffices to construct a sequence of self-maps $f_n: M\to M$ with controlled degree and Lipchitz constant.

\begin{lem}\label{prepare}
    Suppose $M$ is a closed connected orientable manifold. If there exists an infinite sequence of self-maps $f_1,f_2,\ldots$ of $M$ such that $\lim\limits_{n\ra +\infty}\deg f_n=+\infty$, and
    \[
        \deg f_n>C_1(\lip f_n)^\alpha,\quad \frac{\deg f_{n+1}}{\deg f_{n}}<C_2,\quad \forall n=1,2,\ldots
    \]
    for some constants $C_1,C_2,\alpha>0$, then
    \[
        \alpha(M)\geqslant \alpha.
    \]
\end{lem}
\begin{proof}
    Since $\lim\limits_{n\ra +\infty}\deg f_n=+\infty$, by Lemma \ref{VolumeEstimates} we know that $\lim\limits_{n\ra +\infty}\lip f_n=+\infty$. For any real number $L\geqslant \lip f_1$, we can find an integer $k$ such that $\lip f_k\leqslant L<\lip f_{k+1}$. Then
    \[
        P_M(L)\geqslant P_M(\lip f_k)\geqslant \deg f_k>\frac1{C_2}\deg f_{k+1}>\frac{C_1}{C_2}(\lip f_{k+1})^\alpha>\frac{C_1}{C_2}L^\alpha
    \]
    where the first inequality holds because $P_M$ is non-decreasing; other inequalities hold by assumptions.
    This shows that $\alpha(M)\geqslant \alpha$.
\end{proof}

\begin{cor}\label{flexible if and only if positive exponent}
    The flexible exponent of $M$ is positive if and only if $M$ is flexible.
\end{cor}
\begin{proof}
    If $\alpha(M)>0$, then clearly $M$ admits self-maps of degree greater than $1$. On the other hand, if there exists $f:M\ra M$ with $|\deg f|>1$. Let $f_n:=f^n$. Then we have
    \[
        \deg f_n=(\deg f)^n,\quad \lip f_n\leqslant (\lip f)^n.
    \]
    Set 
    \[\alpha:=\frac {\log \deg f}{\log \lip f}.\]
    Then
    \[
        \deg f_n=(\lip f)^{n\alpha}\geqslant (\lip f_n)^\alpha.
    \]
    and $\lip f_{n+1}\leqslant\lip f\cdot \lip f_n$. It follows from Lemma \ref{prepare} that $\alpha(M)\geqslant \alpha>0$.
\end{proof}

% The following Corollary follows directly from Lemma \ref{prepare}. We state it here for future reference.

% \begin{cor}\label{prepare1} 
%  Suppose $M$ is a closed orientable manifold. If there exists an infinite sequence of self-maps $f_1,f_2,\ldots$ of $M$ such that
%     \[
%        \deg f_n>C\cdot(\lip f_n)^\alpha,\quad   \deg f_n = d_n^l ,\quad \forall n=1,2,\ldots
%     \]
%     for a strictly increasing arithmetic sequence $\{d_n\}$ and constants $C$, $l$, $\alpha>0$, then $\alpha(M)\geqslant \alpha$.
% \end{cor}

% \begin{lem}\label{InvarianceOfExponentUnderLifting}
%     Assume that all non-zero self-maps of $M$ has a lifting to $\tilde M$, then $\alpha(M)\leqslant \alpha(\tilde M)$. 
% \end{lem}
% \begin{proof}
%     Fix a Riemannian metric for $M$ and the Riemannian metric for $\tilde M$ induced by the covering. Then $p:\tilde M\ra M$ becomes a local isometry. For Riemannian manifolds, the Lipschitz constant of a differential mapping equals the operator norm of its tangent mapping. Given any self-map $f:M\ra M$ of non-zero degree, consider the lifting $\tilde f:\tilde M\ra \tilde M$, since $p$ is a local isometry, the tangent maps $d\tilde f:T\tilde M\ra T\tilde M$ and $df:TM\ra TM$ have the same operator norm. We have
%     \[
%         \deg \tilde f=\deg f,\quad \lip\tilde f=\|d\tilde f\|_{T\tilde M}=\|df\|_{TM}=\lip f.
%     \]
%     Since the choice of $f$ is arbitrary, this proves that $P_M(L)\leqslant P_{\tilde M}(L)$ under the chosen Riemannian metric, so $\alpha(M)\leqslant\alpha(\tilde M)$.
% \end{proof}

\begin{lem}\label{LiftingUpperBoundsExponent}
    Let $p:\tilde M\ra M$ be a finite covering between closed orientable manifolds. Suppose that any non-zero degree self-map $f:M\ra M$ can be lifted to $\tilde f :\tilde M\ra \tilde M$, then $\alpha(M)\leqslant\alpha(\tilde M)$.
\end{lem}
\begin{proof}
     Fix a Riemannian metric for $M$ and let $\tilde g$ be the pull back Riemannian metric of $\tilde M$ by the covering. By Lemma \ref{LipschitzBasicProperties} and basic properties of covering maps, for any non-zero degree self-map $f:M\ra M$ and its lift $\tilde f :\tilde M\ra \tilde M$ we have
     \[
        \lip f=\lip \tilde f,\quad \deg f=\deg\tilde f.
     \]
     Then it follows from the definition that $\alpha(M)\leqslant \alpha(\tilde M)$.
\end{proof}

When bounding the flexible exponent from above, the following Proposition \ref{HomologyControlsFlexibleExponent} is very useful and can be viewed as a generalization of Lemma \ref{VolumeEstimates}.

\begin{prop}\label{HomologyControlsFlexibleExponent}
    Suppose $M$ is a closed connected orientable differentiable $n$-manifold. Fix any Riemannian metric for $M$. Then there exists a positive constant $C$ such that for any  non-zero degree differentiable self-map $f:M\ra M$, the following statements hold.
    \begin{enumerate}[\quad \rm(1)]
        \item If the $k$-th betti number $\beta_k:=\dim_\Q H_k(M;\Q)$ is positive for some integer $k$ and let $f_*:H_k(M;\Q)\ra H_k(M;\Q)$ be the induced linear map, then 
        \[
            |\det(f_*)|\leqslant C(\lip f)^{k\cdot \beta_k},
        \]
        where $\det(f_*)$ is the determinant of $f_*$.
        \item Moreover, if $f$ induces an isomorphism on $H_k(M;\Z)/\operatorname{Tor}$, then we have
    \[
        |\deg f|\leqslant C(\lip f)^{n-k}.
    \]
    In particular $\alpha(M)\leqslant n-k$.
    \end{enumerate}
\end{prop}
\begin{proof}
    Fix a fundamental classes $[M]\in H_n(M;\Z)$. The cup product induces a non-degenerate bilinear pairing 
    \[
        x(u,v):=\langle u\smile v,[M]\rangle,\quad u\in H^k(M;\Q),\ v\in H^{n-k}(M;\Q)
    \]
    on the rational cohomology groups. By Poincar\'e duality and the universal coefficient theorem,  the two linear spaces $H^{k}(M;\Q)$ and $H^{n-k}(M;\Q)$ have the same $\Q$-dimension which equals $\beta_k$. Choose a rational basis $\eta_1,\ldots,\eta_{\beta_k}$ for $H^{k}(M;\Q)$ and let $\omega_1,\ldots,\omega_{\beta_k}$ be the dual basis for $H^{n-k}(M;\Q)$, such that
    \[
        x(\eta_i,\omega_j)=\delta_{ij},\quad 1\leqslant i,j\leqslant r.
    \]
By Thom's realization theorem \cite{Thom}, for any $j$, we can find a closed immersed $k$-dimensional submanifold $X_j$ (possibly non-connected) of $M$ and an integer $n_j$, such that $\frac 1{n_j}[X_j]$ is Poincar\'e dual to $\omega_j$. Then pairing with $\omega_j$ can be thought of as integration on these submanifolds: for any cohomology class $u\in H^{n-k}(M;\Q)$, we have
\[
    x(u,\omega_j)=\langle u\smile \omega_j,[M]\rangle=\frac{1}{n_j}\int_{X_j}u.
\]

    Under the basis $\{\eta_i\}$ and $\{\omega_j\}$, 
    let $A_1:H^k(M;\Q)\ra H^k(M;\Q)$ and $A_2:H^{n-k}(M;\Q)\ra H^{n-k}(M;\Q)$ be the matrices of the linear maps induced by $f$. Then there is an upper bound for the entries of $A_1$: 
    \[
     (A_1)_{i,j}=x(f^*\eta_i,\omega_j)=\frac{1}{n_j}\int_{X_j}f^*\eta_i\leqslant(\lip f)^k\cdot\frac{1}{n_j}\int_{X_j}\eta_i.
    \]
    By the universal coefficient theorem, the matrix representing $f_*$ is the transpose of $A_1$ and hence we have
    \[
        |\det(f_*)|=|\det A_1|\leqslant C(\lip f)^{k\cdot \beta_k}
    \]
    for some constant $C$ independent of $f$. This proves the first statement.
    
    For the second statement, the same proof also applies to $A_2$, so that we have $\det A_2\leqslant C(\lip f)^{{(n-k)}\cdot \beta_k}$. If $f_*$ induces an isomorphism on $H_k(M;\Z)/\operatorname{Tor}$, then in particular $\det(A_1)=\pm 1$.
    It is clear that $x(f^*(u),f^*(v))=\pm\deg f\cdot x(u,v)$, this implies that $A_1^t\cdot A_2=\deg f\cdot I$ and
    \[
        \det A_1\cdot \det A_2=\pm(\deg f)^{\beta_k}.
\]
    Combining with the discussion above, we have
    \[
        |\deg f|=|\det A_2|^{1/\beta_k}\leqslant C^{1/\beta_k}(\lip f)^{n-k}
    \]and this finishes the proof.
\end{proof}

\section{Legendrian maps and inverse Legendrian maps}

The main purpose of this section is to prove Theorem \ref{thm: legendrian} and Theorem \ref{Thm: main 4}. Firstly we show that product Seifert manifolds supports no Legendrian maps.

\begin{lem}\label{lem: fintie cover for Legendrian}
Let $M$ be a Seifert manifold and let $q: \tilde{M}\to M$ be a finite covering map. Then any Legendrian map (resp. inverse Legendrian map) $f:M\to M$ can be lifted to a Legendrian map (resp. inverse Legendrian map) $\tilde{f}:\tilde{M}\to \tilde{M}$.    
\end{lem}
\begin{proof}
Let $H: I\times M \to M$ be the homotopy from $\operatorname{Id}_{M}$ to $f$. Then $H\circ (\operatorname{Id}_{I}\times q):I\times\tilde{M}\to M$ is a homotopy from $\operatorname{Id}$ to $f\circ q$. By the homotopy lifting property of covering maps, there exists a homotopy $\tilde{H}: I\times\tilde{M}\to \tilde{M}$ of $\operatorname{Id}_{\tilde{M}}$ that lifts $H\circ (\operatorname{Id}_{I}\times q)$. Then $\tilde{f}:=\tilde{H}(1,-)$ is a Legendrian map that lifts $f$. 

The proof of the inverse Legendrian case is identical.
\end{proof}

\begin{lem}\label{lem: product no legendrian}
Suppose $M$ supports a product geometry. Then $M$ admits no Legendrian maps.
\end{lem}
\begin{proof}
By Lemma \ref{lem: fintie cover for Legendrian}, we may pass to a finite cover and assume $M$ is trivial circle bundle $S^1\times \Sigma$. %Consider the projection map $p: M\to S^1$.  Then $p$ has degree-$1$  when restricted to a Seifert fiber $S^1\times \{*\}$ and $p$ is constant when restricted to a section $\{*\}\times \Sigma$. 
A Seifert fiber $S^1\times \{*\}$ has positive intersection number with the surface section $\{*\}\times \Sigma$.
Therefore, there does not exist a map $f:M\to M$ homotopic to the identity and takes a fiber $S^1\times \{*\}$ to a section $\{*\}\times \Sigma$. In particular, no Legendrian map exists on $M$. 
\end{proof}

\subsection{Constructing Legendrian maps}\label{sec: constructing Legendrian maps}
Suppose $N$ is a closed orientable Seifert manifold  supporting either $\mathbb S^3$, $\mathrm{Nil}$ or $\widetilde{\rm SL_2}$ geometries. Then there is a circle bundle $M$ over an orientable surface $\Sigma$ supporting the same geometry, and a finite group $G$ acting on $M$ freely isometrically, such that $N=M/G$. To prove Theorem \ref{thm: legendrian}, we will  construct a Legendrian map on $M$ that is equivariant under $G$. 

According to the geometry of $M$, the base surface $\Sigma$ supports one of $\mathbb S^2$, $\E^2$ or $\hy^2$ geometry, respectively (therefore admits a canonical conformal structure). The isometric $G$-action on $M$ preserves the Seifert fibration and induces a $G$-isometric action on $\Sigma$.

%A two-dimensional closed orbifold $O$ is called \textit{good} if  there is a finite group $G$ acting on a closed surface $\Sigma$, such that $O=\Sigma/G$. The surface $\Sigma$ admits a complex structure such that for any element $g\in G$, the action of $g$ on $\Sigma$ is either biholomorphic or anti-biholomorphic (or equivalently, angle-preserving). Suppose $N$ is an orientable Seifert fibered space over a good orbifold $O$, then the pull back via the quotient map $\Sigma\ra O$ defines a circle bundle $M$ over $\Sigma$. We know that $M$ is a finite covering of $N$ with $N=M/G$.

%Now let's identify $M$ with $S(E)$, the unit bundle of a complex Hermitian line bundle $E\ra \Sigma$. 

Let $E$ be the associated complex line bundle over $M$. Choose a bundle metric of $E$ such that $M$ is identified with its unit circle bundle. Then $G$ acts on $E$ fiber-preservingly with the following properties:
\begin{itemize}
    \item For any $b\in \Sigma$ and $g\in G$, the action of $g$ on the fiber $E_b\ra E_{g(b)}$ is either a complex isomorphism (if $g$ acts on $\Sigma$ orientation-preservingly), or a conjugated complex isomorphism (if $g$ acts on $\Sigma$ orientation-reversingly).
    \item For any $b\in \Sigma$, denote by $G_b:=\{g\in G\mid g(b)=b\}$ the stabilizer. Then $G_b$ is a finite cyclic group. Furthermore, under a local coordinate chart $\Delta\times \C$ of $E$ where $\Delta$ is the unit disk and $b$ is identified with the origin $0\in \Delta$, the action of $G_b$ on $\Delta\times \C$ is generated by 
    \[
        (z,w)\mapsto (\zeta z,\zeta^rw),\quad \forall (z,w)\in \Delta\times \C,
    \]
    in which $\zeta=e^{2\pi {\rm i}/|G_b|}$ is the root of unity, and $r$ is an integer with $\gcd (r,|G_b|)=1$.
\end{itemize}
%We remark if $g\in G_x$ acts on $\Sigma$ anti-holomorphically, then the action of $g$ on $E_x$ is anti-biholomorphic, hence the action of $g$ on $M$ admits a fixed point, which is not allowed since $G$ acts on $M$ by deck transformations. This shows that 
In particular, every element of $G_b$ is an orientation preserving diffeomorphism on $\Sigma$ (otherwise the action of $G_b$ on the circle fiber over $b$ is not free).

Since $G$ acts on $\Sigma$ via isometries (hence angle-preserving), $G$ also acts on the tangent bundle $T\Sigma$. 

\begin{prop}\label{prop find poly map}
    There exists a constant $n>0$ such that for any $k>0$, there is a $G$-equivariant fiber-preserving map $f:E\ra T\Sigma$ covering the identity map on $\Sigma$, and for any $b\in \Sigma$, the restriction to the fiber  $f|_{E_b}:E_{b}\to T_{b}\Sigma$ is a complex polynomial map 
    \[
    f(w)=\sum^{n}_{i=1}a_{i}(b)w^{k_{i}}
    \] that satisfies 
\begin{equation}\label{eq: polynomial degree}
    \min\{k_{1},\cdots,k_{n}\}\geqslant  k,\quad \max\{|a_{1}(b)|,\cdots,|a_{n}(b)|\}=1.
\end{equation}
\end{prop}

\begin{proof}
\textbf{Step 1:} For any $b\in \Sigma$, we find a $G$-invariant neighborhood of $b$ and define a fiber-preserving map $f_b$ on the bundle over this subset, so that the restriction on any fiber is a polynomial.

    For each $b\in \Sigma$, we fix a small neighborhood $U_b$ of $b$ in $\Sigma$, such that in local coordinates we can identify $U_b$ with the unit disk $\Delta\subset \C$ which brings $b$ to the origin, and there are trivializations $E|_{U_b}\cong \Delta\times \C$ and $T\Sigma|_{U_b}\cong \Delta\times \C$, respectively. The actions of $G_b$ on $E|_{U_b}$ and $T\Sigma|_{U_b}$ are generated by
    \[    \alpha_1:(z,w)\mapsto (\zeta z,\zeta^{r_b}w),\quad \forall (z,w)\in \Delta\times \C\]
    and
    \[
        \alpha_2:(z,w)\mapsto (\zeta z,\zeta w),\quad \forall (z,w)\in \Delta\times \C,
    \]
    respectively, where $\zeta=e^{2\pi {\rm i}/|G_b|}$ is the root of unity, and $r_b$ is an integer co-prime to $|G_b|$. For any positive integer $k_b$ with $r_b k_b\equiv 1\mod |G_b|$, define 
    \[
        f_{U_b}:E|_{U_b}\ra T\Sigma|_{U_b},\quad f_{U_b}(z,w):=(z,w^{k_b}),\quad \forall (z,w)\in E|_U \cong\Delta\times\C.
    \]
    This function is $G_b$-equivariant, since 
    \[  
        f_{U_b}(\alpha_1(z,w))=f_{U_b}(\zeta z,\zeta^{r_b}w)=(\zeta z,\zeta^{k_br_b}w^{k_b})=(\zeta z,\zeta w^{k_b})=\alpha_2(f_{U_b}(z,w)).
    \]
    
    By shrinking $U_b$, we assume that the $G$-orbit of $U_b$ in $\Sigma$ is a finite union of disjoint copies of $U_b$:
    \[
        G\cdot U_b=g_1 U_b\sqcup\cdots \sqcup g_n U_b
    \]
    where $G=g_1G_b\sqcup \cdots\sqcup g_n G_b$ is the coset decomposition. Then define $$f_{g_i\cdot U_b}:=g_i\cdot f_{U_b}\cdot g_i^{-1}:E|_{g_i\cdot U_b}\ra T\Sigma|_{g_i\cdot U_b},\quad i=1,\ldots,n.$$ Putting together, we have extended $f_{U_b}$ to a $G$-equivariant bundle map $$f_b:=\bigsqcup_{i=1}^n f_{g_i\cdot U_b}:E|_{G\cdot U_b} \ra T\Sigma|_{G\cdot U_b}.$$ Then $f_{b}$ is a $G$-invariant bundle map defined on $E|_{G\cdot U_b}$, such that $f_b$ restricts to a polynomial of degree $k_b$ on every fiber. The exact value of $k_b$ will be determined in the next step.

    \textbf{Step 2:} We use the partition of unity to find a globally-defined map $f:E\ra T\Sigma$ as required.

    For any $U_b$ defined as in step 1, find open sets $V_b,V_b'\subset U_b$ with $$b\in V_b\subset \overline{V_b}\subset V_b'\subset \overline{V_b'}\subset  U_b,$$ and find a smooth $G$-invariant function $h_b:\Sigma\ra [0,1]$ such that $h_b\equiv 1$ in $G\cdot \overline{V_b}$, and $h_b\equiv 0$ outside of $G\cdot V_b'$. Actually, identify $U_b$ with the unit disk $\Delta$ as in step 1, then $V_b$ (resp. $V_b'$) can be chosen to be the disk centered at the origin of radius $\frac13$ (resp. $\frac 12$), and define $h_b|_{U_b}$ to be a rotation-invariant bump function supported on $V_b'$, then translates the definition to other copies of $G\cdot U_b$ by the group action. Let $\pi:E\ra \Sigma$ be the bundle projection. Define
    \[
        h_bf_b:E\ra T\Sigma,\quad h_bf_b(v)=\left\{\begin{aligned}
            &h_b(\pi(v))\cdot f_b(v),\quad \pi(v)\in G\cdot U_b,\\
            &\text{origin of $E_{\pi(v)}$},\quad \pi(v)\in (G\cdot V_b')^c.
        \end{aligned} \right.
    \]

    There are finitely many points in $\Sigma$, say $b_1,\ldots,b_n$, such that the union of $V_{b_i}$ covers $\Sigma$. Then the following map
    \[
        f:=h_{b_1}f_{b_1}+\cdots+h_{b_n}f_{b_n}:E\ra T\Sigma
    \]
    is a fiber-preserving map covering the identity, and restricts to a polynomial on each fiber. To make sure the polynomial satisfies condition (\ref{eq: polynomial degree}),  we pick $k_i:=k_{b_i}$ for all $1\leqslant  i\leqslant  n$ to be greater than $k$ and pairwise distinct. This is possible since the congruence equation satisfied by $k_b$ has infinitely many solutions. Clearly the restriction of the map $f$ on each fiber satisfies (\ref{eq: polynomial degree}). The proof is finished.
\end{proof}

%\begin{prop}\label{3} Suppose $\pi : M\to \Sigma$ is a principal $S^1$-bundle 
%and   $\pi: E\to \Sigma$ is a conjugated complex line bundle over $\Sigma$ such that $M=S(E)$, the unit $S^1$-bundle of $E$.
%Let $G$ be a finite group which acts on $E$ as orientation preserving  bundle automorphisms.

%Then there is a $G$-equinvariant bundle homomorphism $\phi_1: E\to T\Sigma$ such that
%the restriction on each fiber is a non-constant polynomial, where $T\Sigma$ is the tangent bundle of $\Sigma$.
%\end{prop}
To proceed, we need to compute the area of the unit disk under a polynomial map.
\begin{lem}\label{lem: area disk}
Consider the map $f:\mathbb{C}\to \mathbb{C}$ given by $f(w)=\sum^{n}_{i=1}a_{i}w^{k_{i}}$. Then \[\mathrm{Area}(f(D^2))=\pi\cdot \sum^{n}_{i=1}|a^2_{i}|k_i.\] Here we define $\mathrm{Area}(f(D^2)):=\int_{D^2}f^*(d\mathrm{vol})$, where $d\mathrm{vol}$ is the standard area form on $\mathbb{C}$.     
\end{lem}
%\begin{proof}%Let $X(\theta)={\rm Re}f(e^{{\rm i}\theta}), Y(\theta)={\rm Im}f(e^{{\rm i}\theta}).$
%Then 
\begin{proof} 
We decompose
$f({\rm e}^{{\rm i}\theta})$ and $a_i$ 
into real and imaginary parts: \begin{align*} f({\rm e}^{{\rm i}\theta})=u(\theta)+{\rm i}v(\theta) 
,\quad a_i=b_i+{\rm i}c_i.\end{align*}
Substituting the above expressions into $f(w)=\sum_{i=1}^na_iw^{k_i}$, we get \begin{align*}u(\theta)=\sum_{i=1}^n (b_i{\rm cos}(k_i\theta)-c_i{\rm sin}(k_i\theta)),\quad v(\theta)=\sum_{i=1}^n (b_i{\rm sin}(k_i\theta)+c_i{\rm cos}(k_i\theta)).\end{align*} Notice that ${\rm Area}f(D^2)$ is the area surrounded by the curve $(u(\theta), v(\theta))$, hence ${\rm Area}(f(D^2))=\int_0^{2\pi}u(\theta)v'(\theta)d\theta$. We have $v'(\theta)=\sum_{i=1}^n (k_ib_i{\rm cos}(k_i\theta)-k_ic_i{\rm sin}(k_i\theta))$. Therefore
\begin{align*}
u(\theta)v'(\theta)&=\sum_{i=1}^n\sum_{j=1}^n (b_i{\rm cos}(k_i\theta)-c_i{\rm sin}(k_j\theta))\cdot(k_jb_j{\rm cos}(k_j\theta)-k_jc_j{\rm sin}(k_i\theta))\\&
=\sum_{i=1}^n 
k_i(b_i^2{\rm cos}^2(k_i\theta)+c_i^2{\rm sin}^2(k_i\theta))+R(\theta),
\end{align*} 
where $R(\theta)$ consists of the cross-terms in ${\rm cos}(k_i\theta)$ and ${\rm sin}(k_i\theta)$. Note that $\int_0^{2\pi}R(\theta)d\theta=0$, the conclusion follows by integrating the above expression.
\end{proof}

\begin{lem}\label{find phi2}  With notations as in Proposition $\ref{prop find poly map}$. For any $\epsilon>0$,
there exists a $G$-equivariant smooth map $\phi_2: E\to \Sigma$ such that for each $b\in \Sigma$ the following properties hold:
\begin{enumerate}[\rm\quad (1)]
\item For any $\theta \in [0,2\pi)$, let $\gamma_{\theta}:[0,1]\mapsto E_{b}$ be the radius connecting $0$ and $e^{i\theta}$. Then the length of $\phi_2(\gamma_{\theta})$ is less than $\epsilon$, 
%where $C$ is the Lipschitz constant of the exponential map ${\rm exp}: D_1(T\Sigma)\to \Sigma$.

%The restriction of $\phi_2$ on $D_1(E_b)$ is $\epsilon C$-Lipschitz.

\item $\mathrm{Area}_\Sigma(\phi_{2}(D_{1}(E_b))):=\int_{D_{1}(E_b)}\phi_{2}^*(d\mathrm{vol}_{\Sigma})=2\pi$. Here $D_1(E_b)$ is the unit disk of the fiber $E_b$ and $d\mathrm{vol}_{\Sigma}$ is the area form on $\Sigma$.
    \item The Jacobian of $\phi_2|_{D_{1}(E_b)}$ is positive except at finitely many points.
\end{enumerate}
\end{lem} 

\begin{proof} For $b\in\Sigma$, let $D_r(T_b\Sigma)$ be the disk of radius $r$ of $T_b\Sigma$ centered at the origin. By shrinking $\epsilon$ we may assume without loss of generality that $\exp:D_\epsilon(T_b\Sigma)\ra \Sigma$ is a $2$-Lipschitz differentiable embedding for all $b\in \Sigma$. 
%and that\begin{equation}\label{eq: area form comparison}\tag{Eq2}\mathrm{exp}^*(d\mathrm{vol}_{\Sigma})|_{D_{\epsilon}(T_{b}\Sigma)}\geqslant   \frac{d\mathrm{vol}_{T\Sigma}}{2}|_{D_{\epsilon}(T_{b}\Sigma)}    \end{equation}

By Proposition \ref{prop find poly map}, we have a $G$-equivariant map $\phi_1: E\to T\Sigma$ such that for all $b\in \Sigma$, 
the restriction $E_b \to T_b \Sigma$ is a polynomial
$\phi_1(w)=\sum^{n}_{i=1}a_{i}(b)w^{k_{i}}$
that satisfies (\ref{eq: polynomial degree}). In particular, we set $k:=16n^2/\epsilon^2$ and we can choose $\phi_1$ such that $k_i\geqslant  k$, $i=1,\ldots,n$. Let $D_1(T_{b}\Sigma)$ be the unit disk of $T_{b}\Sigma$. % and let $d\mathrm{vol}_{T\Sigma}$ be the volume form induced by the inner product on $T_{b}\Sigma$. 
Then by Lemma \ref{lem: area disk} and (\ref{eq: polynomial degree}), we have 
\begin{equation}\label{eq: area estimate}
    \operatorname{Area}_{T\Sigma}(\phi_1(D_1(E_b)))=\int_{D_{1}(E_{b})}\phi_1^*(d\mathrm{vol}_{T\Sigma})=
    2\pi\cdot \sum^{n}_{i=1}|a_{i}(b)|^2  k_{i}\geqslant  2\pi k.
\end{equation}

% Here we use (\ref{eq: polynomial degree}) for the second inequality. This implies 
% \begin{equation}\label{eq: area estimate}\tag{Eq2}
% \operatorname{Area}(\phi_1(D_1(E_b)))\geqslant  \frac{8\pi n^2}{\epsilon^2}.    
% \end{equation}

For any $t>0$, define $\psi_t:=\exp(t\cdot\phi_1)$ to be the composition 
\[
\psi_{t}:E\xrightarrow{\phi_{1}} T\Sigma\xrightarrow{t\cdot -}T\Sigma\xrightarrow{\operatorname{exp}}\Sigma.
\]
% \[
% \psi_{b,t}:D_{1}(E_{b})\xrightarrow{\phi_{1}} D_{n}(T_{b}\Sigma)\xrightarrow{t\cdot -}D_{nt}(T_{b}\Sigma)\xrightarrow{\operatorname{exp}}\Sigma.
% \]
Define the map 
\[
    F:\Sigma\times[0,+\infty)\ra \R,\quad F(b,t):=\operatorname{Area}_\Sigma(\psi_t(D_1(E_b))).
    %=\int_{D_{1}(E_b)}\psi_{b,t}^*(d\mathrm{vol}_{\Sigma}).
\]
We have \[
F\Big(b,\frac\epsilon {2n}\Big)
\geqslant  \frac{1}{4}\operatorname{Area}_{T\Sigma}\Big(\frac\epsilon {2n}\cdot \phi_1(D_{1}(E_{b}))\Big)
%\int_{D_{1}(E_{b})}(t\cdot \phi_{1})^{*}(d\mathrm{vol}_{T\Sigma})
=\frac{\epsilon^2}{16n^2}\operatorname{Area}_{T\Sigma}(\phi_1(D_{1}(E_{b})))\geqslant  2\pi.
%\int_{D_{1}(E_{b})} \phi_{1}^{*}(d\mathrm{vol}_{T\Sigma}).    
\]
Indeed, the first inequality is because $\exp_b$ is $2$-Lipschitz over $D_\epsilon(T_b\Sigma)$, and $$\frac\epsilon {2n}\phi_1(D_1(E_b))\subset D_{\epsilon/2}(T_b\Sigma)$$ by (\ref{eq: polynomial degree}). The last inequality follows from (\ref{eq: area estimate}) and our choice of $k$.

% Since $\exp_b$ is $2$-Lipschitz, we have  
% \[
% F(b,t)%=\int_{D_{1}(E_b)}\psi_{b,t}^*(d\mathrm{vol}_{\Sigma})
% \geqslant  \frac{1}{4}\operatorname{Area}(t\cdot \phi_1(D_{1}(E_{b})))
% %\int_{D_{1}(E_{b})}(t\cdot \phi_{1})^{*}(d\mathrm{vol}_{T\Sigma})
% =\frac{t^2}{4}\operatorname{Area}(\phi_1(D_{1}(E_{b})))
% %\int_{D_{1}(E_{b})} \phi_{1}^{*}(d\mathrm{vol}_{T\Sigma}).    
% \]
% By (\ref{eq: area estimate}), we have 
% \[
% F(b,\frac{\epsilon}{n})\geqslant  \frac{n^2}{4\epsilon^2}\cdot \frac{8\pi n^2}{\epsilon^2}=2\pi.
% \]

Since %for each $b\in \Sigma$, the map  \[\psi_{b,t}=\exp \circ(t\cdot \phi_1): D_{1}(E_b) \to D_{\epsilon}(T_{b}\Sigma)\]
for each $b\in \Sigma$ and $t<\frac \epsilon{2n}$, the region $t\phi_1(D_1(E_b))$ is contained in $D_{\epsilon/2}(T_b\Sigma)$, over which the exponential map is an diffeomorphism. It follows that the Jacobian of $\psi_t$ is positive except at finitely many points (recall that $\phi_1$ is a non-constant polynomial). Therefore we have $\partial F/\partial t>0$. By the Implicit Function Theorem, there exists a unique smooth function $t: \Sigma \to \mathbb (0,\frac{\epsilon}{2n}]$ such that 
$F(b,t(b))=2\pi$. Since $\phi_1$ is $G$-equivariant and $G$ acts on $\Sigma$ via isometry, $t$ must be $G$-invariant.
 
Define
\[
    \phi_2:E\ra \Sigma,\quad \phi_2(v)=\psi_{t(b)}(v)=\exp_b(t(b)\cdot \phi_1 ( v)),\quad \forall v\in E_b,\ b\in\Sigma.
\]
Then $\phi_2$ is $G$-equivariant and $\operatorname{Area}_\Sigma(\phi_2(D_1(E_b)))=2\pi$, as required. 

%By definition we have %the triangle inequality and (\ref{eq: polynomial degree}), we have 
%\begin{align*}
%$\gamma_\theta(t)=\phi_2(t{\rm e}^{i\theta})={\rm exp}_b(t(b)\phi_1(t{\rm e}^{i\theta}))={\rm exp}_b(t_b\sum_{i=1}^{n}a_i{\rm e}^{ik_i\theta}t^{k_i})$
%\end{align*}
%and therefore
%\begin{align*}
%$\gamma_\theta'(t)=d{\rm exp}_b(t_b\sum_{i=1}^nk_ia_i{\rm e}^{{\rm i}k_i\theta}t^{k_i-1}).$
%\end{align*}
Since ${\rm exp}_b$ is $2$-Lipschitz over $t(b)\phi_1(D_1(E_b))$, by the triangle inequality and % we have
%\begin{align*}
%$|\gamma_\theta'(t)|\le C|t_b\sum_{i=1}^nk_ia_i{\rm e}^{{\rm i}k_i\theta}t^{k_i-1}|\le Ct_b\sum_{i=1}^n|a_i|k_it^{k_i-1}.$
%\end{align*}
%By 
(\ref{eq: polynomial degree}) we have \[|\phi_2(\gamma_{\theta})|\leqslant  2 |t(b)\cdot \phi_1(\gamma_{\theta})|=2t(b)\int_{\gamma_{\theta}}|\phi_1'(w)|dw\leqslant  2t(b)\cdot \sum^{n}_{i=1}|a_{i}|\leqslant   2t(b)\cdot n\leqslant  \epsilon.\] Also, on each fiber $E_b$, the Jacobian of $\phi_2|_{D_1(E_b)}$ is positive except at finitely many points, since $\phi_1$ has this property and $\exp|_{D_{\epsilon}(T_{b}\Sigma)}$ has positive Jacobian everywhere.
\end{proof}

As the next step, we re-parametrize the angle coordinate of $\phi_2$ to construct a smooth map $\phi_3:E\ra \Sigma$, so that the area of $\phi_3(D_1(E_b))$ is ``evenly distributed" in each sector.

Identify $S^1$ with $\R/2\pi$ and let $\operatorname{Diff}_+(S^1)$ be the group of orientation-preserving diffeomorphisms of the circle. The  \textit{deviation} $\operatorname{Dev}:\operatorname{Diff}_+(S^1)\ra S^1$ is defined as 
$$\operatorname{Dev}(g):=\frac{1}{2\pi}\int_{0}^{2\pi} (\tilde g(x)-x)dx\mod  {2\pi}$$
where $g\in \operatorname{Diff}_+(S^1)$ and $\tilde g:\R\ra \R$ is any lifting of $g$. For example, the rotation $R_\theta$ by angle $\theta$ has deviation $\theta$.

An element $g\in \operatorname{Diff}_+(S^1)$
is called \textit{balanced} if $\operatorname{Dev}(g)\in 2\pi\Z$. In fact, the deviation is additive with respect to composing with rotations: for any $g\in \operatorname{Diff}_+(S^1)$,
\[
    \operatorname{Dev}(g\circ R_\theta)=\operatorname{Dev}(R_\theta\circ g)=\operatorname{Dev}(g)+\theta.
\]
In particular, for any $g\in\operatorname{Diff}_+(S^1)$ there exists a unique rotation, $R_{-\operatorname{Dev}(g)}$, such that composing with it produces a balanced diffeomorphism.

Let $D\subset \mathbb{C}$ be the unit disc. For $\theta\in [0,2\pi]$, a \textit{sector of angle $\theta$} is a closed subset of $D$ of the form $\{ re^{i\phi}\mid r\in[0,1],\ \phi\in [\theta_0, \theta_0+\theta]\}$, see Figure \ref{fig:sector}.

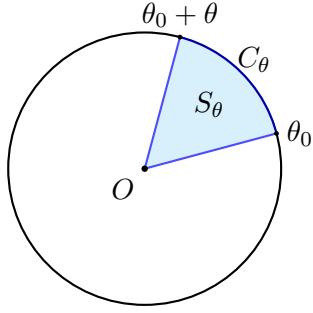
\begin{figure}[htbp]
    \centering
    \begin{tikzpicture}[>=Stealth, scale=.6]
    \draw[thick] (0,0) circle (3cm);
    
    % 绘制扇形（从30°到150°）
    \filldraw[fill=cyan!20, draw=blue, thick, opacity=0.7] 
        (0,0) -- (15:3) arc
        (15:75:3) -- cycle;
    
    % 标记圆心
    \fill (0,0) circle (2pt) node[below left] {$O$};
    
    % 标记半径
    % \draw[->, red, thick] (0,0) -- (30:3) 
    %     node[midway, above right] {$r=3$};
    % \draw[->, red, thick] (0,0) -- (150:3);
    
    % 标记圆弧
    %\draw[thick, blue] (0:3) arc (0:60:3);
    
    % 使用 angles 库标记角度
    \coordinate (A) at (15:3);
    \coordinate (O) at (0,0);
    \coordinate (B) at (75:3);
    
    % \pic [draw, ->, angle radius=1cm, angle eccentricity=1.2, 
    %       "$\alpha=120^\circ$"] {angle = A--O--B};
    
    \node at (45:2) {$S_\theta$};
    \node at (45:3.4) {$C_\theta$};

    % 标记点
    \fill (A) circle (1.5pt) node[right] {$\theta_0$};
    \fill (B) circle (1.5pt) node[above] {$\theta_0+\theta$};
\end{tikzpicture}
\caption{The sector $S_\theta\subset D_{1}(E_b)$}
    \label{fig:sector}
\end{figure}

\begin{lem}\label{find phi3} Suppose $E$, $G$, $\Sigma$ are as in the Proposition $\ref{prop find poly map}$.
Then for any $\epsilon>0$, there is a $G$-equivariant map $\phi_3: E\to \Sigma$ such that
\begin{enumerate}[\rm\quad (1)]
    \item $\phi_{3}(0_{b})=b$, where $0_{b}$ is the origin in $E_{b}$. 
    \item Let $\gamma_\theta$ be any radius in $D_1(E_b)$, then the length of $\phi_3\circ \gamma_\theta$ is less than $\epsilon$. %where $C$ is given in Lemma \ref{find phi2}.
    \item Let $S_{\theta}$ be any  sector of angle $\theta$ in $D_1(E_b)$. Then  $\operatorname{Area}_\Sigma(\phi_3(S_{\theta}))=\theta.$
    \item For any $b\in \Sigma$, the image $\phi_{3}(D_{1}(E_b))$ is contained in the $\epsilon$-neighborhood of $b$.
\end{enumerate}
\end{lem} 

%\begin{lem}\label{6} Let $\phi: D^2\to \Sigma$ be a map such that $Jac>0$ except for finitely many points, and $area(\phi(D^2))=2\pi$.
%Then there is a unique $h: D^2\to D^2$ such that $\hat \phi =\phi \circ h: D^2\to \Sigma$ satisfying
%$$Area (\hat \phi(S_\theta))=\theta .$$
%Moreover $\hat \phi$ continuously depends on $\phi$.
%\end{lem}

%\begin{proof}
%For any $\theta \in [0, 2\pi)$, there is a unique $\theta_1\in [0, 2\pi)$ such that 
%$$\text{Area} (\phi(S_{\theta_1}))=\theta.$$
%Let $h_1(\theta)=\theta_1$. Then $h_1: [0, 2\pi)\to [0, 2\pi)$ is a monotone increasing function.
%Define 
%$$h(\theta)=h_1(\theta_1)-\frac 1{2\pi}\int_0^{2\pi} (h_1(\eta)-\eta) d\eta,$$
%then $h: S^1\to S^1$ is a diffeomorphism. Extends $h$ to a homeomorphism $h: D^2\to D^2$ via $ h(re^{i\theta})=rh(e^{i\theta})$.
%Then one can verify that $h$ is OK. 
%\end{proof}

\begin{proof} Let $\phi_2: E\to \Sigma$ be the map given by Lemma  \ref{find phi2}.
Then $\phi_2$ is $G$-equivariant and $\operatorname{Area}(\phi_2(D_1(E_b)))=2\pi$ for any $b\in \Sigma$. Let $\partial D_1(E_b)$ be the unit circle of the fiber $E_b$, we show that there is a unique \textit{angle reparameterization} $\psi_b\in \operatorname{Diff}_+(\partial D_1(E_b))$, smoothly dependent of $b\in \Sigma$, such that the following conditions hold:
\begin{enumerate}[\rm\quad (i)]
    \item The map $\phi_3:E_b\ra \Sigma$, $\phi_3(r,\theta):=\phi_2(r,\psi_b(\theta))$ sends any sector of angle $\theta$ to a region of area $\theta$ in $\Sigma$, for all $\theta\in[0,2\pi]$. %has the property that $\operatorname{Area}(\phi_3(S_\theta))=\theta$ for any $\theta\in [0,2\pi]$.
    \item $\psi_b$ is balanced.
\end{enumerate}

Note that if $\psi_b$ satisfies condition (i), then so does $\psi_b\circ R_r$ where $R_r\in \operatorname{Diff}_+(\partial D_1(E_b))$ is any rotation.

We first construct $\psi_b$ in local coordinates. Let $U\subset \Sigma$ be a neighborhood of $b$. Identify $E|_U\cong U\times \C$, this amounts to fixing a zero section of the unit circle bundle $\partial D_1(E|_U)$. Hence we can talk about the sector $S_{[0,\theta]}$ consisting of radius whose angle lie in $[0,\theta]$. Consider the area function $F:U\times [0,2\pi]\ra [0,2\pi]$:
\[
    F(b,\theta)= \operatorname{Area}_\Sigma(\phi_2(\text{the sector $S_{[0,\theta]}$ in $D_1(E_b)$})).
\]
Clearly $F(b,0)=0$ and $F(b,2\pi)=2\pi$. Moreover, $\partial F/\partial\theta>0$ by Lemma \ref{find phi2}(3). By Implicit Function Theorem, there is a unique smooth function $(b,\theta)\mapsto f_{b}(\theta)\in[0,2\pi]$ such that  $F(b,f_{b}(\theta))=\theta$. Note that $f_{b}(0)=0$ and $f_{b}(2\pi)=2\pi$, and we can view $f_{b}$ as an element of $\operatorname{Diff}_+(S^1)$. Define $\psi_b:=f_b\circ R_{-\operatorname{Dev}(f_b)}$, then $\psi_b$ is balanced, and is smooth with respect to $b\in U$. The map $\phi_3(r,\theta):=\phi_2(r,\psi_b(\theta))$ sends any sector of angle $\theta$ to a region of area $\theta$ in $\Sigma$.

To prove uniqueness, note that condition (i) determines $\psi_b$ up to a rotation. Namely, if $\psi_b$ and $\psi_b'$ both satisfy condition (i), then $\psi_b'=\psi_b\circ R_r$ for some $r\in \R$. Moreover, there is a unique such choice of $\psi_b$ which is balanced.

The uniqueness implies that if $\psi_b$ is locally defined, then it must be globally well-defined. This proves that there is a unique angle reparameterization $\psi_b$ satisfying (i) and (ii). The uniqueness also implies that $\psi_b$ is $G$-equivariant.

Therefore, $\phi_3(r,\theta):=\phi_2(r,\psi_b(\theta))$ is $G$-equivariant and satisfies conditions (1) and (3). Moreover, $\phi_3$ satisfies (2) since $\phi_2$ does. Finally, (4) follows directly from (2).
\end{proof}

%Suppose now $M$ is a non-trivial $S^1$-bundle, that is the Euler class $e(M)\ne 0$. Then  $M$ supports $Nil$ geometry if $\Sigma$ is a torus, and supports  $\widetilde {SL_2}$ geometry if  $\Sigma$ has genus $>1$. In this case we assume that $M$ supports either Nil-geometry (then  $\Sigma$ supports a matched Euclidean geometry)   or $\widetilde {SL_2}$ geometry (then $\Sigma$ supports a matched hyperbolic geometry). 

Recall that we have a decomposition $TM\cong \mathcal{D}\oplus \mathcal{V}$, where $\mathcal{V}$ is the tangent space of the $S^1$-fiber and $\mathcal{D}$ is its orthogonal complement. In particular, consider $M$ as a principal $S^1$-bundle over $\Sigma$, then $\mathcal D$ is a contact field invariant under both $S^1$ and $G$-actions.

The horizontal distribution $\mathcal{D}$ induces an affine connection on the principal bundle $M$, denoted by $A$. It is known that the curvature form $dA$ is proportional to the area form of $\Sigma$. Therefore we have the following:

%Indeed  for any closed orientable Seifert manifold $M$ supporting Nil or $\widetilde {SL_2}$, $\tilde {\mathcal V}$  in the universal cover descends to  $S^1$ fibers $\mathcal V$ of $M$, and $\tilde {\mathcal D}$ descends to a contact field $\mathcal D$ on $M$ which is orthogonal to $\mathcal V$.
 
 %From now on, each closed orientable Seifert manifold $M$ supporting Nil or $\widetilde {SL_2}$ is associated such   such $S^1$ fibers $\mathcal V$ and a contact field $\mathcal D$ which is orthogonal to $\mathcal V$.
 
%The $S^1$-invariant distribution $\mathcal D$ on $M$ determines an affine connection $A$ . It is known that $K_A=C\omega$, where $C$ is nonzero constant and $\omega$ is the area form of $\Sigma$. Then we can rewrite Lemma \ref{1} as the following 

\begin{lem}\label{lem linear transport} 
Suppose $\gamma: [0,1]\to \Sigma$ is a null-homotopic loop which extends to $\gamma:D\to \Sigma$. Let $\tilde  \gamma : [0,1]\to M$ be a lift of $\gamma$ given by parallel transport along $\mathcal D$. Then 
there exists a constant $C$ independent of $\gamma$ such that
$$\tilde \gamma(1)-\tilde \gamma(0)=\int_{\partial D}\gamma^*A=\int_D \gamma^*(dA)=C \operatorname{Area}_\Sigma (\gamma(D))$$
where we have identified $S^1$ with $\R/2\pi$, and the difference $\tilde \gamma(1)-\tilde \gamma(0)$ is taken in $\R$ and then projected to the $S^1$-fiber over $\gamma(0)=\gamma(1)\in\Sigma$.
\end{lem}
   
\begin{rem}
      By rescaling the Riemannian metric on $\Sigma$ and choosing an appropriate orientation on $\Sigma$, we can take $C=1$.
\end{rem}

Now we are ready to prove Theorem \ref{thm: legendrian}.

%Recall that a curve $\gamma\subset M$ is called \textit{Legendrian} if $\gamma$ is tangent to the contact field $\mathcal D$.  A map $f: M\to M$ is called \textit{Legendrian} if $f$ is homotopic to the identity, and $f$ maps any $S^1$-fiber to a Legendrian curve.

%\begin{thm} \label{7} Suppose $N$ is a closed orientable Seifert manifold supporting either $\widetilde{\rm SL_2}$ or Nil geometry. Then there exists a Legendrian map $f:N\ra N$.

%Then there is a $G$-equinvariant bundle homomorphism $\phi_1: E\to T\Sigma$ such that
%the restriction on each fiber is a non-constant polynomial, where $T\Sigma$ is the tangent bundle of $\Sigma$

%Then there is a  $G$-equivariant map $\phi_4: S(E)\to S(E)$ which is Legendrian, that is, $\phi_4$ is homotopic to the identity, and $\phi_4$ map the $S_1$-fiber $\mathcal V$  to the contact field $\mathcal D$ simultaneously.
%\end{thm}

\begin{proof}[Proof of Theorem $\ref{thm: legendrian}$] The only if part is already proved in Lemma \ref{lem: product no legendrian}. 

Now we prove the if direction.
Let $M$ be the circle bundle over an orientable surface $\Sigma$ which finitely covers $N$, then $M$ supports the same geometry, and there is a finite group $G$ acting via isometry on $M$ such that $N=M/G$. As discussed above, $M$ admits an $S^1$-invariant and $G$-invariant contact field $\mathcal D$. We show that $M$ admits a $G$-equivariant Legendrian map $\phi_4:M\ra M$ which descends to a Legendrian map $f:N\ra N$.

Let $E$ be the complex line bundle over $\Sigma$ associated to $M$ and identify $M$ with the unit circle bundle $\partial D_1(E)$.
By Lemma \ref{find phi3}, we have a $G$-equivariant map $\phi_3: E\to \Sigma$ such that
for each $b\in \Sigma$, we have
$\operatorname{Area} (\phi_3(S_{\theta}))=\theta,$
and $\phi(0_{b})=b$, where $S_\theta$ is any sector of angle $\theta$ in $E_b$.

For any path $\gamma:[0,1]\ra \Sigma$ and any $x\in M$ which projects to $\gamma(0)$, define $\tilde \gamma^x:[0,1]\ra M$ to be the parallel transport of $\gamma$ along $\mathcal D$ which starts at $x$.

For each $x\in M= \partial (D_{1}E)$ which projects to $b\in \Sigma$. Let $r_x: [0,1]\to E_b$ be the radius from the origin of $E_b$ to $x$. Then $\phi_3\circ r_x$ is a path on $\Sigma$.
Define $\phi_4:M\ra M$ as
$$\phi_4(x)=\widetilde {\phi_3\circ r_x}^x (1).$$
The map $\phi_4$ is clearly smooth and $G$-equivariant. We verify that $\phi_4$ is Legendrian. %For any $b\in \Sigma$ and choose any sector $S_\theta\subset D_b$ as in Figure \ref{fig:sector}. Let $\widetilde{C_\theta}$ be the parallel transport of $\phi_3(C_\theta)$ starting at $\phi_4(\theta_0)$, and let $\widetilde{C_\theta}(1)$ be its endpoint. it suffices to show that $\phi_4(C_\theta)$ is Legendrian.

Fix $b\in \Sigma$, choose any sector $S_\theta\subset D_1(E_b)$ shown as in Figure \ref{fig:sector}, then
$$\partial S_\theta= r_{\theta_0} \cup C_\theta \cup r^{-1}_{\theta_0+\theta}.$$
View $\phi_3(\partial S_\theta)$  as a null-homotopic loop in $\Sigma$ with base point $b$,
its parallel transport $\widetilde{\phi_3(\partial S_\theta)}^{\theta_0}$ is obtained in the following way:
first go along $\widetilde {\phi_3(r_{\theta_0})}^{\theta_0}$ to its endpoint $\theta_1\in E_{\phi_3(\theta_0)}$,
then go along $\widetilde {\phi_3(C_\theta)}^{\theta_1}$ to its endpoint $\theta_2\in E_{\phi_3(\theta_0+\theta)}$, and
finally go along $\widetilde {\phi_3(r^{-1}_{\theta_0+\theta})}^{\theta_2}$.
The final endpoint $\widetilde {\phi_3(r^{-1}_{\theta_0+\theta})}^{\theta_2}(1)$ is again in $\partial D_1(E_b)$. By Lemma \ref{lem linear transport},
$$\widetilde {\phi_3(r^{-1}_{\theta_0+\theta})}^{\theta_2}(1)- \theta_0=\operatorname{Area}(\phi_3(S_\theta))=\theta.$$
%The first "=" follows from the definition, the second "=" is based on Lemma \ref{2}, and the third "=" follows from Lemma \ref{find phi3}.
That is to say, $\widetilde {\phi_3(r^{-1}_{x_0+\theta})}^{\theta_2}(1)= \theta_0+\theta$, and equivalently, 
$\phi_4(\theta_0+\theta):=\widetilde {\phi_3(r_{\theta_0+\theta})}^{\theta_0+\theta}(1)= \theta_2$. 
So the parallel transport of $C_\theta$ connects $\phi_4(\theta_0)$ and $\phi_4(\theta_0+\theta)$.
Since $x_0$ and $\theta$ are arbitrary, this proves that $\phi_4(\partial D_1(E_b))$ is a Legendrian curve. 

Finally, define the homotopy $H:M\times I\ra M$ as $$H(x, t):=\widetilde {\phi_3\circ r_x}^x (t), \quad t\in [0,1],$$
then $H(x, 0)=x$ and $H(x,1)=\phi_4(x)$. Therefore, $H$ is a $G$-equivariant homotopy between $\phi_4$ and the identity. This proves that $\phi_4$ descends to a Legendrian map $f:N\ra N$. 
By choosing $\epsilon$ in Lemma \ref{find phi3} sufficiently small, %and arranging the degrees of the monomials carefully, 
we can make $f$ arbitrary closed to the identity map. %The method is that 
%The $C^0$-distance from $f$ to ${\rm id}$ is bounded above by the maximal length of the paths $H(x,t)(t\in [0,1])$ for all $x$. The length of each $H(x,t)$ is equal to the length of $\pi\circ H(x,t)=\phi_3\circ r_x(t)$, since the lifting preserves the length. By our construction, the path $\phi_3\circ r_x: [0,1]\to \Sigma$ is in the form of $\gamma_\theta:[0,1]\to \Sigma$ in Lemma \ref{find phi2}.
%Therefore, we can require the $C^0$-distance from $f$ to ${\rm id}$ bounded above by $C\epsilon$ by Proposition \ref{find phi3}.
% whose length is bounded above by $C\epsilon$.%$\phi_3\circ r_x$ is a modification of $\phi_2\circ r_x$, which corresponds to the $\gamma_\theta$ in Proposition 3.6. We know that the length of $\gamma_\theta$ can be very small. So we want to show the modification from $\phi_2$ to $\phi_3$ is not very big. Rigorously speaking, we need to show the angle reparametrization is $C$-bi-Lipschitz, where $C$ is independent of $\epsilon$. %The degrees of the monomials only relies on the open set in the partition of unity, and is indenpendent of $\epsilon$.
%Indeed, when the degrees of the monomials are "sparse"(the ratio of any two degrees is very large), the behavior of the polynomial is dominated by its leading term. One can verify that in this situation, the angle reparametrization is very slight, and the length of the path $H(x,t)(t\in [0,1])$ is approximately the length of $\gamma_\theta$ in Proposition 3.6.
\end{proof}

% \begin{proof}[Proof of Theorem \ref{thm: legendrian}]
% Since each closed orientable Seifert manifolds supporting $Nil$ (respectively $\widetilde {SL_2}$) geometry has a regular finite cover
% which is a $S^1$-bundle supporting $Nil$ (respectively $\widetilde {SL_2}$) geometry. The first part of Theorem \ref{thm: legendrian}
% follows from Theorem \ref{7}.
%\end{proof}

\subsection{Obstructing inverse Legendrian maps}
In this section, we prove Theorem \ref{Thm: main 4}.  

Example \ref{exmp: ADE} implies the ``if'' part of Theorem \ref{Thm: main 4}. Now we prove the ``only if'' part.

\begin{lem}\label{lem: no inverse Legendrian map for product manifold}
Any trivial $S^1$-bundle over a surface $\Sigma$ does not admit an inverse Legendrian map.    
\end{lem}
\begin{proof}
Let $M=F\times S^1$ and let $F=\Sigma\times \{1\}$ be a section. Consider the projection map $p:M\to S^1$. Suppose $f:M\ra M$ is any smooth map homotopic to the identity. Then the composition
\[
g: F\hookrightarrow M\xrightarrow{f}M\xrightarrow{p}S^1
\]
is null-homotopic. %Hence $g$ can be lifted to $\widetilde{g}: F\to \mathbb{R}$. Let $x\in F$ be a critical point of $\widetilde{g}$. Then $x$ is a critical point of $g$. 
Therefore $g$ has a critical point. Let $x$ be a critical point of $g$.
Then we have 
\[
f_*(\D_{x})=f_{*}(T_{x}F)\subset \ker (p_*: T_{f(x)}M\to T_{g(x)}S^1)=\mathcal{D}_{f(x)}.
\]
Hence $\V_{f(x)}\not\subset f_{*}(\D_x)$.
\end{proof}

\begin{lem}\label{lem: inverse Lengendrian implies trivial bundle}
    If a closed orientable Seifert manifold $M$ admits an inverse Legendrian map, then the contact plane field $\mathcal D$ has a nowhere vanishing section.
\end{lem}
\begin{proof}
We first assume the base orbifold $\mathcal{O}$ of $M$ is orientable. Then the plane field $\mathcal D$ is orientable. Let $\alpha\in \Omega^1(M)$ be the contact form with $\ker \alpha=\mathcal D$. Suppose $f:M\to M$ is an inverse Legendrian map. Then the 1-form $f^{*}(\alpha)$ is nowhere vanishing when restricted to the horizontal distribution $\mathcal{D}\subset TM$. For each $x\in M$, there exists a unique orthonormal, oriented basis $(u(x),v(x))$ of $\mathcal{D}_x$ such that $\alpha(u(x))=1$ and $\alpha(v(x))=0$. Then we define the section $s$ of $\mathcal{D}$ by setting $s(x)=v(x)$. 

Now suppose $\mathcal{O}$ is unorientable. Take a double cover $\widetilde{M}\to M$ whose base orbifold is orientable. Then the inverse Legendrian map on $M$ lifts to an inverse Legendrian map on $\widetilde{M}$. By the previous case, the contact plane field $\widetilde{\mathcal D}$ has a nowhere vanishing section $\widetilde{s}$. Note that the covering transformation $\tau:\widetilde{M}\to \widetilde{M}$ reverses the orientation of $\widetilde{\mathcal{D}}$ and pulls back the contact form $\widetilde{\alpha}$ to  $-\widetilde{\alpha}$. So the section $\widetilde{s}$ is invariant under $\tau$ and descends to a section $s$ of $\mathcal{D}$.
 \end{proof}

\begin{lem}\label{lem: seifert with sections}
Let $M$ be a Seifert manifold with the $\mathbb S^3$-geometry. Suppose the contact plane field $\mathcal{D}$ has a nowhere vanishing section $s$. Then there exists a finite subgroup $\Gamma$ of $G_0$ (defined in {\rm (\ref{eq: symmetry group of rotation})}) such that $M$ is isometric to $S^3/\Gamma$.
\end{lem}
To prove Lemma \ref{lem: seifert with sections}, we need some preparations. Given a map $f:S^1\to S^1$, let $d(f)$ be the mapping degree of $f$ and define
\[
\operatorname{Dev}(f):=\frac{1}{2\pi}
\int^{2\pi}_{0}(\widetilde{f}(x)-d(f)\cdot x)dx\in \mathbb{R}/2\pi \mathbb{Z} 
\]
where $\widetilde{f}:\mathbb{R}\to \mathbb{R}$ is any lift of $f$. We define the linearization of $f$ by
\[
f^{L}:S^1\to S^1,\quad f^{L}(z):=e^{{\rm i}\cdot\operatorname{Dev}(f) }\cdot z^{d(f)}.
\]
It satisfies the property that 
\begin{equation}\label{eq: equivariance of linearization}
(g_{1}\circ f\circ g_{2})^{L}=g_{1}\circ f^{L}\circ g_{2}  \end{equation}
for any $g_{1},g_2\in O(2)$. When $f\in \operatorname{Diff}_+(S^1)$ this recovers the ``deviation" defined in Section \ref{sec: constructing Legendrian maps}.

\begin{proof}[Proof of Lemma $\ref{lem: seifert with sections}$] Let 
$\pi:M\to \mathcal{O}$ be the Seifert fibration. Since $M$ is of the $\mathbb S^3$-geometry, we have $\mathcal{O}=S^2/\Gamma_0$ for some finite subgroup $\Gamma_0$ of $O(3)$. By pulling back $\pi$, we get a circle bundle $\widetilde{\pi}:\widetilde{M}\to S^2$. Then $\widetilde{M}\to M$ is a covering map with $\Gamma_0$ being the group of covering transformations. The contact plane field $\mathcal{D}$ pulls back to the contact plane field $\widetilde{\mathcal{D}}$ on $\widetilde{M}$. And the section $s$ of $\mathcal{D}$ pulls back to a $\Gamma_0$-equivariant section $\widetilde{s}$ of $\widetilde{\mathcal{D}}$. After normalization, we may assume that $\widetilde{s}$ has length $1$ everywhere. Note that $\widetilde{\mathcal{D}}\cong \widetilde{\pi}^{*}(TS^2)$. So $\widetilde{s}$ induces a map  $\widetilde{f}: \widetilde{M}\to S(TS^2)$ which is $\Gamma_0$-equivariant and covers the identity map on $S^2$. For any $b\in S^2$, we let $F_{b}=\widetilde{\pi}^{-1}(b)$ and let $R_{b}=S(T_{b}S^2)$. Then the restriction of $\widetilde{f}$ gives a map $\widetilde{f}_{b}:F_{b}\to R_{b}$.

We pick any isometries $\rho_{1}:F_{b}\cong  S^1$, and $\rho_{2}:R_{b}\cong  S^1$. Using these isometries, we define $\widetilde{f}^{L}_{b}:F_{b}\to R_{b}$
as the linearization of $\widetilde{f}_{b}$. By (\ref{eq: equivariance of linearization}), the map $\widetilde{f}^{L}_{b}$ is independent of the choice of $\rho_1,\rho_2$. Then we define $\widetilde{f}^{L}:\widetilde{M}\to S(TS^2)$ by $\widetilde{f}^{L}(x)=\widetilde{f}^{L}_{b}(x)$ for any $x\in F_{b}$. Note that $d(\widetilde{f}_{b})=d(\widetilde{f}^{L}_{b})$ is independent of $b\in S^2$. Suppose $d(\widetilde{f}^{L}_{b})=0$. Then $\widetilde{f}^{L}$ is constant on each fiber of $M$ and descends to a section of the bundle $S(TS^2)\to S^2$, which is impossible. Therefore $d(\widetilde{f}^{L}_{b})\neq 0$ for any $b\in S^2$. This implies that $\widetilde{f}^{L}$ is a covering map. Just like $\widetilde{f}$, the map $\widetilde{f}^{L}$ is also $\Gamma_0$-equivariant. So $\widetilde{f}^{L}$ induces a covering map $f: M\to S(TS^2)/\Gamma_0$.

Note that the preimage of $O(3)$ under the map $$\mathrm{Isom}_{+}(S^3)\to \mathrm{Isom}_{+}(\mathbb{RP}^3)=\mathrm{Isom}_{+}(S(TS^2))$$ is the group $G_{0}$ defined in (\ref{eq: symmetry group of rotation}). So $S(TS^2)/\Gamma_0\cong S^3/\Gamma_1$ for some subgroup $\Gamma_1$ of $G_0$. Since $M$ is isometric to a covering space of $S^3/\Gamma_1$, it is also isometric to $S^3/\Gamma$ for some subgroup $\Gamma\subset \Gamma_1\subset G_0$.
 \end{proof}

The proof of Theorem \ref{Thm: main 4} reduces to the following Proposition \ref{prop: inverse Legendrian implies spherical}.

\begin{prop} \label{prop: inverse Legendrian implies spherical}
    Let $M$ be a circle bundle over a closed orientable surface. Suppose $M$ admits an inverse Legendrian map. Then $M$ admits $\mathbb S^3$-geometry.
\end{prop}

\begin{proof}[{Proof of Theorem $\ref{Thm: main 4}$ admitting Proposition $\ref{prop: inverse Legendrian implies spherical}$}]
    The ``if" part is established in Example \ref{exmp: ADE}. For the ``only if" part, let $N$ be a closed Seifert manifold admitting an inverse Legendrian map. Since the Seifert fibration on $N$ is induced by a Seifert geometry, there is a 
    a circle bundle $M$ over a closed orientable surface and a covering map $M\to N$ that preserves the Seifert fibration and the metric. The Seifert manifold $M$ also admits an inverse Legendrian map by Lemma \ref{lem: fintie cover for Legendrian}. By Proposition \ref{prop: inverse Legendrian implies spherical} $M$ admits $\mathbb S^3$-geometry, hence so is $N$. Combining 
    Lemma \ref{lem: no inverse Legendrian map for product manifold}--Lemma \ref{lem: seifert with sections}, we see that $N$ is isomorphic to $S^3/\Gamma$ for some $\Gamma\subset G_0$. 
\end{proof}

The remaining part of this section is devoted to the proof of Proposition \ref{prop: inverse Legendrian implies spherical}.

We begin with some preparations. Let $F$ be a compact, oriented surface with boundary. Let $C\subset \partial F$ be a union of some path components and let $X\in \Gamma(TF|_{C})$ be a nowhere vanishing vector field over $C$. Given a trivialization $\varphi: TF|_{C}\to \mathbb{R}^{2}$ compatible with the orientation of $F$, we define the \textit{rotation number} $\rot_{\varphi}(X)\in \mathbb{Z}$ as the total mapping degree of the map 
\[
\varphi \circ X: C\to \mathbb{R}^2\setminus \{0\}.
\]
Here we orient $C$ as the boundary of $F$. Given $X_{1},X_{2}\in \Gamma(TF|_{C})$, we define \[
\rot(X_1,X_2):=\rot_{\varphi}(X_1)-\rot_{\varphi}(X_2)\in \mathbb{Z}.
\]
The number $\rot(X_1,X_2)$ is independent with the choice of $\varphi$. The following lemma is a special case of the well-known Poincar\'e--Hopf theorem. 

\begin{lem}\label{lem: poincare-Hopf}
Let $\vec{n}\in \Gamma(TF|_{\partial F})$ be the outward normal vector field of $F$. And let $X\in \Gamma(TF)$ be a nowhere vanishing vector field over $F$. Then we have  $\rot(X|_{\partial F},\vec{n})=-\chi(F)$.   
\end{lem}
\begin{proof}
Let $C_1,\ldots, C_m$ be the boundary-components of $F$  and    let $\hat F$ be the closed oriented surface obtained by capping off a disk $D_i$ to each $C_i$. %boundary component $C_i$ of $F$. 
%We can extend $X$ to the %whole surface 
Extend $X$ to a vector field $\tilde X$ on $\hat F$ with a unique isolated singularity $p_i$ on each $D_i$. %, still denoted as $X$. %over 
%$\hat F$ such that $X$ has exactly one isolated singularity (zero) $p_i$ on each disk $D_i$. 
%One can verify% by definition 
%that the index of $\tilde X$ at the singularity $p_i$, ${\rm index}(\tilde X,p_i)$, is equal to $1-{\rm rot} (X|_{C_i}, \vec{n}|_{C_i})$. %, where $C_i$ is the corresponding boundary component of $F$. 
Denote the index of $\tilde X$ at the singularity $p_i$ by $\text{index}(\tilde X, p_i)$. By Poincar\'e--Hopf Theorem, we have $\sum_{i=1}^m \text{index}(X,p_i)=\chi(\hat{F})$. %Then $\vec{n}|_{C_i}$ is the inward normal vector field of $C_i=\partial D_i\subset D_i$. Note that in the definition of $\text{rot}$ we use the induced orientation from $F$. %One can check that
%By this one get that 
%\noindent{\bf Claim.} 
%One can verify that $\text{index}(\tilde X,p_i)=1-\text{rot}(X|_{C_i},\vec{n}|_{C_i}).$
%\noindent{\bf Proof.} We identify $D_i$ with the unit disk $D\subset \mathbb{R}^2$. 
    %Since the sum of ${\rm rot}(X|_{C_i}, \vec{n}|_{C_i})$ is equal to ${\rm rot}(X|_{\partial F}, \vec{n})$, 
    %using the Poincar\'e--Hopf theorem on $\hat F$ we obtain that
  % Since $-\chi(F)=m-\chi(\hat F)=\sum_{i=1}^m(1-\text{index}(\tilde X, p_i))$, we need only to show 
  % \begin{align*}
   %\text{rot}(X|_{\partial F}, \vec{n})=\sum_{i=1}^m(1-\text{index}(\tilde X, p_i)).
   %\end{align*}
   Let $(X_i,\vec{n}_i)=(X|_{C_i}, \vec{n}|_{C_i})$ be the restriction of the vector fields $(X,\vec{n})$ to $C_i$.
   %We write $X|_{C_i}=X_i$ and $\vec{n}|_{C_i}=\vec{n}_i$. %Note that $\text{rot}(X|_{\partial F}, \vec{n})=\sum_{i=1}^m\text{rot}(X_i, \vec{n}_i),$ we need only to prove
   %\begin{align*}
   %\text{rot}(X_i,\vec{n}_i)=1-\text{index}(\tilde X, p_i).
   %\end{align*}
   Let $X'_i$ be a nowhere vanishing vector field on  an open neighborhood of $D_i$ in $\hat F$. Then clearly
   \begin{align*}
   \text{rot}(X_i, \vec{n}_i)=\rot(X_i, X'_i)+\rot(X'_i, \vec{n}_i).
   \end{align*}
   Note that $D_i$ and $F$ induce opposite orientations on $C_i$, one can check that
   \begin{align*}
   \rot(X_i, X'_i)=-\text{index}(\tilde X, p_i),\quad \rot(X'_i, \vec{n}_i)=1.
   %\text{rot}(X_i,\vec{n}_i)+\text{index}(\tilde X, p_i)=\text{index}(\vec{n_i}, p_i).
   \end{align*}
   So we have
   \begin{align*}
   \rot(X_i,\vec{n}_i)=1-\text{index}(\tilde X, p_i).
   \end{align*}
    Note that 
    \begin{align*}
    \text{rot}(X,\vec{n})=\sum_{i=1}^m\text{rot}(X_i,\vec{n}_i).
    \end{align*}
    We have
   \begin{align*}
   \rot(X,\vec{n})=\sum_{i=1}^m(1-\text{index}(\tilde X, p_i))=\sum_{i=1}^m 1-\chi(\hat F)=m-\chi(\hat F)=-\chi(F).
   \end{align*}
   %This can be verified by definition.
   %To show the result we need only to show $\text{rot}(X|_{\partial F}, \vec{n})=\sum_{i=1}^m(1-\text{index}(\tilde X, p_i))$, since then we would have
    %\begin{align*}
     %   \rot(X|_{\partial F},\vec n)%&=\sum_{i=1}^m \rot(X|_{C_i},\vec n|_{C_i})=
      %  =\sum_{i=1}^m(1-\operatorname{index}(\tilde X,p_i))%\\&
       % =m-\chi(\hat F)=-\chi(F).
    %\end{align*}
    %the sum of ${\rm index}(X,p_i)$ is equal to $\chi(\hat F)$., we obtain the result.
    %Now summing up and using Poincare-Hopf theorem for $X$ leads to the result.
\end{proof}
From now on, we assume $M$ is a nontrivial circle bundle over a closed surface $\Sigma$. We orient $M$ such that $e(M)<0$. Let $\alpha\in \Omega^{1}(M)$ be the contact form with $\ker \alpha=\D$. Then our orientation convention implies that $\alpha\wedge d\alpha>0$ everywhere: to see this,
%let $\omega$ be the connection 1-form of the connection $A$ in a local trivialized coordinate chart $(U\times S^1,(x_1,x_2,\theta)).$ %$(x_1,x_2$ indicates the $\Sigma$-components and $\theta$ indicates the $S^1$-component. 
a direct computation shows that $\alpha\wedge d\alpha=%-d\theta\wedge d\omega=
-d\theta\wedge K_A$, where $d\theta$ is the fiberwise volume form. %volume form on $S^1$-fiber.
Since $\int_\Sigma K_A=e(M)<0,$ and $K_A$ is nowhere vanishing, we have that $K_A<0$ everywhere. Therefore $\alpha\wedge d\alpha=-d\theta\wedge K_A>0$.

We let $f:M\to M$ be an inverse Legendrian map. We use $\pi: M\to \Sigma$ to denote the projection map. Consider the composition
\[
\pi_{1}:= \pi\circ f: M\to \Sigma.
\]
We take a regular value $b\in \Sigma$ of $\pi_{1}$. And we take an orientation compatible local chart 
\[\psi: W\xrightarrow{\cong} D^2\] near $b$ such that all points in $W$ are regular values of $\pi_1$. We set $V=\pi^{-1}(W)$ and set $U=\pi^{-1}_{1}(W)=f^{-1}(V)$.

\begin{lem}
The map $f|_{U}: U\to V$ is a covering map.    
\end{lem}
\begin{proof} It suffices to check that the differential 
\[
f_*: T_{x}M\to T_{f(x)}M
\]
is surjective for all $x\in U$. For all such $x$, $\pi_{1}(x)\in V$ is a regular value of $\pi_{1}$. Hence the composition 
\[
T_{x}M\xrightarrow{f_*} T_{f(x)}M\xrightarrow{\pi_*} T_{\pi_{1}(x)}\Sigma
\]
is surjective. On the other hand, we also have 
\[
\ker \pi_{*}=\V_{f(x)}\subset f_{*}(\D_x)\subset f_*(T_xM).
\]
Therefore $f_*(T_xM)\supseteq \ker\pi_*$ and $\pi_*(f_*(T_xM))=T_{\pi_1(x)}\Sigma$. 
Hence $f_*(T_xM)=T_{f(x)}M$.
\end{proof}
We denote the path components of $U$ by $U_1,U_2,\ldots, U_{k}$. Let $f_{i}:U_i\to V$ be the restriction of $f$. We equip both $U_{i}$ and $V$ with the orientation restricted from $M$ and denote the mapping degree of $f_{i}$ by $d_{i}$. Then
%\begin{align*}
    $\sum_{i=1}^kd_i={\rm deg}(f|_{(U,\partial U)}:(U,\partial U)\to (V,\partial V))={\rm deg}(f)=1.$
%\end{align*}
We fix an orientation preserving diffeomorphism 
$\varphi: D^2\times S^1\xrightarrow{\cong }V$ such 
that $\varphi_{*}(\frac{\partial}{\partial t})=\V|_{V}$
and the composition 
\[
D^2\times S^1\xrightarrow{\varphi} V\xrightarrow{\pi} W\xrightarrow{\psi} D^2
\]
sends $(x,y,t)$ to $(x,y)$. We also fix an orientation-preserving diffeomorphism $\varphi_{i}:D^2\times S^1\xrightarrow{\cong }U_{i}$ such that the composition \[D^2\times S^1\xrightarrow{\varphi_{i}} U_{i}\xrightarrow{f_i} V\xrightarrow{\varphi^{-1} } D^2\times S^1\] sends $(x,y,t)$ to $(x,y,d_{i}\cdot t)$. Note that this composition sends the vector field $\frac{\partial}{\partial t}$ to the vector field $d_{i}\cdot \frac{\partial}{\partial t}$. Hence \[(f_{i})_*\circ(\varphi_{i})_*\Big(\frac{\partial}{\partial t}\Big)=d_{i}\cdot \V|_{V}=\varphi_{*}\Big(d_{i}\cdot \frac{\partial}{\partial t}\Big).\] Since $f$ is inverse Legendrian, we have $\V|_{V}\subset (f_{i})_*(\ker \alpha)$. So $(f_{i})_*\circ(\varphi_{i})_*(\frac{\partial}{\partial t})\subset (f_{i})_*(\ker \alpha)$. Since the map $(f_{i})_*:T_{x}U_{i}\to T_{f(x)}V$ is injective, we have $(\varphi_{i})_*(\frac{\partial}{\partial t})\subset \ker \alpha$, i.e. $\frac{\partial}{\partial t}\in \ker \varphi_{i}^*(\alpha)$. So after rescaling, we may write 
\[\alpha|_{U_i}=\cos \theta_i \cdot dx+\sin \theta_i \cdot dy\in \Omega^{1}(D^2\times S^1)\]
for some function $\theta_{i}: U_i\to S^1$. 

On the torus $\partial V$, we have a meridian $m=\varphi(\partial D^2\times \{0\})$ and the longitude $l=\varphi(\{0\}\times S^1)$. We orient $m$ as the boundary of $\varphi(D^2\times \{0\})$ and orient $l$ via the vector field $\varphi_{*}(\frac{\partial}{\partial t})$. Similarly, we can define the meridian $m_{i}$ and the longitude $l_{i}$ on $\partial U_{i}$.

The mapping degree of the map $\theta_i:m_i\to S^1$ equals $0$ because it extends to $D^2$. We use $L_{i}$ to denote the mapping degree of the map 
\[
\theta_i: l_i\to S^1.
\]
\begin{lem} $L_{i}<0$ for any $i$. 
\end{lem}
\begin{proof}
It is straightforward to compute that  \begin{equation}\label{eq: local expression of alpha wedge dalpha}
(\alpha\wedge d\alpha)|_{U_i}=-\frac{\partial \theta_i}{\partial t}(x,y,t) dx\wedge dy\wedge dt.    
\end{equation} 
Recall that we have equipped both $U_{i}$ and $V$ with the orientation restricted from $M$. Then the form \[dx\wedge dy\wedge dt\in \Omega^{1}(V)\] is everywhere positive. We have $f^*_{i}(dx\wedge dy\wedge dt)=d_{i}\cdot dx\wedge dy\wedge dt$. When  $d_{i}>0$, $f_{i}$ is orientation preserving so $d_{i}\cdot dx\wedge dy\wedge dt$ is everywhere positive. When $d_{i}<0$, $f_{i}$ is orientation reversing so $d_{i}\cdot dx\wedge dy\wedge dt$ is everywhere negative. In both case, the form $dx\wedge dy\wedge dt\in \Omega^{1}(U_{i})$ is positive everywhere. Since $\alpha\wedge d\alpha$ is also everywhere positive, by  (\ref{eq: local expression of alpha wedge dalpha}), we  have $\frac{\partial \theta_i}{\partial t}<0$ everywhere.  This implies \[L_i=\int_{S^1}\frac{\partial\theta_i}{\partial t}dt<0.\]
\end{proof}
Now we consider the bundle 
\[
S^1\hookrightarrow M\setminus \mathring{V}\xrightarrow{\pi} \Sigma\setminus \mathring{W}.
\]
This bundle is trivial and has a section $s:\Sigma\setminus \mathring{W}\to M\setminus \mathring{V}$. We use $F'$ to denote $s(\Sigma\setminus \mathring{W})$. Then $F'$ is a properly embedded surface in $M\setminus \mathring{V}$. 
\begin{lem}\label{lem: f transverse to F'}
The map $f:M\setminus \mathring{U}\to M\setminus \mathring{V}$ is transverse to the surface $F'$.     
\end{lem}
\begin{proof}
Because $\V_{f(x)}\subset f_{*}(T_{x}M)$, we have
\[
T_{f(x)}M=T_{f(x)}F'+\V_{f(x)}\subset T_{f(x)} F'+f_{*}(T_{x}M).
\]
\end{proof}
By Lemma  \ref{lem: f transverse to F'},   $f^{-1}(F')$ is a surface embedded in $M\setminus \mathring{U}$, denoted by $F$. Furthermore, the orientation on $\Sigma\setminus \mathring{W}$ induces an orientation on $F'$ and further induces an orientation on $F$. Hence for each point $x\in F$, we have oriented subspaces $\D_{x}$ and $T_{x}F$ of $T_xM$. Note that for  any $x\in F$, we have $\D_{x}\neq T_{x}F$. This is because $f_{*}(\D_{x})$ contains $\V_{f(x)}$, while $f_{*}(T_{x}F)$ is contained in $T_{f(x)}F'$ so intersects trivially with $\V_{f(x)}$. Then $\D_{x}\cap T_{x}F$ is an oriented 1-dimensional vector space. By taking the unit vector in $\D_{x}\cap T_{x}F$, we obtain a nowhere vanishing vector field $X\in \Gamma(TF)$.

For $1\leqslant  i\leqslant  k$, we let $C_{i}=\partial F\cap \partial U_{i}$, oriented as the boundary of $F$. Let $X_{i}=X|_{C_{i}}\in \Gamma(TF|_{C_{i}})$. Let $\vec{n}_{i}$ be the outward normal vector field of $F$ on $C_{i}$. 
\begin{lem}
We have $\rot(X_i,\vec{n}_{i})=-  e(M)L_{i}+d_{i}$.    
\end{lem}
\begin{proof} Consider the loop $S^1=\partial F'\subset \partial V$, oriented as the boundary of $F'\cong \Sigma\setminus \mathring{W}$. Then we have 
\[
[\partial F']=-([m]+e(M)[l])\in H_{1}(\partial V).
\]
Note that the map  $f_{i}: \partial U_{i}\to \partial V$ satisfies 
\[
f_{i,*}([m_{i}])=[m],\quad f_{i,*}([l_{i}])=d_{i}[l],\quad f_{i,*}[C_{i}]=d_{i} [\partial F'].
\]
This implies
\begin{equation}\label{eq: homology of Ci}
[C_{i}]=[f^{-1}_{i}(\partial F')]= -d_{i}[m_{i}]- e(M)\cdot [l_{i}].    
\end{equation}
Consider the map 
\[
F\xrightarrow{f} F'\xrightarrow{ \pi} \Sigma\setminus \mathring{W}
\]
and the trivialization 
\[
\tau_{i}: TF|_{C_{i}}\xrightarrow{ \pi_*\circ f_*} T\Sigma|_{\partial (\Sigma\setminus \mathring{W})}=T\Sigma|_{\partial \mathring{W}}\cong D^2\times \mathbb{R}^2.
\]
Under this trivialization, the unit tangent vector field on $C_{i}$, denoted by $T_{i}$, is mapped to a multiple of the unit tangent vector field of $\partial D^2$. Therefore, by (\ref{eq: homology of Ci}), we have 
\[
\rot_{\tau_{i}}(\vec{n}_{i})=\rot_{\tau_{i}}(T_{i})=-d_{i}.
\]
On the other hand, the rotation number $\rot_{\tau_{i}}(X_{i})$ equals the mapping degree of the map $\theta_{i}: C_{i}\to S^1$. Hence by (\ref{eq: homology of Ci}), we have 
\[
\rot_{\tau_{i}}(X_{i})=-e(M) L_{i}.
\]
Hence we have proved that $\rot(X_{i},\vec{n}_{i})=\rot_{\tau_{i}}(X_{i})- \rot_{\tau_{i}}(\vec{n}_{i})=-e(M)L_{i}+d_{i}$.
\end{proof}

\begin{proof}[Proof of Proposition $\ref{prop: inverse Legendrian implies spherical}$] Assume that a circle bundle $S^1\hookrightarrow M\xrightarrow{\pi} \Sigma$ admits an inverse Legendrian map. By Lemma \ref{lem: no inverse Legendrian map for product manifold}, we have $e(M)\neq 0$. We orient $M$ such that $e(M)<0$. It remains to show that $\Sigma=S^2$.

Let $F, X$ be defined as above. Let $\vec{n}$ be the outward normal vector of $F$. Then we have 
\[
\rot(X|_{\partial F},\vec{n})=\sum^{k}_{i=1}\rot(X_{i},\vec{n}_i)=-\sum^{k}_{i=1}e(M)L_{i}+\sum^{k}_{i=1}d_{i}=-\sum^{k}_{i=1}e(M)L_{i}+1.
\]
Since $e(M)<0$ and $L_{i}<0$, we have $\rot(X|_{\partial F},\vec{n})\leqslant  0$. 
By Lemma \ref{lem: poincare-Hopf}, we see that $\chi(F)\geqslant 0$. Since $F$ admits a nowhere vanishing tangent vector field, $F$ has no spherical components. Therefore, either $F$ is a union of tori and annuli, or $F$ contains a disk component.

Consider the composition  
\[
h=\pi\circ f: F\to  \Sigma\setminus \mathring{W}
\]
Then the restriction 
\[
h|_{\partial F}:\partial F\to \partial (\Sigma\setminus \mathring{W})
\]
is a covering map of total degree $\sum^{k}_{i=1}d_{i}=1$. Given any annulus component $A\subset F$, the total mapping degree of the map $h|_{\partial A}: \partial A\to \partial (\Sigma\setminus \mathring{W})$ must be zero. Since  the total mapping degree of $h|_{\partial F}$ equals $1$, some component of $F$ must have a nonempty boundary and is not an annulus. 

Therefore, $F$ has a disk component $D^2$. Since $h|_D:D\to \Sigma\setminus \mathring W$ is a covering map on its boundary, $\Sigma\setminus \mathring W$ must be a disk. It follows that $\Sigma=S^2$.
% Since $f:M\to M$ is inverse Legendrian, the 1-form $f^{*}(\alpha)$ is nowhere vanishing when restricted to the horizontal distribution $\mathcal{D}\subset TM$.  This implies  $\mathcal{D}$ has a nowhere vanishing section. So $\mathcal{D}$ is trivial as a plane bundle over $M$. Hence the Euler class $e(\mathcal{D})=0\in H^{2}(M)$. Note that $\mathcal{D}=\pi^{*}(T\Sigma)$. So $e(\mathcal{D})=\pi^{*}(e(T\Sigma))$. Since $\langle e(T\Sigma),[\Sigma]\rangle =2$, we see that
% \[
% \pi^{*}(e(T\Sigma))=2\in H^{2}(M)\cong \mathbb{Z}/e(M)\cdot \mathbb{Z}.
% \]
% Therefore, we have $e(M)\in \{\pm 1,\pm 2\}$.
% The key of the previous paragraph is if $M$
% is a circle bundle over $\Sigma$ such that $M$ has an horizontal distribution  which admit a nowhere vanishing section, then $e(M) \mid \chi(\Sigma)$. 
\end{proof}

\section{Lifting property of nonzero degree maps}
\begin{lem}\label{MapLiftingAlgebraicLemma}
    Suppose there is a commutative diagram of groups 
    \[ \begin{CD}
1 @> >> K @> \tilde{i} >> \tilde G @> \tilde j
>> \tilde H @> >>1\\
@. @V p|VV  @V pVV @V \bar{p}VV @.\\
1 @> >> K @> i >> G @> j
>> H @> >>1\\
@. @V f|VV  @V fVV @V \bar{f}VV @.\\
1 @> >> K @> i >> G @> j
>>H@> >>1
\end{CD}  \]
with the following properties:
\begin{enumerate}[\quad\rm (1)]
    \item each row is exact;
    \item $p|:K\ra K$ is a surjection;
    \item $\bar f(\operatorname{im}\bar p)\subset \operatorname{im}\bar p$.
\end{enumerate}
Then we have $f(\operatorname{im}p)\subset \operatorname{im}p$.
\end{lem}
\begin{proof}
    Fix any $\alpha\in \tilde G$, let $\beta:=f\circ p(\alpha)\in G$, we need to prove that $\beta\in\operatorname{im}f$.

    By commutativity, we have $j(\beta)=\bar f\circ \bar p \circ \tilde j(\alpha)$, this element belongs to $\operatorname{im} \bar p$ by condition (3). Since $\tilde j:\tilde G\ra \tilde H$ is surjective, we can find $\gamma\in \tilde G$ such that $j(\beta)=\bar p\circ \tilde j(\gamma)$. This amounts to say that $j(\beta)=j\circ p(\gamma)$, so $\beta\cdot p(\gamma^{-1})\in \ker j$. By exactness and by condition (1), we can find $w\in K$ such that $i\circ p|(w)=\beta\cdot p(\gamma^{-1})$. Finally, we have
    \[
        \beta=i\circ p|(w)\cdot p(\gamma)=p\circ \tilde i(w)\cdot p(\gamma)=p(\tilde i(w)\gamma)\in\operatorname{im}p.
    \]
    This finishes the proof.
\end{proof} 

%The following Lemma is \cite[Theorem 2.9]\cite{SWW}.

\begin{lem}\label{lift1}
    Let $M$ be a closed orientable manifold which supports either {\rm Nil} or 
    $\mathbb{H}^2\times \mathbb{E}^1$-geometry. Then there is a finite covering $p: \tilde M \to M$, where $\tilde M$ is a $S^1$-bundle
    over a closed orientable surface,
    such that any non-zero degree map $f: M\to M$ can be lifted to  map $\tilde f: \tilde M \to \tilde M$.
\end{lem}

\begin{proof} It is known that $M$ has a unique Seifert fibration$$S^1\xrightarrow{i} M\xrightarrow{q} O,$$
where the orbifold  $O$ is a closed hyperbolic or Euclidean orbifold. 

Since $M$ is aspherical, it  induces a short exact sequence on fundamental groups 
$$1\longrightarrow \pi_1(S^1)\xrightarrow{i_*} \pi_1(M)\xrightarrow{q_*}\pi_1(O)\longrightarrow 1.$$
Then 
there is a charactersitic covering $\Sigma_g\rightarrow O$, where $\Sigma_g$ is a closed hyperbolic, or Euclidean orientable surface. 
Now we have the following commutative diagram:
\[    \begin{CD}
S^1@> \tilde i>>\tilde M@> \tilde{q} >> \Sigma_g\\
 @V p| VV@V p VV @VV \bar p V\\
S^1@> i>>M @>  q >> O
\end{CD}
\]
where 
$S^1\xrightarrow {\tilde i}\tilde M\xrightarrow{\tilde q}\Sigma_g$
is the pullback of $M$ via $\Sigma_g\rightarrow O$. 
%Since $\tilde M$ also supports $\mathbb{H}^2\times \mathbb{E}^1$-geometry with the base space a surface, $\tilde M=S^1\times \Sigma_g$ is a product manifold.
Denote by $p$ the map $\tilde M\rightarrow M$ and $\bar{p}$ the covering $\Sigma_g\rightarrow O$.
Moreover, $p|:S^1\ra S^1$ induces an isomorphism on fundamental groups, $\bar p: \pi_1(\Sigma_g)\to \pi_1(O)$ is injective.

Let $f: M\to M$ be any map of non-zero degree. It is known that $f_* :\pi_1(M)\to \pi_1(M)$ is injective. We are left to show that $$f_*(p_*(\pi_1(M))) \subset p_*(\pi_1(M)),$$ so that $f$ can be lifted to a map $\tilde f:\tilde M \to \tilde M$. It is known \cite{Ro} that $f$ is homotopic to a Seifert fiber preserving map.
Since a homotopy does not affect the desired lifting property, we may assume that $f$ is a Seifert fiber preserving map.
Then we have the following commutative diagram on fundamental groups:
\[ \begin{CD}
1 @> >> \pi_1(S^1) @> \tilde{i}_* >> \pi_1(\tilde M) @> \tilde{q}_* >>\pi_1(\Sigma_g)@> >>1\\
@. @V (p|)_*VV  @V p_*VV @V \bar{p}_*VV @.\\
1 @> >> \pi_1(S^1) @> i_* >> \pi_1(M) @> q_*>>\pi_1(O)@> >>1\\
@. @V (f|)_*VV  @V f_*VV @V \bar{f}_*VV @.\\
1 @> >> \pi_1(S^1) @> i_* >> \pi_1(M) @> q_*>>\pi_1(O)@> >>1\\
%@.  @A (p|)_*AA  @A p_*AA @A \bar{p}_*AA @.\\
%1 @> >> \pi_1(S^1) @> \tilde{i} >> \pi_1(S^1\times \Sigma) @> \tilde{q}>>\pi_1(\Sigma_g)@> >>1,\\
\end{CD}  \]
where $\bar{f}: O\to O$ is the induced orbifold map. Since the degree of  $f$  is non-zero, the degree of $\bar f$ must be non-zero.
Since  $f_*$ is an injection, and it follows that $\bar f_*$ is an injection.

\textbf{Case 1.}  $O$ is a hyperbolic orbifold: Then degree of $\bar f$ is $\pm 1$. It follows that $f_*: \pi_1(O)\to \pi_1(O)$ is surjective.
 So $\bar f_*$ is an isomorphism. Note that $\bar p_*(\pi_1(\Sigma_g))$ is characteristic in $\pi_1(O)$, we then have $\bar f_*(\operatorname{im}\bar p_*)\subset \operatorname{im}\bar p_*$. All conditions of Lemma \ref{MapLiftingAlgebraicLemma} are satisfied and the proof is finished.

\textbf{Case 2.} $O$ is an Euclidean orbifold: Then $\Sigma$ is the torus $T$. 

We first assume that $O$ is orientable.
Consider the representations  $\pi_1(O)\to \operatorname{Isom}_+\mathbb R^2$ and $\pi_1(T)\to \pi_1(O)\to \operatorname{Isom}_+\mathbb R^2$ given by the Euclidean geometry. Then each $g \in \pi_1(O)$ is either a translation, or a rotation of finite order. Note that 
$$\im {\bar p_*}=\pi_1(T)=\{g\in \pi_1(O)\mid g=e, \ \text{or} \  \operatorname{ord}(g)=\infty \}.$$
Since $\bar f_*$ is injection, each element in $\bar f_*(\im {\bar p_*})$ is either the identity or has infinite order. Therefore 
$$\bar f_*(\im {\bar p_*})\subset \im {\bar p_*}.$$
All conditions of Lemma \ref{MapLiftingAlgebraicLemma} are satisfied and the proof is finished.

If $O$ is non-orientable, then $O$ is either the Klein bottle, or the real projective plane with two singular points of index two.
In each case, $M$ has the torus semi-bundle structure \cite[pages 38, 40]{Ha}. $M$ is doubly covered by a unique torus bundle over circle $\overline M$, and Theorem 2.9 of \cite{SWW} shows that $f$ can be lifted to $\overline M$. A torus bundle over circle supporting Nil geometry is either a circle bundle over the torus, or a twisted circle bundle over the Klein bottle. If $\overline M$ falls in the later case, then any nonzero degree maps on $\overline M$ can be lifted to a circle bundle over the torus. This completes the proof.
\end{proof}

\section{Flexible exponent of geometric  3-manifolds}

In this section we prove Theorem \ref{main1}. Since $\alpha(M)=0$ for $M$ supporting either  $\mathbb H^3$ or 
$\widetilde {\rm SL_2}$ 
geometries, it suffices to prove the following:

\begin{thm}\label{main1.}
Let $M$ be a connected closed orientable 3-manifold.
    \begin{enumerate}[\quad \rm (1)]
        \item $\alpha(M)=3$ if $M$ supports the geometry of either  $\mathbb S^3$, or $\mathbb E^3,$ or  $\mathbb S^2\times \mathbb E^1$;

    \item $\alpha(M)= 8/3$ if $M$ supports the {\rm Nil} geometry; 
        
        \item  $\alpha(M)=2$ if $M$ supports the {\rm Sol} geometry;
        
         \item $\alpha(M)=1$ if $M$ supports the  $\mathbb H^2\times \mathbb E^1$ geometry;
          \end{enumerate}
        \end{thm}

         Theorem \ref{main1.} (1)--(4) will be proved in the next four subsections respectively.

\subsection{$\alpha(M)=3$ for $\mathbb S^3$, $\mathbb S^2\times \mathbb E^1$ and  $\mathbb E^3$ geometries}
\begin{prop}\label{ScalableGeometry}
    Suppose $M$ is a closed orientable $3$-manifold admitting either $\mathbb S^3$,  or $\mathbb S^2\times \mathbb E^1$ or $\mathbb E^3$ geometry. Then %$P_M(L)\asymp L^3$.
    $\alpha(M)=3$.
\end{prop}
    We need the following fact.
   \begin{lem}\label{OddScalable}
        Let $S^{n}$ be the standard $n$-dimensional sphere. Then there is a constant $C=C(n)$ such that for any $L \geqslant 1$, there is an odd map  $h_L:S^n\ra S^n$  (i.e. $h_L$ commutes with the antipodal map $\iota:S^n\ra S^n$) such that
        \[
            \deg h_L\geqslant CL^n,\quad\lip h_L\leqslant L.
        \]
    \end{lem}
    \begin{figure}[h]
        \centering
        \includegraphics[width=0.65\linewidth]{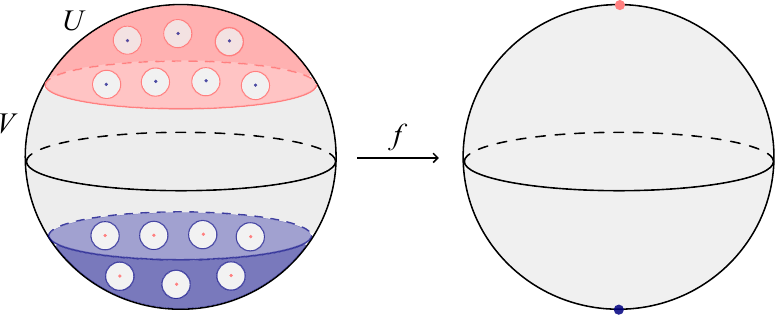}
        \caption{An illustration of an odd Lipschitz map $f$. All red (blue) regions shrink to the red (blue) point, all gray regions are mapped homeomorphically to the gray region.}
        \label{fig:OddLipschitzMapping}
    \end{figure}
    \begin{proof}
        The basic observation is that, as $L$ tends to $+\infty$, one can find approximately $L^n$ points in a standard $n$-sphere so that their pairwise distance are greater than $\frac 2L$. Let $B_1,\ldots,B_k$ be the $\frac 1L$-neighborhood of these points, construct a map $h_L:S^n\ra S^n$ shrinking $S^n-(B_1\sqcup\cdots \sqcup B_k)$ to a point $p\in S^n$ and mapping each $B_i$ diffeomorphically and orientation preservingly onto $S^n-\{p\}$. This map can be made $\pi L$-Lipschitz, and has degree $\deg h_L\approx L^n$.
        
        To construct odd maps $h_L:S^n\ra S^n$ with controlled Lipschitz constants and large mapping degrees, the strategy is similar. let $H_+$ and $H_-$ be the upper and lower hemisphere of the standard unit $n$-sphere, then we only need to construct the map $(h_L)|_{H_+}$ and define $(h_L)|_{H_-}=\iota\circ (h_L)|_{H_+}\circ \iota$. 
        Let $\gamma$ be the equator of $S^n$. Decompose
        \[
            H_+=U\sqcup V,
        \]
        where
        \begin{align*}
            &U=\Big\{x\in H_+\mid \operatorname{dist}(x,\gamma)> \frac\pi 4\Big\},\\&
            V=\Big\{x\in H_+\mid \operatorname{dist}(x,\gamma)\leqslant\frac\pi 4\Big\}.
        \end{align*}
        See Figure \ref{fig:OddLipschitzMapping} for an illustration. The subset $V$ has two boundaries: $\partial V=\gamma \sqcup \gamma'$ where $\gamma$ is the equator and $\gamma'$ is an $(n-1)$-sphere consisting of points whose distance to the equator equal $\frac \pi 4$. Fix a map $\phi:V\ra H^+$ such that
        \begin{itemize}
            \item $\phi|_{\gamma}$ is the identity, $\phi(\gamma')=\{\text{north pole}\}$.
            \item 
            $\phi$ restricts to a homeomorphism between $\mathring V$ and $\mathring H_+-\{\text{north pole}\}$.
            \item $\lip \phi\leqslant 10$.
        \end{itemize}
    Next, there exists a constant $C=C(n)>0$ such that we can find more than $CL^n$ points in $U$ with their pairwise distance greater than $\frac 2L$, and their distance to $\partial U=\gamma'$ greater than $\frac 1L$. Let $B_1,\ldots,B_k$ be the $\frac 1L$-neighborhood of these points, construct a map $\psi_L:U\ra S^n$ shrinking $U-(B_1\sqcup\cdots \sqcup B_k)$ to the north pole and mapping each $B_i$ homeomorphically and orientation preservingly onto $S^n-\{\text{north pole}\}$. This map can be made $\pi L$-Lipschitz. Finally, we define $h_L:S^n\ra S^n$,
    \[
        h_L(x)=\left\{
        \begin{aligned}
            &\phi(x),\quad x\in V.\\&
            \psi_L(x),\quad x\in U.\\&
            \iota\circ h_L(x)\circ \iota,\quad x\in H_-.
        \end{aligned}
        \right.
    \]
    This is an odd map. Moreover, any point on the equator is a regular value of $h_L$, whose inverse image consists of one point on the equator, $k$ points in $H_+$ and $k$ points in $H_-$. All these inverse image points have the same local degree, so we have
    \[
        \deg h_L=2k+1> 2CL^n,\quad \lip h_L=\max\{10,\ \pi L\}.
    \]
   Reassigning the constant $C(n)$ and the proof is finished.
    \end{proof}

    \begin{proof}[{Proof of Proposition $\ref{ScalableGeometry}$}]
    
        \textbf{Case 1.} Suppose $M$ admits $\mathbb S^3$-geometry, then $S^3$ is a finite covering of $M$. On the other hand, any closed orientable 3-manifold admits a degree-one map to $S^3$.

      By  Lemma \ref{OddScalable} and Corollary \ref{Topological invariance of flexible exponent},  $\alpha(S^3)=3$. Then by Corollary \ref{Mutual nonzero degree map have the same flexible exponent}, we have $\alpha(M)=\alpha(S^3)=3$.

       % Since $P_{S^3}(L)\asymp L^3$, by Lemma \ref{BiFiniteDegreeMapLemma}, we have $P_M(L)\asymp L^3$.

    \textbf{Case 2.} Suppose $M$ admits $\mathbb S^2\times\mathbb E^1$-geometry, then $M$ is homeomorphic to either $S^2\times S^1$ or $\RP^3\#\RP^3$. In fact, $\RP^3\#\RP^3$ is 2-fold covered by $S^2\times S^1$. Let $\iota:S^2\ra S^2$ be the antipodal map; view $S^1$ as the set of unit complex numbers and let $\alpha:S^1\ra S^1$ be the complex conjugate. Endow $S^2\times S^1$ with the standard product metric, then there is an isometric involution on $S^2\times S^1$:
    \[
        \phi:S^2\times S^1\longrightarrow S^2\times S^1,\quad (x,z)\longmapsto (\iota(x),\bar z).
    \]
    This defines free action of $\Z/2\Z$ over $S^2\times S^1$, the quotient space is $\RP^3\#\RP^3$.

    Let's construct a family of $\Z/2\Z$-equivariant self-maps of $S^2\times S^1$. For any positive integer $k$, let $h_k:S^2\ra S^2$ be an odd map such that $\lip h_k\leqslant k$ and $\deg h_k\geqslant Ck^2$ by Lemma \ref{OddScalable}. Consider
    \[
        \tilde f_k:S^2\times S^1\longrightarrow S^2\times S^1,\quad (x,z)\longmapsto(h_k(x), z^k).
    \]
    Then
    \[
        \phi\circ \tilde f_k(x,z)=(\iota\circ h_k(x),\bar z^k)=(h_k\circ\iota(x),\bar z^k)=\tilde f_k\circ\phi(x,z),
    \]
    i.e. $\tilde f_k$ commutes with the $\Z/2\Z$-action, and induces $f_k:\RP^3\#\RP^3\ra \RP^3\#\RP^3$. We know that 
    \[
        \deg f_k=\deg \tilde f_k\geqslant Ck^3,\quad \lip f_k=\lip \tilde f_k\leqslant k,
    \]
    this shows that
    \[
        \alpha(S^2\times S^1)\geqslant 3,\quad \alpha(\RP^3\#\RP^3)\geqslant 3.
    \]
    Combining with the natural upper bound given by Corollary \ref{Topological invariance of flexible exponent}, we have $$\alpha(S^2\times S^1)=\alpha(\RP^3\#\RP^3)=3.$$

     \textbf{Case 3.} 
    Suppose $M$ admits $\mathbb E^3$-geometry. Then $M=\mathbb E^3/\Gamma$, where $\Gamma\subset  \operatorname{Isom}_+ \mathbb E^3$
    is a discrete, torsion-free, co-compact subgroup. By Bieberbach Theorem, the set of translations in $\Gamma$ forms a finite-index subgroup $T\cong\Z^3$. 
    % Recall that 
    % $$\operatorname{Isom}_+ \mathbb E^3 \subset \operatorname{Isom} \mathbb E^3 = \mathbb R^3 \rtimes O(3)$$ where
    % $\mathbb R^3$ represents the translations, and the orthogonal group  $O(3)$ acts on  $\mathbb E^3$ in the usual way.
    
    Each element $\gamma\in \operatorname{Isom} \mathbb E^3$ has the form $\gamma(v)=Av+b$, where $v\in \mathbb E^3$, $A\in O(3),$
    and $b\in \mathbb R^3$. We write $\gamma=:(A,b)$ and call $A$ the \textit{linear part} of $\gamma$.

    \begin{claim}
        Up to conjugating $\Gamma$ by a translation, there exists an integer $n$ such that 
        \[
            \Gamma\cdot 0\subset \frac 1n T.
        \]
        Here $0\in \R^3$ is the origin, and we have identified $T$ with a lattice in $\R^3$.
    \end{claim}
    \begin{proof}[Proof of the Claim]
        The linear part defines a group homomorphism $\Gamma\ra \Gamma/T\ra  O(3)$ of finite image group $P:=\{A \in O(3)\mid A\text{ is the linear part of an element of }\Gamma\}$. Suppose $P=\{A_1,\ldots,A_n\}$ and choose representatives $\gamma_1,\ldots,\gamma_n\in \Gamma$ such that $\gamma_i=(A_i,b_i)$. Define $c=\frac1n\sum_{i=1}^nb_i$.

        For any $\gamma=(B,b)\in \Gamma$, we have $(B,b)\circ (A_i,b_i)=(BA_i,b+Bb_i)$. There exists an index $j$ such that $BA_i=A_j$, and $b+Bb_i\subset b_j+T$. As $i$ runs through $1,\ldots,n$, $j$ also runs through $1,\ldots,n$. Hence
        \[
            \gamma\cdot c=\frac 1n\sum_{i=1}^n(b+Bb_i)\subset \frac 1n\sum_{j=1}^n(b_j+T)=c+\frac 1n T.
        \]
        Since $\gamma\in\Gamma$ is arbitrary, it follows that $\Gamma\cdot c\subset c+\frac1n T$. The Claim holds as we conjugate $\Gamma$ by the translation $(I,c)$.
    \end{proof}
    
     %We may assume that $\Gamma$ and $\mathbb Z^3$ have a common finite index subgroup.
    
    Therefore, we may assume without loss of generality that $\Gamma\cdot 0\subset \frac 1n T$.
    
    For each integer $k>0$, define $\mu_k: \R^3\to \R^3$ by $\mu_k(v)=(nk+1)v$. Clearly 
        $\lip \mu_k=(nk+1).$
    
    For any $\gamma=(A,b)\in \Gamma$, we have $b=\gamma\cdot 0\in \frac 1n T$, and
    \[
        \mu_k\circ \gamma=((nk+1)A,(nk+1)b),\quad \gamma\circ \mu_k=((nk+1)A,b).
    \]
    In particular, $\mu_k\circ \gamma-\gamma\circ \mu_k=nb\in T$. Hence $\mu_k\circ \gamma=\gamma'\circ \mu_k$ for some $\gamma'\in\Gamma$. Therefore,  
    $\mu_{k}$ induces  a map  $f_k : M\to M$ with 
    $\lip f_n=nk+1$. 
    
    On the other hand,  $\deg f_k =(nk+1)^3$. This is because $\mu_k$ preserves $T$ and induces a map on the 3-torus $\R^3/T$ whose degree is $(nk+1)^3$. 
    
    Therefore, $\deg f_k \geqslant (\lip f_k)^3$ for all $k$ and $\frac{\deg f_{k+1}}{{\deg f_{k}}}$ is bounded.
    By Lemma \ref{prepare},
     we have $\alpha(M)\geqslant 3$.

  %  Then $M$ is a torus bundle over the circle with periodic monodromy, or $M$ is a torus semi-bundle. Suppose $M$ is a torus bundle over the circle with monodromy $A$, where $A\in \operatorname{SL}(2,\Z)$ with $A^r=I$. Identify the torus $T$ with the 2-plane quotienting the integral lattice $\R^2/\Z^2$ and identify $M$ with the quotient 
 %   \[
  %      M=(T\times \R)/\Z
%    \]
 %   where the generator of $\Z$ is the translation
 %   \[
%        \tau:T\times \R\longrightarrow T\times \R,\quad ([x],t)\longmapsto ([Ax],t+1),\quad\forall x\in\R^2,\ t\in\R.
%    \]
%    The manifold 
%    \[
%        \tilde M=(T\times \R)/r\Z
%    \]
 %   is a natural $r$-fold covering of $M$, and is diffeomorphic to the three-dimensional torus $T^3$. For all positive integer $k$, consider the mapping
 %   \[
 %       f_k:M\longrightarrow M,\quad ([x],t)\longmapsto ([krx],krt),\quad\forall x\in \R^2,\ t\in\R.
 %   \]
 %   This mapping has a lifting
  %  \[
 %       \tilde f:\tilde M\longrightarrow \tilde M.
  %  \]
%    Identify $\tilde M$ with $T^3$ and given the flat metric. The mapping $\tilde f$ is $kr$-fold self covering in each $S^1$-factor, and has Lipschitz constant $kr$
 %   under the standard flat metric. So by Lemma \ref{LiftingUpperBoundsExponent} we have
  %  \[
 %       \lip f_k\leqslant C\lip \tilde f_k=Ckr
%    \]
  %  for some constant $C>0$. Since $\deg f_k=(kr)^3$, we have $P_M(Ckr)\geqslant (kr)^r$ for all $k\geqslant 1$ and hence $P_M(L)\gtrsim L^3$.
  
  In each Cases 1--3, we have $\alpha(M)\geqslant 3$. By the natural upper bound Corollary \ref{Topological invariance of flexible exponent} we have $\alpha(M)= 3$, finishing the proof of Proposition $\ref{ScalableGeometry}$.
 \end{proof}

\subsection{$\alpha(M)=8/3$ for Nil-geometry}

\begin{prop}\label{Nil}
Suppose $M$ is a closed orientable 3-manifold supporting Nil-geometry.
Then $$\alpha(M)= \frac83.$$
\end{prop}

\begin{proof} We first prove that  $\alpha(M)\geqslant \frac83.$

We refer to \cite[Section 12.5]{Mar} for a introduction of Nil geometry. The  \textit{Heisenberg group} consists of all matrices 
$$
\left(
                                                                                  \begin{array}{ccc}
                                                                                    1 & x & z \\
                                                                                    0 & 1 & y \\
                                                                                    0 & 0&  1 \\
                                                                                  \end{array}
                                                                                \right)$$
 with $x, y, z\in \mathbb R$.  When we consider the  Heisenberg group as a homogeneous space with certain metric, it is the Nil space. Identify Nil with $\mathbb R^3$ using the coordinate $(x, y, z)$,
then  the product operation is given by
 $$(x,y, z)\cdot (x', y', z')=(x+x', y+y', z+z'+xy').$$ 
 The volume form $\omega_{{\rm Nil}}$ is given by 
 $\omega_{{\rm Nil}} =dx\wedge dy \wedge dz$, the vertical line field $\tilde {\mathcal V}$ is given by $e_3=\partial _z$, and the horizontal contact field $\tilde {\mathcal D}$ is given by the span of $e_1=\partial_x$
 and $e_2=\partial_y+x\partial _z$. Both fields are invariant under $\operatorname{Isom}_+ {\rm Nil}$.

 Recall that any orientation-preserving isometry of Nil preserves the $z$-direction and induces an isometry on $\E^2$. There is a short exact sequence:
 \[
    1\ra \R\ra \operatorname{Isom}_+({\rm Nil})\ra \operatorname{Isom}(\E^2)\ra 1
 \]
 where $\R$ is the center of $\operatorname{Isom}_+({\rm Nil})$ consisting of translations on the $z$-direction.

% By assumption $M=Nil/\Gamma$, where $\Gamma$ is a discrete, torsion-free subgroup of $\operatorname{Isom}_+ Nil$.
 
 Let $T_k: {\rm Nil} \to {\rm Nil} $ be the map given by $$T_k(x,y,z)=(kx,ky,k^2z).$$ One can verify that $T_k$ is well-defined homomorphism.
 
 \begin{lem}\label{N1} For each closed orientable {\rm Nil} $3$-manifold $M$, there exists a discrete faithful representation $\Lambda:\pi_1(M)\ra \operatorname{Isom}_+({\rm Nil})$ giving the {\rm Nil}-geometry of $M$, and an integer $k>1$ such that $T_{k}: {\rm Nil}\to {\rm Nil}$ descends to a map $\bar T_{k}: M\to M$.
 \end{lem}
 
 \begin{proof} A Nil 3-manifold $M$ is a circle bundle of nontrivial Euler number over an closed Euclidean orbifold $\mathcal{O}$. 
 
 \begin{claimA}
 Let $M\to \mathcal{O}$ be given as above. There is an Euclidean metric of $\mathcal{O}$ such that the map $$A_k:\E^2\ra \E^2,\quad (x,y)\mapsto (kx,ky)$$ descends to an orbifold covering $\overline A_k:\mathcal{O}\ra \mathcal{O}$. Moreover, $\overline A_k$ is covered by a self-covering $M\to M$.
 \end{claimA}
 
 \begin{proof}[{Proof of Claim A}] The proof is essentially contained in \cite[Section 4]{SWWZ}.
 
 Let $F(q_1, ... , q_n)$ be the orbifold which is surface $F$ with singular points of index $q_1, ..., q_n$. Let $q$ be the least 
 common multiple of those $q_i$. There are seven Euclidean orbifolds associated to Nil 3-manifolds:
 $T, K, S^2(2,2,2,2),  \RP^2(2,2), S^2(3,3,3), S^2(2, 4,4 ), S^2(2,3,6)$, where $T, K, S^2, RP^2$ are torus, Klein bottle, 2-sphere, and projective plane respectively.
 
 The map $\overline A_k:  \mathcal{O}\to \mathcal{O}$ is well-defined for any integer $k>1$ when $\mathcal{O}$ is in the form $S^2(q_1, q_2, q_3)$, which is verified in \cite[Proposition 4.3]{SWWZ}.
 Similarly and may be easier, one can verify $\overline A_k:  \mathcal{O}\to \mathcal{O}$ is well-defined for any integer $k>1$ when $\mathcal{O}$ is $T$ or $S^2(2,2,2,2)$,
 and for any odd integers $k>1$ when $\mathcal{O}$ is $K$ or $\RP^2(2,2)$.  For the verification of the case $RP^2(2,2)$, the fundamental region and group presentation of this wallpaper group of type III.17 given in \cite[p178, p182]{NS} are helpful.

 %For the verification of the case $\RP^2(2,2)$, the fundamental region and group presentation  given in \cite[p178]{NS} are helpful.
 
 In the end of the proof of \cite[Theorem 4.4]{SWWZ}, it is proved that each orbifold covering $\mathcal{O}\to \mathcal{O}$ of degree $l$ can be induced from a covering $M\to M$ if $l \equiv 1 \mod q$. 
 
 Now let $k=5$, then $\bar A_k: \mathcal{O}\to \mathcal{O}$ exists  and also $l=k^2=25 \equiv 1 \mod q$ for each $\mathcal{O}$. We finish the proof. 
 \end{proof}

%Equivalently, the conclusion of Claim A says that the transformation $A_k\circ \lambda(g)\circ A_k\inv$ of $\E^2$ still lies in $\lambda(\pi_1(O))$ for all $k\in \{k_i\}$ and $g\in \pi_1(O)$. 
%The Euclidean orbifold groups are known as the wallpaper groups and are well-studied. The proof of Claim A is direct from looking at the tiling of $\mathbb E^2$ by the action of $\pi_1(O)$ on a fundamental region.  

 %Moreover, we have the following commutative diagram.  Therefore, any Nil-geometry of $M$ determines an Euclidean geometry of $O$. The conversely also holds true.

 We also need the following Claim B which is essentially \cite[Proposition 12.5.11]{Mar}.

 \begin{claimB}
     Any discrete faithful representation $\lambda:\pi_1(\mathcal{O})\ra \operatorname{Isom}(\E^2)$ of $\mathcal{O}$ can be lifted to a representation $\Lambda:\pi_1(M)\ra \operatorname{Isom}_+({\rm Nil})$, such that the following diagram commutes:
 \[
 \begin{tikzcd}
\pi_1(S^1) \arrow[r] \arrow[d] & \pi_1(M) \arrow[r, "\pi"] \arrow[d, "\Lambda"] & \pi_1(\mathcal{O}) \arrow[d, "\lambda"]    \\
\mathbb R \arrow[r]            & \operatorname{Isom}_+({\rm Nil}) \arrow[r]    & \operatorname{Isom}(\mathbb E^2)
\end{tikzcd}
 \]
 Moreover, any such lifting $\Lambda$ is discrete and faithful. \end{claimB}

  An element $g\in\pi_1(\mathcal{O})$ is called \textit{type I} if $g$ is orientation-preserving. Otherwise, $g$ is called \textit{type II}. A function $t:\pi_1(\mathcal{O})\ra \R$ is called a \textit{skew homomorphism} if for any $g,h\in\pi_1(M)$, we have 
 \[t(gh)=
    \begin{cases}
        t(g)+t(h), & h \text{ is type I},\\
        -t(g)+t(h), & h \text{ is type II}.
    \end{cases}
 \]
 Let $\mathcal D$ be the set of all skew homomorphisms, then $\mathcal D$ is a vector space over $\R$ whose rank is one plus the first Betti-number of the orientation-double cover of $\mathcal{O}$. %Note that $\mathcal D$ is a real vector space of dimension not greater than $2$.
 
 If we fix $\lambda:\pi_1(\mathcal O)\ra \operatorname{Isom}(\E^2)$ and let $\mathcal M$ be the set of all liftings $\Lambda:\pi_1(M)\ra \operatorname{Isom}_+({\rm Nil})$. Then $\mathcal M$ is affine over $\mathcal D$. Namely, choose a base point $\Lambda_0\in \mathcal M$, then there is an bijection $\Phi:\mathcal D\ra \mathcal M$ given by $\Phi(t)=\Lambda_0+ t$ for all $t\in \mathcal D$, where $$(\Lambda_0+ t)(g)=\Lambda_0(g)\cdot L_{t(\pi(g))},\quad \forall g\in \pi_1(M).$$
 %Here $\{e^s\mid s\in\R\}$ is the center of $\operatorname{Isom}_+(Nil)$, namely,  
 Here for any $s\in \R$, $L_s$ is the isometry of {\rm Nil} sending $(x,y,z)\in {\rm Nil}$ to $(x,y,z+s)\in {\rm Nil}$. It is easy to verify that $\Lambda_0+ t$ is in $\mathcal M$. Conversely, the proof of \cite[Proposition 12.5.11]{Mar} shows that any element of $\mathcal M$ arises in this way.

%, in which it is proved that any Euclidean holonomy representation $\lambda:\pi_1(O)\ra \operatorname{Isom}(\mathbb E^2)$ can be lifted to a representation $\Lambda:\pi_1(M)\ra \operatorname{Isom}_+(Nil)$. It turns out that $\Lambda$ is discrete and faithful. When $O$ has genus 0, such lifting is unique.

Fix the discrete faithful representation $\lambda:\pi_1(\mathcal{O})\ra \operatorname{Isom}(\mathbb E^2)$ and the map $A_k:\mathbb E^2\ra \mathbb E^2$ given by Claim A.  %Endow $M$ with the Nil-geometry given by a lifting $\Lambda_s:\pi_1(M)\ra \operatorname{Isom}^+(Nil)$ from Claim B (if the underlying space of $O$ is $S^2$, we choose the unique $\Lambda$. Otherwise, $s$ will be specified in a moment). It is easy to verify that $T_k\circ \Lambda_s\circ T_k\inv$ is also a discrete faithful representation of $\pi_1(M)$. We need to show that $T_k\circ \Lambda_s(\pi_1(M))\circ T_k\inv\subset  \Lambda_s(\pi_1(M))$.
 %Choose the origin of $\mathbb E^2$ as its base point and this also determines a base point of $O$.
 Let $\iota:\pi_1(\mathcal{O})\ra \pi_1(\mathcal{O})$ be the homomorphism induced by the self-covering $\overline A_k$.

 \begin{claimC}
     The linear map
 \[
    \mathcal D\ra \mathcal D,\quad t\mapsto \frac 1{k^2}t\circ \iota
 \]
 has no eigenvalue $1$.
 \end{claimC}
\begin{proof}[{Proof of Claim C}]
    Suppose there is $t\in \mathcal D$ such that $t=\frac 1{k^2}t\circ \iota$. For any $g\in \pi_1(\mathcal{O})$. Since $\mathcal{O}$ is an Euclidean orbifold, it has a finite index subgroup $\Z^2$, and $\lambda(\Z^2)\subset \operatorname{Isom}(\E^2)$ is generated by two translations. Therefore, for any $g\in \Z^2\subset \pi_1(\mathcal{O})$ we have $\iota(g)=g^k$, and 
    \[
        t(g)=\frac 1{k^2}t(g^k)=\frac 1k t(g)
    \]
    which implies $t(g)=0$. Here the second equality follows from $t(g^k)=kt(g)$, which is because $g$ is of type I and $t$ is an element in $\mathcal{D}$. Since $\Z^2\subset \pi_1(\mathcal{O})$ is a finite-index subgroup, let $h$ be any type I element of $ \pi_1(\mathcal{O})$, then $h^m\in \Z^2$ for some $m\in \Z$ and hence $t(h)=0$ (similarly this follows from $t(h^m)=mt(h)$). Let $g_1,g_2$ be any type II element of $\pi_1(\mathcal{O})$, then $g_1=hg_2$ for a type I element $h\in \pi_1(\mathcal{O})$ and $$t(g_1)=t(hg_2)=-t(h)+t(g_2)=t(g_2).$$
    This shows that $t$ is constant on type II elements. In particular, $t(g_1)=\frac 1{k^2}t(\iota (g_1))=t(g_1)$ implies that $t(g_1)=0$. Here the second equality follows from both $\iota (g_1)$ and $g_1$ are of type II. In conclusion, $t\equiv 0$ and this completes the proof.
\end{proof}

 Let $\iota:\pi_1(\mathcal{O})\ra \pi_1(\mathcal{O})$ be the  homomorphism induced by $\overline A_k: \mathcal{O}\to \mathcal{O}$ in Claim A, then it can be lifted to a homomorphism $\tau:\pi_1(M)\ra \pi_1(M)$. Recall that $T_k: {\rm Nil} \to {\rm Nil} $ is the map given by $T_k(x,y,z)=(kx,ky,k^2z)$.
 \begin{claimD}
     There exists a unique $\rho\in \mathcal M$ such that $T_k\cdot \rho \cdot T_k\inv =\rho\circ \tau$.
 \end{claimD}
 \begin{proof}[{Proof of Claim D}]
     Define $F:\mathcal M\ra \mathcal M$ by $\rho\mapsto T_k\inv \cdot \rho\circ \tau \cdot T_k$. Write $\rho=\Lambda_0+t$, we have 
     \begin{align*}
         F(\Lambda_0+t)(g)&=T_k\inv \cdot (\Lambda_0(\tau(g)) \cdot L_{t\circ\pi\circ \tau(g)} )\cdot T_k\\&
         =(T_k^{-1}\cdot \Lambda_0(\tau(g))\cdot T_k)\cdot (T_k^{-1}L_{t\circ\iota\circ\pi(g)}T_k)\\&=F(\Lambda_0)(g)\cdot (T_k\inv L_{t\circ \iota\circ \pi(g)}T_k).
     \end{align*}
     Note that for each $s\in\mathbb{R}$,
     $$T_k^{-1}L_sT_k=L_{\frac{s}{k^2}},$$
     Therefore,
     $$T_k\inv L_{t\circ \iota\circ \pi(g)}T_k=L_{\frac{t\circ \iota\circ \pi}{k^2}(g)},
        $$
        we obtain         $F(\Lambda_0+t)=F(\Lambda_0)+\frac1{k^2}t\circ \iota$.
     This is an affine map on $\mathcal M$ whose linear term is $t\mapsto \frac1{k^2}t\circ \iota$. By Claim C, the linear term has no eigenvalue 1, it follows that $F$ has a unique fixed point.
 \end{proof}
 % \begin{claimD}
 %     There exists a unique $\rho\in \mathcal M$ such that $T_k\cdot \rho \cdot T_k\inv =\rho\circ \tau$.
 % \end{claimD}
 % \begin{proof}[{Proof of Claim D}]
 %     Define $F:\mathcal M\ra \mathcal M$ by $\rho\mapsto T_k\inv \cdot \rho\circ \tau \cdot T_k$. Write $\rho=\Lambda_0+t$, we have $$F(\Lambda_0+t)(g)=T_k\inv \cdot (\Lambda_0(\tau(g)) \cdot e^{t\circ\pi\circ \tau(g)} )\cdot T_k=$$
 %    $$(T_k^{-1}(\Lambda_0(\tau(g))T_k)(T_k^{-1}e^{t\circ\iota\pi(g)}T_k)=F(\Lambda_0)(g)\cdot T_k\inv e^{t\circ \iota\circ \pi(g)}T_k.$$
 %     Note that for each $s\in\mathbb{R}$, there is
 %     $$T_k^{-1}e^sT_k=e^{\frac{s}{k^2}},$$
 %     Therefore,
 %     $$T_k\inv e^{t\circ \iota\circ \pi(g)}T_k=e^{\frac{t\circ \iota\circ \pi}{k^2}(g)},
 %        $$
 %        we obtain         $F(\Lambda_0+t)=F(\Lambda_0)+\frac1{k^2}t\circ \iota$.
 %     This is an affine map on $\mathcal M$ whose linear term is $t\mapsto \frac1{k^2}t\circ \iota$. By Claim C, the linear term has no eigenvalue 1, it follows that $F$ has a unique fixed point.
 % \end{proof}

 Equip $M$ with the Nil geometry given by $\rho\in \mathcal M$, then $$T_k\cdot \rho(\pi_1(M))\cdot T_k\inv=\rho\circ\tau(\pi_1(M))\subset \rho(\pi_1(M)),$$ therefore $T_k$ descends to a map on $M$ and the proof of Lemma \ref{N1} is complete.
 \end{proof}
 
By Lemma \ref{N1}, fix the Nil geometry of $M$ and the integer $k>1$ such that the map $T_k: {\rm Nil} \to {\rm Nil}$ descends to $\bar T_k: M\to M$. Then
 \begin{align*}
     T_k^* (\omega_{Nil}) &=T_k^* (dx\wedge dy \wedge dz)%= T_k^* (dx) \wedge T_k^* (dy) \wedge T_k^* (dz)
     =d(kx)\wedge d(ky) \wedge d(k^2z)\\&=k^4dx\wedge dy \wedge dz=k^4\omega_{Nil}.
 \end{align*}

 Then it follows that  $$\bar T_k^* (\omega_M)= k^4 \omega_M,$$
 and we  conclude that $\deg \bar T_k=k^4$.
 For each integer $n>1$, let $$\bar S_n=\bar T_k^n\circ\bar \varphi\circ \bar T_k^n:M\to M,$$  where $\bar \varphi: M\to M$ is the Legendrian map given by Theorem \ref{thm: legendrian}. Recall that $\varphi$ is homotopic to the identity. 
%$T_k[x,y,z]=[kx,ky,k^2z]$, $M=\R^3/\sim$, $\sim$ is the equivalent class generated by $$(x,y,z)\sim(x,y+1,z)\sim(x+1,y,z+y).$$ 
Then we have 
$$\deg \bar S_n= \deg(\bar T_k)^n\deg(\bar \varphi) \deg (\bar T_k)^n=k^{8n}.$$

Under the basis $e_1=\partial_x$, $e_2=\partial_y+x\partial _z$, $e_3=\partial_z$ of $T_*{\rm Nil}$, it is easy to check that

$$ (T_k)_{*}\left(
                                                                                  \begin{array}{c}
                                                                                    e_1 \\
                                                                                    e_2 \\
                                                                                    e_3 \\
                                                                                  \end{array}
                                                                                \right)=\left(
                                                                                  \begin{array}{ccc}
                                                                                    k&0&0 \\
                                                                                    0&k&0\\
                                                                                    0&0&k^2 \\
                                                                                  \end{array}
                                                                                \right)\left(
                                                                                  \begin{array}{c}
                                                                                    e_1 \\
                                                                                    e_2 \\
                                                                                    e_3 \\
                                                                                  \end{array}
                                                                                \right)
.$$

Let $\varphi: {\rm Nil}\to {\rm Nil} $ be a lift of $\bar \varphi$. Since $\varphi$ is Legendrian, $\varphi(e_3)$ is a linear combination of $e_1$ and $e_2$.
So we have  $$ \varphi_*\left(
    \begin{array}{c}
  e_1 \\ e_2 \\    e_3 \\     \end{array}
    \right)=\left(     \begin{array}{ccc}
       a_{11}&a_{12}&a_{13} \\    a_{21}&a_{22}&a_{23} \\ a_{31}&a_{32}&0  \\  \end{array}
     \right)\left(  \begin{array}{c}
   e_1 \\ e_2 \\ e_3 \\  \end{array}   \right)
$$
Note that $$S_n= T_k^n\circ \varphi\circ  T_k^n:{\rm Nil} \to {\rm Nil}$$ 
descends to $\bar S_n: M\to M$ and
$$S_{n*}\!\left(  \begin{array}{c} e_1 \\
  e_2 \\ e_3 \\  \end{array}   \right)\!=\!T_{k*}^n\varphi_*T_{k*}^n\!\left(  \begin{array}{c} e_1 \\ e_2 \\ e_3 \\  \end{array}   \right)\!=\!\left(  \begin{array}{ccc} k^{2n}a_{11}&k^{2n}a_{12}&k^{3n}a_{13} \\ k^{2n}a_{21}&k^{2n}a_{22}&k^{3n}a_{23} \\ k^{3n}a_{31}&k^{3n}a_{32}&0 \\\end{array}   \right)\!\left(  \begin{array}{c} e_1 \\ e_2 \\ e_3 \\  \end{array}   \right). $$
  It is clear that  
$\mathrm{Lip}(S_n)\leqslant  C|k|^{3n}$ for some constant $C$ independent of $n$. 
 Hence $\mathrm{Lip}(\bar S_k)=\mathrm{Lip}(S_n)\leqslant  C|k|^{3n}$. Then  $\deg \bar S_n\geqslant C^{-1}|\mathrm{Lip}(\bar S_n)|^{8/3}$ for all $n>1$. Since $\deg \bar S_{n+1}/\deg \bar S_{n}=k$ is bounded, by Lemma \ref{prepare}
we have $$\alpha(M)\geqslant 8/3.$$

 \begin{exmp}\label{exp: Explicit example} For the simplest Nil 3-manifold $M_1=\mathrm{Nil}/\mathrm{Nil}_{\mathbb Z}$, where ${\rm Nil}_\mathbb{Z}$ %$$\mathrm{Nil}_{\mathbb Z}=\left\{\left({\begin{array}{ccc}
       %1 &a &c \\    0 & 1 &b \\ 0 & 0 & 1  \\   \end{array}}\right):a, b,c\in\mathbb{Z}\right\}$$ 
       is the discrete Heisenberg group consisting of integer entries matrices in $\text{Nil}$ and $M_1$ is homeomorphic to the $S^1$-bundle over the torus $T$
 with Euler class $e=1$, we can  explicitly construct a Legendrian map $f: M_1\to M_1$.
 %This is equivalent to 
 
 To do this, we will construct a $\mathrm{Nil}_{\mathbb Z}$-equivariant Legendrian map  $\varphi:\mathrm{Nil}\to\mathrm{Nil}$. Write $\varphi(x,y,z)=(X,Y,Z)$ where $X=X(x,y,z),\ Y=Y(x,y,z),\ Z=Z(x,y,z)$ are real-valued functions on $\mathbb{R}^3$. The Seifert fibration on $\text{Nil}$ is induced by $\frac{\partial}{\partial z}$, and the contact structure on $\text{Nil}$ is induced by $dz-xdy$. Then $\varphi$ is Legendrian if and only if 
 \begin{align*}
     \frac{\partial Z}{\partial z}-X\frac{\partial Y}{\partial z}=0.
 \end{align*}
 If we do the substitution $(X,Y,Z)=(x+\tilde x, y+\tilde y, z+x\tilde y+\tilde z)$, where $\tilde x=\tilde x(x,y,z),$ $\tilde y=\tilde y(x,y,z)$, $\tilde z=\tilde z(x,y,z)$, then one can check that
$\phi$ is Legendrian if and only if $(\tilde x, \tilde y, \tilde z)$ satisfy %the following partial differential equation.%$$\phi {\rm\,\, is\,\,Legendrian\quad if\,\, and\,\, only\,\, if\quad} \tilde x, \tilde y, \tilde z\,\, {\rm satisfies\,\, the \,\,following \,\,PDE}:$$
\begin{equation}\label{eq: Legendrian PDE} 1+\frac{\partial \tilde z}{\partial z}-\tilde x\frac{\partial \tilde y}{\partial z}=0.\end{equation} One can check that $\phi$ is ${\rm Nil}_\mathbb{Z}$-equivariant if and only if each of $\tilde x, \tilde y,\tilde z$ is ${\rm Nil}_\mathbb{Z}$-invariant.
%$$\phi {\rm\,\, is\,\,} Nil_{\mathbb Z}\text{-equivariant\quad  if and only if\quad } \tilde x, \tilde y, \tilde z\  {\rm are\ } Nil_{\mathbb Z}\text{-invariant}.$$
%Therefore we need only to find 
An explicit $\mathrm{Nil}_{\mathbb Z}$-invariant solution of (\ref{eq: Legendrian PDE}) is given by%. Let
%Such a solution can be explicitly given by
%Consider the functions of the following form
\begin{align*}
    &\tilde x=\sum_{n\in\Z}(\psi_1(x+n)\cos(2\pi(z+ny))+\psi_2(x+n)\cos(4\pi(z+ny))),\\&
    \tilde y=\sum_{n\in\Z}(2\psi_1(x+n)\sin(2\pi(z+ny))+\psi_2(x+n)\sin(4\pi(z+ny))),\\&
    \tilde z=\sum_{n\in\Z}(\psi_1^2(x+n)\sin(4\pi(z+ny))/2+\psi_2^2(x+n)\sin(8\pi(z+ny))/4)\\&
    \quad+\sum_{m,n\in\mathbb{Z}}\psi_1(x+m)\psi_2(x+n)[{\rm sin}(2\pi(z+(2n-m)y))+{\rm sin}(2\pi(3z+(2n+m)y))/3],
\end{align*}
%where the functions
%$\psi_1,\psi_2\in C_{c}^{\infty}(\R)$ are supposed to satisfy %such that 
where $\psi_1$, $\psi_2\in C_c^\infty(\mathbb{R})$ satisfy the following conditions:
\begin{enumerate}
\item $\operatorname{supp}\psi_1\subset[-1/2,1/2] $, $\operatorname{supp}\psi_2\subset[0,1] $, and 

%and 

\item $ 2\pi \sum_{n\in\Z}(\psi_1^2(x+n)+\psi_2^2(x+n))=1$.
\end{enumerate}
A typical example of $(\psi_1, \psi_2)$ satisfying the above conditions is given by
$$\psi_1(x)=\begin{cases}(2\pi)^{-1/2}(1+{\rm e}^{\frac{1}{1/2-|x|}-\frac{1}{|x|}})^{-1/2},& \text{ when } 0<|x|<1/2,\\
0,& \text{ when } |x|\geqslant 1/2, \\
(2\pi)^{-1/2}, & \text{ when }x=0,\end{cases}
$$
and $$\psi_2(x)=\psi_1(x-1/2).$$
%Then $\psi_1,\psi_2\in C^\infty_c(\mathbb{R})$ and ${\rm supp}\text{ }\psi_1\subset [-1/2,1/2]$, ${\rm supp}\text{ }\psi_2\in [0,1]$. We need only to verify (ii) on the interval $[0,1/2]$. For $x\in (0,1/2)$, 
%\begin{align*}
 %   2\pi\sum_{n\in\mathbb{Z}}(\psi_1^2(x+n)+\psi_2^2(x+n))=2\pi(\psi_1^2(x)+\psi_2^2(x))=(1+\text{e}^{\frac{1}{1/2-x}-\frac{1}{x}})^{-1}+(1+\text{e}^{\frac{1}{x}-\frac{1}{1/2-x}})^{-1}=1.
%\end{align*}
%Note that all cross-terms are zero in $\tilde x^2$ due to the condtion (i). So
%$$4\pi \tilde x^2=4\pi\sum_{n\in\mathbb{Z}}(\psi_1^2(x+n){\rm cos}^2(2\pi(z+ny))+\psi_2^2(x+n){\rm cos}^2(4\pi(z+ny)))$$
%$$=2\pi(\sum_{n\in\mathbb{Z}}\psi_1^2(x+n){\rm cos}(4\pi(z+ny))+\psi_2^2(x+n){\rm cos}(8\pi(z+ny)))+1=\frac{\partial \tilde z}{\partial z}+1.$$
%The second equality uses the condition (ii).
%Now we have explicitly constructed a 
%The above $(\tilde x, \tilde y,\tilde z)$ gives a ${\rm Nil}_\mathbb{Z}$-invariant solution to the equation $(*)$. 
The corresponding $\varphi$ is a ${\rm Nil}_\mathbb{Z}$-equivariant Legendrian map on ${\rm Nil}$, which descends to a Legendrian map $f$ on $M_1={\rm Nil}/{\rm Nil}_\mathbb{Z}$. 
\end{exmp}

Next  we prove $\alpha(M)\leqslant \frac83.$ 

 By Lemma \ref{lift1}, each non-zero degree map $f: M\to M$ can be lifted to $\tilde f: \tilde M\to \tilde M$ 
where $\tilde M$ is an $S^1$-bundle over the torus $T$ covering $M$.  By Lemma \ref{LiftingUpperBoundsExponent}  $\alpha(M)\leqslant \alpha(\tilde M)$,
we may assume without loss of generality that $M$ is a $S^1$-bundle over the torus.

%\textcolor{red}{$\alpha\leqslant\frac83$ TO BE ADDED}.

By \cite{Wang}, any nonzero degree map $\tilde f:M\to M$ induces injection on fundamental group, and is homotopic to a fiber-preserving covering $g$ by \cite{Wald}. 
Thus we have the following commutative diagram: 
\[ \begin{CD}
S^1 @>  >> M @> q>>T\\
@V g|_{S^1}VV @V gVV  @V \bar{g}VV @.\\
S^1 @>  >> M @> q>>T\\
\end{CD}  \]

By \cite{Wang}, $${\rm deg}(g|_{S^1})={\rm deg}(\bar{g}).$$
Since ${\rm deg}(g)={\rm deg}(g|_{S^1})\cdot{\rm deg}(\bar{g})$, we have
${\rm deg}(g)={\rm deg}(\bar{g})^2%=\det(\bar g_{\# 1})^2=\det(g_{\# 1})^2
$.
There is a commutative diagram on the first homology groups:
\[ \begin{CD}
H_1(M;\mathbb{Q}) @> q_{\# 1}>>H_1(T;\mathbb{Q})\\
@V g_{\# 1}VV  @V \bar{g}_{\# 1}VV @.\\
 H_1(M;\mathbb{Q}) @> q_{\# 1}>>H_1(T;\mathbb{Q})\\
\end{CD}  \]
where $q_{\# 1}: H_1(M;\mathbb{Q})\to H_1(T;\mathbb{Q})$ is an isomorphism.
It follows that $\det(\bar g_{\# 1})=\det(g_{\#1})$ and hence
\[
    \deg(g)=\deg(\bar g)^2=\det(\bar g_{\#1})^2=\det(g_{\# 1})^2.
\]

Since $f$ is homotopic to $g$, by the homotopy invariance of homology, we have
$${\rm deg}(f)=\det(f_{\# 1})^2.$$
By Poincar\'e duality, we have
$${\rm deg}(f)^2=\det(f_{\# 1})\cdot \det(f_{\# 2})$$
where $f_{\#i}:H_i(M;\Q)\ra H_i(M;\Q)$ is the induced map on homology.
Combining the above two equations 
we have
$${\rm deg}(f)^3=\frac{{\rm deg}(f)^4}{{\rm deg}(f)}=\frac{\det(f_{\# 1})^2\cdot \det(f_{\# 2})^2}{\det(f_{\# 1})^2}=\det(f_{\# 2})^2.$$
Note that $b_2(M)=2$. By Proposition \ref{HomologyControlsFlexibleExponent} (1), we have
$$\det(f_{\# 2})\leqslant C\cdot ({\rm Lip}\,f)^{2b_2(N)}=C({\rm Lip}\,f)^4.$$
So
$${\rm deg}(f)=\det(f_{\# 2})^{\frac23}\leqslant C^{\frac23}({\rm Lip}\,f)^4)^{\frac 23}= C^{\frac23}({\rm Lip}\,f)^{\frac83}.$$
Thus $$\alpha(M)\leqslant \frac83.$$
\end{proof}

\begin{rem}
    Given a group $G$ and a finite generating set $S$, the \textit{word length} $l_S(g)$ of $g\in G$ is the smallest length of a word in $S$ that represents $g$. The \textit{algebraic Lipschitz constant} $\lip_S(f)$ of a homomorphism $f:G\ra G$ is defined as 
    \[
        \lip_S(f):=\sup_{g\in G}\frac{l_S(f(g))}{l_S(g)}.
    \]
    Note that this is the Lipschitz constant of $f$ as a map on the vertex set of the Cayley graph $C_S$ associated to $S$.
    
    Suppose $G$ is the fundamental group of a closed orientable aspherical manifold $X$. Then define the \textit{algebraic flexible exponent} $\widetilde \alpha(X)$ %to be the infimum of $\alpha\geqslant0$ such that there exists a generating set $S$ and $C>0$, such that $\lip_S(f)\leqslant C(\deg f)^\alpha$ for all $f:G\ra G$.
    using Definition \ref{Definition of the flexible exponents} and the algebraic Lipschitz constant defined above.

    For the Nil manifold $M_1$ in the previous Example \ref{exp: Explicit example}, its algebraic flexible exponent is even greater than $\dim M_1=3$. Indeed, we have 
    \[
        \pi_1(M_1)=\langle a,b: [a,b]=h,\ ah=ha,\ bh=hb\rangle.
    \]
    Choose the natural generating set $S=\{a,b\}$. Define $f_n:\pi_1(M_1)\to \pi_1(M_1)$ by $f_n(a)=a^n$ and $f_n(b)=b^n$. Then $\lip_S(f_n)\leqslant n$. However, we have $\det (f_n)_*=n^2$ where $(f_n)_*$ is the induced map on $H_1(M_1;\Q)$. Therefore $\deg f_n=n^4$. It follows that $\widetilde \alpha(M_1)\geqslant 4$.
\end{rem}

%$\section{Flexible exponents of prime 3-manifolds: lower bounds}

%\subsection{Sol-geometry}

\subsection{$\alpha(M)=2$ for Sol-geometry}

\begin{prop}\label{Sol}
    If $M$ is an orientable {\rm Sol} $3$-manifold, then $\alpha(M)=2$.
\end{prop}

The closed orientable Sol-manifolds consist of two types: the torus bundle over the circle $M_\psi$ with gluing map $\psi\in {\rm SL}_2(\mathbb{Z})$
such that $|\text{tr}(\psi)|> 2$, and the semi-torus bundles $N_\phi$,  which are doubly covered by a torus bundle.

We first prove that $\alpha(M)\geqslant 2$.  This follows from  the following 

\begin{lem}
    If $M$ is a torus bundle or a torus semi-bundle, then $\alpha (M)\geqslant 2$.
\end{lem}

\begin{proof} We only prove the lemma for torus semi-bundle, and the proof for torus bundle is similar and more direct.

Let $K$ be  the Klein bottle and
$N=K\widetilde{\times}I$ be the twisted $I$-bundle over $K$. A {\it
torus semi-bundle} $N_\phi=N\bigcup_\phi N$ is obtained by gluing
two copies of $N$ along their  torus boundary $\partial N$ via a
diffeomorphism $\phi$. Note $N_\phi$ is foliated by tori parallel to
$\partial N$ with a Klein bottle at the core of each copy of $N$.

Identify $S^1$ with $\R/2\pi\Z$ and let $(x,y,t)$ be the coordinate of $S^1\times S^1\times [-1, 1]$. Then
$N=S^1\times S^1\times [-1,1]/\tau$, where $\tau$ is an orientation-preserving involution such that  $\tau(x,y, t)=(x+\pi,-y,-t)$, and
we have the double covering $p: S^1\times S^1\times [-1,1]\to N$.  Let
$C_x$ and $C_y$ be the two circles on $S^1\times S^1\times \{1\}$
defined by $y$ to be constant and $x$ to be constant.
Denote by $l_0=p(C_x)$ and
$l_\infty= p(C_y)$ on $\partial N$. 
Once we choose 
canonical coordinates on each $\partial N$,  $\phi$  is
identified with an element $\left(\begin{array}{cc}
                                                    a & b \\
                                                    c & d \\
                                                  \end{array}
                                                \right)$ of ${\rm GL}_2(\mathbb{Z})$ given by
$\phi(l_0,l_\infty)= (l_0,l_\infty) \left(\begin{array}{cc}
                                                    a & b \\
                                                    c & d \\
                                                  \end{array}
                                                \right)$.
                                                
                                                Let $\iota_k: S^1\times S^1\times [-1,1]\to S^1\times S^1\times [-1,1]$ be a map given by
                                                $(x, y, t)\mapsto (kx, ky, t)$. Then  $\iota_k\circ \tau=\tau\circ\iota_k$ for each odd integer $k>0$,
                                                therefore induces an map $\bar \iota_k: N\to N$, and furthermore, 
                                                such two copies of $(N, \bar \iota_k)$ provide a smooth map $f_{k, \phi} : N_\phi\to N_\phi$
                                                which keeps each Euclidean fiber invariant and the restriction on each fiber is degree $k^2$, therefore $f_{k, \phi}$ is of degree $k^2$.
                                                
Now we  put a Riemannian metric on $N_\phi$ such that $\lip (f_{k, \phi}) =k$ (note the Euclidean metric on each fiber of $N_\phi$
does not matched to give a Riemannian metric on $N_\phi$).  To do that, we present $N_\phi$ as
$$N_\phi=N_-\cup_{\phi_-}(T\times [-1, 1])\cup _{\phi_+} N_+,$$ 
 where $N_-$ and $N_+$,  two copies of $N=T\times [-1, 1]/\tau$, has the induced $\mathbb E^3$ metrics, and 
where $\phi_{\pm} : T\times \{\pm 1\}\to \partial N_{\pm}$
are linear gluing maps such that $\phi_-^{-1}\circ \phi_+=\phi$.

 Suppose $T\times \{\pm 1\}$ has the induced metric  $$g_\pm = A_\pm dx^2+ B_\pm dxdy +C_\pm dy^2.$$
 
 Let  $$g_t = A(t) dx^2+ B(t) dxdy +C(t) dy^2,$$ $t\in[-1, 1]$,  be a path of Euclidean metrics connecting  $g_\pm$. %(for instance, one can take $(A(t), B(t), C(t))=[(1-t)(A_-, B_-, C_-)+(1+t)(A_+, B_+, C_+)]/2.$)
 Then $\{g_t\}_{t\in[-1,1]}$ gives a fiberwise Euclidean structure on $T\times [-1, 1]$. %, i. e., on each fiber $T\times\{t\}$ there is a Euclidean structure.
 Let $g=g_t+dt^2=A(t)dx^2+B(t)dxdy+C(t)dy^2+dt^2$. Then $g$ is a metric on $T\times [-1, 1]$. Together with the given metric on $N_\pm$, we have a 
 Riemannian metric on $N_\phi$, still denoted as $g$,  and with this  metric one can  verify that $\lip (f_{k, \phi}) =k$.
\end{proof}

We then prove that $\alpha(M)\leqslant 2$. Let $N$ be a torus semi-bundle and let $p: M\rightarrow N$ be the unique fiber-preserving double covering of $N$ by a torus bundle $M$. By Theorem 2.9 of \cite{SWW},
any non-zero degree map $f: N\to N$ can be lifted to a map $\tilde f: M \to M$. 
By Lemma \ref{LiftingUpperBoundsExponent}, $\alpha(N)\leqslant \alpha(M)$. So we need only to prove the following

\begin{lem}\label{SolTorusBundleFlexibleExponentUpperBound}
    Let $M$ be the torus bundle over $S^1$ with monodromy $A\in {\rm SL}_2(\Z)$
with the $|\operatorname{tr} A|> 2$. Then $\alpha(M)\leqslant 2$.
\end{lem}

\begin{proof}
    It is known from Mayer-Vietoris sequence that
    \[
        H_1(M;\Z)=H_1(S^1;\Z)\oplus \operatorname{coker}(I-A).
    \]
    Since $|\operatorname{tr}A|>2$, there is a canonical isomorphism $H_1(M;\Z)/{\rm Tors}\stackrel{\cong}{\ra}H_1(S^1;\Z)=\Z$ induced by the projection $M\ra S^1$. By Lemma \ref{HomologyControlsFlexibleExponent}, we only need to prove that any non-zero degree self-map of $M$ induces an isomorphism on $H_1(M;\Z)/{\rm Tors}$.
    
    Let $f:M\ra M$ be a non-zero degree self-map of $M$. By \cite[Corollary 0.4]{Wang}, we can assume that $f$ is a fiber-preserving covering map. Then we have the commutative diagram
    \[
        \begin{CD}
    H_1(S^1;\Z) @> \cong >>H_1(M;\Z)/{\rm Tors}\\
@V \bar f_* VV  @V f_* VV @.\\
 H_1(S^1;\Z) @> q_{\# 1}>>H_1(M;\Z)/{\rm Tors}\\
\end{CD}
    \]
    where $\bar f:S^1\ra S^1$ is the induced map on the base space $S^1$. Suppose $\deg \bar f=k$, it suffices to prove that $|k|=1$. 

    Consider the infinite cyclic covering $\widetilde M$ associated to $\ker( \pi_1(M)\ra \pi_1(S^1))$, then $\widetilde M$ is homeomorphic to $T^2\times \R$ and the deck transformation group is generated by $\phi:\widetilde M\ra \widetilde M$,
    \[
        \phi(x,t)=(Ax,t+1),\quad \forall (x,t)\in T^2\times \R.
    \] Let $\tilde f:\widetilde M\ra \widetilde M$ be a lifting of $f:M\ra M$. We claim that $\tilde f\circ \phi =\phi^k\circ \tilde f$. To see this, choose a point $p\in \widetilde{M}$ and an arc $\gamma$ connecting $p$ and $\phi(p)$. Then $\tilde f\circ \gamma$ is an arc connecting $\tilde f(p)$ and $\tilde f\circ \phi (p)$. Since $f$ has degree $k$ on the base circle, we know that $\tilde f\circ \gamma$ projects to a loop on $M$ winding the base circle $k$-times. Hence $\tilde f\circ \phi(p)=\phi^k\circ \tilde f(p)$ and therefore $\tilde f\circ \phi =\phi^k\circ \tilde f$.
    
    Passing to the first integral homology and note that $H_1(\widetilde{M};\Z)\cong H_1(T^2;\Z)$, we have the identity
    \[
        \tilde f_*\cdot A=A^k\cdot \tilde f_*.
    \]
    Since $\deg f$ is nonzero, the induced map $\tilde f_*:H_1(\widetilde{ M};\Z)\ra H_1(\widetilde{ M};\Z)$ is non-singular. This implies that the matrix $A\in {\rm SL}(2,\Z)$ is congruent to the $k$ power of itself. Since $|\operatorname{tr} A|>2$, this can happen only if $|k|=1$. The proof is finished.
    %Consider $M$ as a torus bundle over the circle. Let $T$ be a fiber torus of $M$, then cutting $M$ along $T$ results in a trivial product $T\times I$. It is clear that $f^{-1}(T)$ is a collection of tori in $M$ and cutting along them results in trivial products, so the collection $f^{-1}(T)$ contains a JSJ-torus, i.e. one of the components is isotopic to $T$. Since other components of $f^{-1}(T)$ lie in a trivial product, they are all parallel to the fiber $T$.    If $f^{-1}(T)$ contains more than one components, let $T_0,\ldots,T_k\ (k\geqslant 1)$ be the distinct components of $f^{-1}(T)$. Identify $M\setminus T_0$ with $T\times [0,1]$, in which $T\times\{0\}$ and $T\times \{1\}$ are the two sides of $T_0$, and $T\times \{\frac{i}{k+1}\}$ is identified with $T_i\ (1\leqslant i\leqslant k)$. Define $f_i:T\ra T$ be the restriction map $f|_{T\times \{\frac{i}{k+1}\}}$, then we have $f_{i+1}\simeq A\circ f_i$. In particular, $f_{k+1}\simeq A^{k+1}\circ f_0$. By the gluing condition, we have $f_{k+1}=A\circ f_0$ and hence $A^{k+1}\circ f_0\simeq A\circ f_0$. Passing to the first homology, this implies $A^k=I$ as matrices, but it is impossible since $A$ is not periodic. In conclusion, we have proved that $f^{-1}(T)$ is connected and parallel to $T$, so $f$ induces an isomorphism on $H_1(S^1;\Z)=H_1(M;\Z)/Tor$. Then $P_M(L)\lesssim L^2$ by Proposition \ref{HomologyControlsFlexibleExponent}, so     $\alpha(M_\psi)=2$.
\end{proof}

\subsection{$\alpha(M)=1$ for $\mathbb H^2\times\mathbb E^1$-geometry}

 \begin{prop}\label{prod}
    Let $M$ be a closed orientable $3$-manifold admitting $\mathbb H^2\times\mathbb E^1$-geometry. Then $\alpha(M)=1$.
\end{prop}

\begin{proof}

For the product geometry $\mathbb H^2\times\mathbb E^1$, we have  
    $$\operatorname{Isom}(\mathbb H^2\times\mathbb E^1) = \operatorname{Isom}\mathbb H^2\times\operatorname{Isom} \mathbb E^1$$
    %where $\operatorname{Isom}_+(\mathbb H^2\times\mathbb E^1) $ has two components. 
    (see \cite[Section 12.4]{Mar}). We choose Poincar\'e disk model $D$ for $\mathbb H^2$, 
    $\bar z$ for the complex conjugation of $z\in D$.
    
   We first prove that $\alpha(M)\geqslant 1$.

    Each element $\gamma\in\operatorname{Isom}_+(\mathbb H^2\times\mathbb E^1)$ has the form of either $\gamma(z, t)=(\alpha(z), t+b)$, or
     $\gamma(z, t)=(\alpha(\bar z), -t+b)$, where $z\in \mathbb H^2$ and $t\in \mathbb E^1$, $\alpha\in \operatorname{Isom}_+\mathbb H^2$.

    Suppose $M$ admits $\mathbb H^2\times \mathbb E^1$-geometry. Then $M=\mathbb H^2\times \mathbb E^1/\Gamma$, where $\Gamma\subset  \operatorname{Isom} (\mathbb H^2\times \E^1)$
    is a discrete, torsion free, co-compact subgroup. It is known that $M$ is finitely covered by trivial circle bundle over a hyperbolic surface, hence $\Gamma$ has a finite index subgroup of the form $I_1\times I_2\subset \operatorname{Isom}\mathbb H^2\times\operatorname{Isom} \mathbb E^1.$ In particular, the projection of $\Gamma$ to the $\operatorname{Isom}\mathbb E^1$ factor is a discrete subgroup of $\operatorname{Isom}(\E^1)$, we may assume without loss of generality that elements $\gamma$ of $\Gamma$ are of the form $$\gamma(z,t)=(\alpha(z),t+n) \quad \text{or}\quad \gamma(z,t)=(\alpha(\bar z),-t+n),\quad n\in \Z,\ \alpha\in \operatorname{Isom}(\mathbb H^2).$$ 
    Assume that for some $m>0$ the isometry $t\mapsto t+m$ belongs to $I_2$. Therefore, the isometry $\gamma_0:(z,t)\mapsto (z,t+m)$ belongs to $\Gamma$.
    
    For each integer $k>0$, define $\mu_k: \mathbb H^2\times \mathbb E^1\to \mathbb H^2\times \mathbb E^1$ by $\mu_k(z,  t)=(z, kt)$, for each $(z, t)\in \mathbb H^2\times \mathbb E^1$. Clearly   \[
        \lip \mu_k=k.
    \]
    
    \begin{claim}
        There is a (arithmetic) sequence of integers $\{k\}$ such that $\mu_k$ descends to a self-map of $M=(\mathbb H^2\times \mathbb E^1)/\Gamma$.
    \end{claim}

  \begin{proof}[Proof of the Claim]
  Given $\gamma\in \Gamma$. If $\gamma(z,t)=(\alpha(z),t+n)$ for some $\alpha\in\operatorname{Isom_+(\mathbb H^2)}$, $n\in \Z$, then
  \[\mu_k\circ \gamma(z, t)=\mu_k(\alpha(z), t+n)=(\alpha(z), kt+kn)\]
  and 
   \[ \gamma\circ \mu_k(z, t)=\gamma_i(z, kt)=(\alpha(z), kt+n).\]
   For those integer $k$ such that $k\equiv 1\mod m$, we have
        $\gamma_0^{\frac{(k-1)n}{m}}\circ \gamma \circ \mu_k=\mu_k\circ \gamma.$ This statement also holds true if $\gamma(z,t)=(\alpha(\bar z),-t+n)$ for some $\alpha\in\operatorname{Isom_+(\mathbb H^2)}$, $n\in \Z$. This shows that $\mu_k(\Gamma\cdot x)=\Gamma\cdot \mu_k (x)$ for all $x\in \mathbb H^2\times \mathbb E^1$ and proves the claim. 
  \end{proof}
  
   Then for an arithmetic sequence of integers $\{k\}$, $\mu_k$ descends to a $\bar \mu_k: M\to M$ with 
    $$\lip \bar \mu_k=k, \quad \deg \bar \mu_k=k.$$ So by Lemma \ref{prepare} we have $\alpha(M)\geqslant 1$.

   It remains to prove that $\alpha(M)\leqslant 1$.
  By  Lemma \ref{LiftingUpperBoundsExponent} and Lemma \ref{lift1}, there is a closed orientable surface $\Sigma$ of genus greater than 1, such that $M$ is finitely covered by $\Sigma\times S^1$ and $\alpha(M)\leqslant \alpha(\Sigma\times S^1)$. So Proposition \ref{prod} follows from the following Lemma \ref{H2E1UpperBounds}.
  \end{proof}

\begin{lem}\label{H2E1UpperBounds}
    Let $\Sigma$ be a connected orientable surface of genus greater than $1$ and let $M:=\Sigma\times S^1$ be the product manifold. Then $\alpha(M)\leqslant1$.
\end{lem}
\begin{proof}
    Let $f$ be any non-zero degree self-map of $M$ and let 
    \[
        p_1:M\ra \Sigma,\quad p_2:M\ra S^1
    \]
    be the projections. Since $f$ has non-zero degree, the image of $f_*:\pi_1(M)\ra\pi_1(M)$ have finite index in the target group. %, with the index equal to $\deg f$. 
    In particular, the image of $(p_1\circ f)_*:\pi_1(M)\ra \pi_1(\Sigma)$ is a finite index subgroup of $\pi_1(\Sigma)$. Since $\pi_1(\Sigma)$ is centerless, so $(p_1\circ f)_*$ maps elements of $\pi_1(S^1)$ to the identity, and $(p_1\circ f)_*$ maps $\pi_1(\Sigma)$ to a finite index subgroup of $\pi_1(\Sigma)$. In other words,
    let $\iota_1:\Sigma\ra M$ be the inclusion map sending $\Sigma$ to $\Sigma\times \{0\}$, then
    \[
        \phi:=p_1\circ f\circ \iota_1:\Sigma\longrightarrow \Sigma
    \]
    induces an endomorphism on $\pi_1(\Sigma)$ which has finite index image. Because $\Sigma$ is a hyperbolic surface, so $\phi$ must be a homotopy equivalence and induces an isomorphism on $\pi_1(\Sigma)$. Fix a point $x\in\Sigma$ and let $\iota_2:S^1\ra M$ be the inclusion map sending $S^1$ to $\{x\}\times S^1$. We can write the group homomorphism $f_*$ in the following form:
    \[
        f_*(g, t)=(\phi_*(g),\alpha_*(g)+kt),\quad g \in\pi_1(\Sigma),\ t\in\pi_1(S^1)
    \]
    in which $\alpha=p_2\circ f\circ \iota_1$ and $k$ is the degree of $p_2\circ f\circ \iota_2:S^1\ra S^1$. It is clear that the image of $f_*$ has index $k$ in the target group, so
    \[
        \deg f=\deg(p_2\circ f\circ \iota_2)=k.
    \]
    Equip $M$ with the product metric, then $k$ is no greater than the Lipschitz constant of $p_2\circ f\circ \iota_2$. Therefore
    \[
        k\leqslant \lip(p_2\circ f\circ \iota_2)\leqslant \lip p_2 \cdot \lip f\cdot\lip \iota_2=\lip f.
    \]
    This shows that $\alpha(M)\leqslant1$.
\end{proof}

% The following Lemmas \ref{lift} and \ref{lift1}, combining with Lemma \ref{LiftingUpperBoundsExponent}, tells us that to  provide an upper bound $\alpha$ for the flexible exponent $\alpha(M)$ for all manifolds $M$
% supporting geometry $\mathcal G$, 
%we need only  to  provide an upper bound $\alpha$ for the flexible exponent $\alpha(M)$ for certain class of manifolds $M$
 %supporting a geometry $\mathcal G$. 

%Each torus semi-bundle $N$ is uniquely doubly covered by a torus bundle 
%$$p: M\longrightarrow N$$ which is torus fiber preserving. 

%The following algebraic lemma will be used to prove Lemmas \ref{lift} and \ref{lift1}.

%The following Lemma is \cite[Theorem 2.9]\cite{SWW}.

\bibliographystyle{amsalpha}
\bibliography{ref.bib}

@article {BGM,
    AUTHOR = {Berdnikov, Aleksandr and Guth, Larry and Manin, Fedor},
     TITLE = {Degrees of maps and multiscale geometry},
   JOURNAL = {Forum Math. Pi},
  FJOURNAL = {Forum of Mathematics. Pi},
    VOLUME = {12},
      YEAR = {2024},
     PAGES = {Paper No. e2, 48},
      ISSN = {2050-5086},
   MRCLASS = {53C23 (42B25 55P62)},
  MRNUMBER = {4691574},
       DOI = {10.1017/fmp.2023.33},
       URL = {https://doi.org/10.1017/fmp.2023.33},
}

@article {BM,
    AUTHOR = {Berdnikov, Aleksandr and Manin, Fedor},
     TITLE = {Scalable spaces},
   JOURNAL = {Invent. Math.},
  FJOURNAL = {Inventiones Mathematicae},
    VOLUME = {229},
      YEAR = {2022},
    NUMBER = {3},
     PAGES = {1055--1100},
      ISSN = {0020-9910,1432-1297},
   MRCLASS = {57Q20 (55P62)},
  MRNUMBER = {4462624},
MRREVIEWER = {Leonard\ R.\ Rubin},
       DOI = {10.1007/s00222-022-01118-9},
       URL = {https://doi.org/10.1007/s00222-022-01118-9},
}

@article {Gr,
    AUTHOR = {Gromov, Michael},
     TITLE = {Volume and bounded cohomology},
   JOURNAL = {Inst. Hautes \'Etudes Sci. Publ. Math.},
  FJOURNAL = {Institut des Hautes \'Etudes Scientifiques. Publications
              Math\'ematiques},
    NUMBER = {56},
      YEAR = {1982},
     PAGES = {5--99},
      ISSN = {0073-8301,1618-1913},
   MRCLASS = {53C20 (53C21 57R99 58E99)},
  MRNUMBER = {686042},
MRREVIEWER = {Karsten\ Grove},
       URL = {http://www.numdam.org/item?id=PMIHES_1982__56__5_0},
}

@book {Gr1,
    AUTHOR = {Gromov, Mikhael},
     TITLE = {Structures m\'etriques pour les vari\'et\'es riemanniennes},
    SERIES = {Textes Math\'ematiques [Mathematical Texts]},
    VOLUME = {1},
    EDITOR = {Lafontaine, J. and Pansu, P.},
 PUBLISHER = {CEDIC, Paris},
      YEAR = {1981},
     PAGES = {iv+152},
      ISBN = {2-7124-0714-8},
   MRCLASS = {53C20 (53-02)},
  MRNUMBER = {682063},
}

@misc{Gu,
      title={The {H}opf invariant and simplex straightening}, 
      author={Larry Guth},
      year={2009},
      eprint={0709.1247},
      archivePrefix={arXiv},
      primaryClass={math.DG},
      url={https://arxiv.org/abs/0709.1247}, 
      note = {arXiv preprint arXiv:0709.1247}
}

@misc{Mar,
   author = {Martelli, Bruno},
   title = {An introduction to geometric topology},
   note = {arXiv preprint arXiv:1610.02592},
   year = {2016},
   type = {Journal Article}
}

@article {NWW,
    AUTHOR = {Neofytidis, Christoforos and Wang, Shicheng and Wang, Zhongzi},
     TITLE = {Realising sets of integers as mapping degree sets},
   JOURNAL = {Bull. Lond. Math. Soc.},
  FJOURNAL = {Bulletin of the London Mathematical Society},
    VOLUME = {55},
      YEAR = {2023},
    NUMBER = {4},
     PAGES = {1700--1717},
      ISSN = {0024-6093,1469-2120},
   MRCLASS = {55M25},
  MRNUMBER = {4623679},
MRREVIEWER = {Ikumitsu\ Nagasaki},
       DOI = {10.1112/blms.12813},
       URL = {https://doi.org/10.1112/blms.12813},
}

@article{LSW,
    AUTHOR ={Lin, Jianfeng and Sun, Hongbin and Wang, Zhongzi}, 
    TITLE = {Flexible exponents of non-geometric 3-manifolds},
    note = {arXiv preprint arXiv:2604.23965},
    URL = {https://doi.org/10.48550/arXiv.2604.23965},
    }

@article {Ro,
    AUTHOR = {Rong, Yong Wu},
     TITLE = {Maps between {S}eifert fibered spaces of infinite {$\pi_1$}},
   JOURNAL = {Pacific J. Math.},
  FJOURNAL = {Pacific Journal of Mathematics},
    VOLUME = {160},
      YEAR = {1993},
    NUMBER = {1},
     PAGES = {143--154},
      ISSN = {0030-8730,1945-5844},
   MRCLASS = {55R55 (55M25 57M12 57N10)},
  MRNUMBER = {1227509},
MRREVIEWER = {R.\ J.\ Daverman},
       URL = {http://projecteuclid.org/euclid.pjm/1102624570},
}

@article {Sc,
    AUTHOR = {Scott, Peter},
     TITLE = {The geometries of {$3$}-manifolds},
   JOURNAL = {Bull. London Math. Soc.},
  FJOURNAL = {The Bulletin of the London Mathematical Society},
    VOLUME = {15},
      YEAR = {1983},
    NUMBER = {5},
     PAGES = {401--487},
      ISSN = {0024-6093,1469-2120},
   MRCLASS = {57N10 (22E10 53C20)},
  MRNUMBER = {705527},
MRREVIEWER = {John\ Hempel},
       DOI = {10.1112/blms/15.5.401},
       URL = {https://doi.org/10.1112/blms/15.5.401},
}

@article {SWW,
    AUTHOR = {Sun, Hongbin and Wang, Shicheng and Wu, Jianchun},
     TITLE = {Self-mapping degrees of torus bundles and torus semi-bundles},
   JOURNAL = {Osaka J. Math.},
  FJOURNAL = {Osaka Journal of Mathematics},
    VOLUME = {47},
      YEAR = {2010},
    NUMBER = {1},
     PAGES = {131--155},
      ISSN = {0030-6126},
   MRCLASS = {57M10 (55M25)},
  MRNUMBER = {2666128},
MRREVIEWER = {Wilbur\ Whitten},
       URL = {http://projecteuclid.org/euclid.ojm/1266586789},
}

@article {SWWZ,
    AUTHOR = {Sun, Hongbin and Wang, Shicheng and Wu, Jianchun and Zheng,
              Hao},
     TITLE = {Self-mapping degrees of 3-manifolds},
   JOURNAL = {Osaka J. Math.},
  FJOURNAL = {Osaka Journal of Mathematics},
    VOLUME = {49},
      YEAR = {2012},
    NUMBER = {1},
     PAGES = {247--269},
      ISSN = {0030-6126},
   MRCLASS = {57M50 (55M25 57N10)},
  MRNUMBER = {2903262},
MRREVIEWER = {Bruno\ P.\ Zimmermann},
       URL = {http://projecteuclid.org/euclid.ojm/1332337246},
}

@article {Thom,
    AUTHOR = {Thom, Ren\'e},
     TITLE = {Quelques propri\'et\'es globales des vari\'et\'es
              diff\'erentiables},
   JOURNAL = {Comment. Math. Helv.},
  FJOURNAL = {Commentarii Mathematici Helvetici},
    VOLUME = {28},
      YEAR = {1954},
     PAGES = {17--86},
      ISSN = {0010-2571,1420-8946},
   MRCLASS = {56.0X},
  MRNUMBER = {61823},
MRREVIEWER = {W.\ S.\ Massey},
       DOI = {10.1007/BF02566923},
       URL = {https://doi.org/10.1007/BF02566923},
}

@book {Th,
    AUTHOR = {Thurston, William P.},
     TITLE = {The geometry and topology of three-manifolds. {V}ol. {IV}},
      NOTE = {Edited and with a preface by Steven P. Kerckhoff and a chapter
              by J. W. Milnor},
 PUBLISHER = {American Mathematical Society, Providence, RI},
      YEAR = {[2022] \copyright 2022},
     PAGES = {xvii+316},
      ISBN = {978-1-4704-6391-5; [9781470468361]; [9781470451646]},
   MRCLASS = {57K32 (53C15 57R30)},
  MRNUMBER = {4554426},
MRREVIEWER = {Thilo\ Kuessner},
}

@article {Wang,
    AUTHOR = {Wang, Shi Cheng},
     TITLE = {The {$\pi_1$}-injectivity of self-maps of nonzero degree on
              {$3$}-manifolds},
   JOURNAL = {Math. Ann.},
  FJOURNAL = {Mathematische Annalen},
    VOLUME = {297},
      YEAR = {1993},
    NUMBER = {1},
     PAGES = {171--189},
      ISSN = {0025-5831,1432-1807},
   MRCLASS = {57N10 (57M25)},
  MRNUMBER = {1238414},
MRREVIEWER = {Allan\ Edmonds},
       DOI = {10.1007/BF01459495},
       URL = {https://doi.org/10.1007/BF01459495},
}

@article {Wald,
    AUTHOR = {Waldhausen, Friedhelm},
     TITLE = {On irreducible {$3$}-manifolds which are sufficiently large},
   JOURNAL = {Ann. of Math. (2)},
  FJOURNAL = {Annals of Mathematics. Second Series},
    VOLUME = {87},
      YEAR = {1968},
     PAGES = {56--88},
      ISSN = {0003-486X},
   MRCLASS = {57.05},
  MRNUMBER = {224099},
MRREVIEWER = {W.\ Haken},
       DOI = {10.2307/1970594},
       URL = {https://doi.org/10.2307/1970594},
}

@misc{Ha,
  title={Notes on Basic 3-Manifold Topology},
  author={Allen Hatcher},
  year={2001},
    note={avalible at \url{https://pi.math.cornell.edu/~hatcher/3M/3M.pdf}},
  url={https://api.semanticscholar.org/CorpusID:9792594},}

@article {BG,
    AUTHOR = {Brooks, Robert and Goldman, William},
     TITLE = {The {G}odbillon-{V}ey invariant of a transversely homogeneous
              foliation},
   JOURNAL = {Trans. Amer. Math. Soc.},
  FJOURNAL = {Transactions of the American Mathematical Society},
    VOLUME = {286},
      YEAR = {1984},
    NUMBER = {2},
     PAGES = {651--664},
      ISSN = {0002-9947,1088-6850},
   MRCLASS = {53C12 (55R40 57R32)},
  MRNUMBER = {760978},
MRREVIEWER = {Jean\ Wouafo Kamga},
       DOI = {10.2307/1999813},
       URL = {https://doi.org/10.2307/1999813},
}

@article {DLSW,
    AUTHOR = {Derbez, Pierre and Liu, Yi and Sun, Hongbin and Wang,
              Shicheng},
     TITLE = {Volume of representations and mapping degree},
   JOURNAL = {Adv. Math.},
  FJOURNAL = {Advances in Mathematics},
    VOLUME = {351},
      YEAR = {2019},
     PAGES = {570--613},
      ISSN = {0001-8708,1090-2082},
   MRCLASS = {57M50 (51H20)},
  MRNUMBER = {3954039},
MRREVIEWER = {Thilo\ Kuessner},
       DOI = {10.1016/j.aim.2019.05.015},
       URL = {https://doi.org/10.1016/j.aim.2019.05.015},
}

@article {KL,
    AUTHOR = {Kotschick, D. and L\"oh, C.},
     TITLE = {Fundamental classes not representable by products},
   JOURNAL = {J. Lond. Math. Soc. (2)},
  FJOURNAL = {Journal of the London Mathematical Society. Second Series},
    VOLUME = {79},
      YEAR = {2009},
    NUMBER = {3},
     PAGES = {545--561},
      ISSN = {0024-6107,1469-7750},
   MRCLASS = {53C23 (20F34 20F67 57N65)},
  MRNUMBER = {2506686},
MRREVIEWER = {Peter\ Haskell},
       DOI = {10.1112/jlms/jdn089},
       URL = {https://doi.org/10.1112/jlms/jdn089},
}

@article {LS,
    AUTHOR = {Lafont, Jean-Fran\c cois and Schmidt, Benjamin},
     TITLE = {Simplicial volume of closed locally symmetric spaces of
              non-compact type},
   JOURNAL = {Acta Math.},
  FJOURNAL = {Acta Mathematica},
    VOLUME = {197},
      YEAR = {2006},
    NUMBER = {1},
     PAGES = {129--143},
      ISSN = {0001-5962,1871-2509},
   MRCLASS = {53C20 (57M50 57R99)},
  MRNUMBER = {2285319},
MRREVIEWER = {Vagn\ Lundsgaard\ Hansen},
       DOI = {10.1007/s11511-006-0009-1},
       URL = {https://doi.org/10.1007/s11511-006-0009-1},
}

@book {NS,
    AUTHOR = {Nikulin, V. V. and Shafarevich, I. R.},
     TITLE = {Geometries and groups},
    SERIES = {Universitext},
      NOTE = {Translated from the Russian by M. Reid,
              Springer Series in Soviet Mathematics},
 PUBLISHER = {Springer-Verlag, Berlin},
      YEAR = {1987},
     PAGES = {viii+251},
      ISBN = {3-540-15281-4},
   MRCLASS = {51H20 (11E57 11E72)},
  MRNUMBER = {917939},
       DOI = {10.1007/978-3-642-61570-2},
       URL = {https://doi.org/10.1007/978-3-642-61570-2},
}

\end{document}